\documentclass[11pt,openany]{thesis}
\usepackage{makeidx}
\usepackage{float}
\usepackage{verbatim}
\usepackage{graphicx}
\usepackage{amsmath}
\usepackage{epsfig}
\usepackage{psfrag}
\newtheorem{theorem}{Theorem}[section]

\newtheorem{corollary}[theorem]{Corollary}

\newtheorem{definition}[theorem]{Definition}

\newtheorem{lemma}[theorem]{Lemma}

\newtheorem{problem}[theorem]{Problem}

\newtheorem{remark}[theorem]{Remark}
\newtheorem{remarks}[theorem]{Remarks}

\newenvironment{proof}[1][Proof]{\textbf{#1.} }{}
\newcommand{\qd}{\ \rule{0.5em}{0.5em}}

\newenvironment{namelist}[1]{%
\begin{list}{}
    {
      
      \settowidth{\labelwidth}{#1}
      \setlength{\leftmargin}{1.1\labelwidth}
     }
     }{%
\end{list}}
\restylefloat{figure}
\makeindex
\DeclareMathOperator{\gap}{gap}
\begin{document}
\frontmatter
\thispagestyle{empty}
\begin{center}
{\Large DETOURS IN GRAPHS}\\
\vspace{10mm}

{\large by}\\
\vspace{10mm}

{\Large FRANK ERIC SUTHERLAND BULLOCK}\\
\vspace{10mm}

{\large submitted in accordance with the requirements}\\
\vspace{10mm}

{\large for the degree of}\\
\vspace{10mm}

{\Large DOCTOR OF PHILOSOPHY}\\
\vspace{10mm}

{\large in the subject}\\
\vspace{10mm}

{\Large MATHEMATICS}\\
\vspace{10mm}

{\large at the}\\
\vspace{10mm}

{\Large UNIVERSITY OF SOUTH AFRICA}\\
\vspace{10mm}

{\Large PROMOTER: PROF M FRICK}\\
\vspace{5mm}

{\Large JUNE 2004}
\end{center}
\newpage
\setcounter{page}{1}
\vspace*{2.5cm}
\noindent \textbf{\huge Acknowledgements}\\

\noindent I wish to express my gratitude to  my promoter,
Prof M Frick, for her help and constant encouragement 
during the preparation of this thesis.

\noindent I would also like to thank my wife,
Margaret, for her understanding  and patience during those
times when I was so preoccupied with
``detours".
%promotor pleasure privilege enthusiam encouragement
%I should like to thank
%family wife support during time preparation studying 
\addtocontents{toc}{\protect\hfill\protect\textit{Page}\vskip0pt}
\addcontentsline{toc}{chapter}{\protect\numberline{\ }{Acknowledgements}}
\tableofcontents
\addtocontents{lof}{\protect\hfill\protect\textit{Page}\vskip0pt}
\addcontentsline{toc}{chapter}{\protect\numberline{\ }{List of Figures}}
\listoffigures
\chapter[\ \ \ \ Summary]{Summary}
A \textit{detour} $P$ in a graph $G$ is a longest path in $G$,
and the order of $P$ is called the detour order
of $G$, denoted by $\tau (G)$. 
The \textit{detour order} of a vertex $v$ of 
$G$ is the order of a longest path in $G$ with endvertex $v$.
The \textit{detour sequence} of $G$ is a nondecreasing
sequence consisting of the detour orders of its vertices.
A simple, connected graph is called a 
\textit{detour graph} if its detour sequence
is constant. The \textit{detour deficiency} of a
graph $G$ is $|V(G)| - \tau(G)$.

In Chapter One we prove some basic properties of
detour sequences of simple, connected graphs. 
We show that every detour order
between the least and greatest detour order occurs
in the detour sequence of a connected, unicyclic graph.
We characterise the detour sequence of trees
(using a proof based on detours) and complete
multipartite graphs. We also give a new upper
bound for the $n$\textit{th} detour chromatic 
number of a simple, connected graph $G$ 
in terms of $\tau(G)$.

In Chapter Two we use some basic
properties of connected, nontraceable 
detour graphs (CND graphs) to prove that such a graph 
has order greater than $9$. We 
give a number of constructions 
for CND graphs of
all orders greater than 17 and all
detour deficiencies greater than zero.
These constructions are used to give examples
of CND graphs with chromatic number $k$,
$k \geq 2$, and girths up to 7. Moreover we show that,
for all positive integers $l \ge 1$ and $k \geq 3$, there
are nontraceable detour graphs with chromatic number $k$
and detour deficiency $l$. In the last section
of Chapter Two we present a number of open 
problems.

In Chapter Three we determine the two smallest, claw-free,
$2$-connected, nontraceable graphs. Both these graphs
are in fact CND graphs, and are therefore also the smallest
claw-free, CND graphs. We use one of these graphs
to construct a new family of $2$-connected, claw-free, maximal
nontraceable graphs. 

\noindent \textbf{Key terms}:\\
graph theory; longest path; detour; detour sequence;
girth; homogeneously traceable; claw-free; traceable;
nontraceable; detour chromatic number; hamiltonian path

\mainmatter
\chapter{Detour Sequences}
\section{Definitions and notation}
A \emph{simple graph} $G$ \index{simple graph} 
with $n$ vertices \index{vertices of a graph} 
and $m$ edges \index{edges of a graph}
consists of a vertex set $V(G) = \{v_1,v_2,\ldots,v_n\}$
and an edge set $E(G) = \{e_1,e_2,\ldots,e_m\}$,
where each edge is an unordered pair of distinct vertices.
Since $E(G)$ is a set, in a simple graph no edge
is repeated. A \textit{multigraph}\index{multigraph}
is obtained if we allow repeated edges 
(then $E(G)$ is a multiset) and edges of the
form $\{u,u\}$, called a \textit{loop}\index{loop in a multigraph}.
With the exception of some constructions in Chapter 2 
we consider only simple graphs, 
and for brevity the term \textit{graph}\index{graph} 
will mean simple graph.
Often we denote an edge $\{u,v\}$ by
$uv$ or $vu$. 

We denote the cardinality \index{cardinality of a set}
of any set $S$ by $|S|$. The cardinalities $|V(G)|$ and $|E(G)|$ 
are called the \index{order of a graph} 
\emph{order} and \emph{size} \index{size of a graph} of the graph $G$ respectively.

An edge $e \in E(G)$ is \emph{incident}\index{incident with a vertex} 
with a vertex $v \in V(G)$
if $v \in e$, and the 
\emph{degree}\index{degree of a vertex} of a vertex $v$,
denoted by $\deg(v)$, is the number of edges incident with $v$.  
The \textit{maximum degree}\index{maximum degree}, denoted by $\Delta(G)$,
and the \textit{minimum degree}, 
denoted by $\delta(G)$\index{minimum degree}, of a graph $G$ 
are, respectively, the maximum and minimum degrees of the
vertices of $G$.
If $uv \in E(G)$ 
%%%%%%%%%%%%%%%%%%%%%is an edge of $G$ 
then we say that the vertices
$u$ and $v$ are adjacent\index{adjacent vertices}, or that $u$
is a \textit{neighbour}\index{neighbour of a vertex}
of $v$.

A \emph{trail}\index{trail} is a sequence
$v_0,e_1,v_1,e_2,\ldots,e_k,v_k$ of vertices and edges
such that $e_i = v_{i-1}v_i$ for all $i$ and no edge
is repeated. A \emph{path}\index{path} is a trail with no repeated
vertex. We usually denote a path $P$ by listing the vertices
in the path as follows,\   $P: v_0,v_1,v_2,\ldots,v_k$.
The vertices $v_0$ and $v_k$ are called
\emph{endvertices}\index{endvertices of a path} of $P$,
and the vertices of $P$ which are not endvertices
are called \textit{internal}\index{internal vertices of a path}
vertices of $P$. The size of a path $P$ is also
called the 
\textit{length}\index{length of a path} of $P$.

If $u$, $v$ are endvertices of a path $P$
we say that $u$ is 
\textit{joined}\index{joined by a path}
to $v$ by $P$, or $u$ and $v$ are joined
by $P$. A graph $G$ 
is \textit{connected}\index{connected graph}
if any two vertices $u$, $v$ in $V(G)$ are joined
by a path in $G$. 

If $Q: w_0,w_1,\ldots,w_m$ 
and $P:v_0,v_1,\ldots,v_k$ are vertex disjoint paths
with $w_i$ adjacent to $v_0$ and $w_j$
adjacent to $v_k$ ,where $j > i$, 
then $w_0,w_1,\ldots,w_i,P,w_j,\ldots,w_m$ denotes
the path\index{path!notation for a path}
$w_0,w_1,\ldots,w_i,v_0,v_1,\ldots,v_k,w_j,\ldots,w_m$.

If $u,v$ are two vertices of $G$ then
$\tau_{G}(u,v)$ denotes the order of 
a longest path\index{longest path notation}
in $G$ with
endvertices $u$ and $v$. 
The \textit{detour order}\index{detour order of a vertex}
of a vertex
$v\in V(G)$ is the order of a longest path
$P$ in $G$ having $v$ as an initial vertex. The detour order
of $v$ is denoted by $\tau_{G}(v)$. 
If there is no danger of
confusion then we simply write $\tau(v)$ and $\tau(u,v)$ 
instead of $\tau_{G}(v)$ and $\tau_{G}(u, v)$ respectively.
See, for example, Figure~\ref{ex-det-seq}.
\begin{figure}[H]
\begin{center}
\psfrag{a}[c][c]{$v_3$}
\psfrag{b}[c][c]{$v_1$}
\psfrag{c}[c][c]{$v_5$}
\psfrag{d}[c][c]{$v_4$}
\psfrag{e}[c][c]{$v_2$}
\psfrag{f}[c][c]{$v_6$}
\psfrag{g}[c][c]{$\tau(v_3)=6$}
\psfrag{h}[c][c]{$\tau(v_1)=5$}
\psfrag{i}[c][c]{$\tau(v_5) =6$}
\psfrag{j}[c][c]{$\tau(v_4)=6$}
\psfrag{k}[c][c]{$\tau(v_2) =5$}
\psfrag{l}[c][c]{$\tau(v_6) =6$}
\epsfig{file=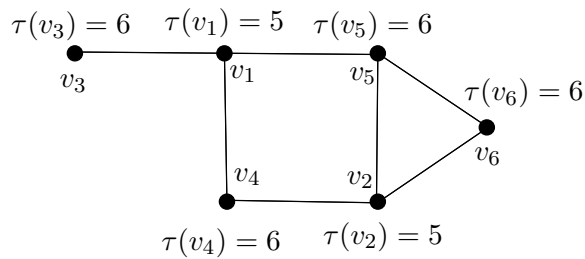}
\end{center}
\caption{\label{ex-det-seq} Detour orders of vertices}
\end{figure}
The detour order of a graph $G$,\index{detour order of a graph} 
which we will denote
by $\tau(G)$, is the maximum of the detour orders of
its vertices. In other words, $\tau(G)$ is the order
of a longest path in $G$. A longest path in a graph
$G$ is called a \textit{detour}\index{detour in a graph}. 

The difference $|V(G)|-\tau(G)$ is called the
\textit{detour deficiency}\index{detour deficiency of a graph} of $G$.
If we label the vertices of $G$ by $v_1,v_2,\ldots,v_n$ 
such that 
$\tau_{G}(v_1) \le \tau_{G}(v_2) \le \tau_G(v_3) \le \ldots \le  \tau_{G}(v_n)$
then the nondecreasing sequence
$\tau_{G}(v_1),\tau_{G}(v_2), \dots, \tau_{G}(v_n)$ is called
the \textit{detour sequence}\index{detour sequence of a graph} of $G$.
A nondecreasing sequence $D$ is called
a \textit{detour sequence} if there exists a graph
whose detour sequence is $D$.
For example, the detour sequence of the graph shown in
Figure~\ref{ex-det-seq} is $5, 5, 6, 6, 6, 6$ since
$\tau(v_1) = \tau(v_2) = 5$
and $\tau(v_3) = \tau(v_4) = \tau(v_5) =\tau(v_6) = 6$.

A \textit{cycle}\index{cycle} of order $n$, denoted by $C_n$,
consists of $n$ vertices $v_1,v_2,\ldots,v_n$ such that $v_i$ is
adjacent to $v_{i+1}$ for $i=1,2,\ldots,n-1$ and
$v_n$ is adjacent to $v_1$.  
The \emph{girth}\index{girth of a graph} $g(G)$ 
and the \emph{circumference}\index{circumference of a graph}
$c(G)$ are, respectively, the order of a 
shortest and a longest cycle in $G$.

Two graphs $G$ and $H$ are 
\textit{isomorphic}\index{isomorphic graphs} if there exists
a bijection $\psi : V(G) \rightarrow V(H)$ 
such that $uv \in E(G)$ if and only if $\psi(u)\psi(v) \in E(H)$.
A graph which is isomorphic to a graph $G$ is called
a \textit{copy}\index{copy of a graph} of $G$.
A graph $H$ is a \textit{subgraph}\index{subgraph} of a graph
$G$ if $V(H) \subseteq V(G)$ and $E(H) \subseteq E(G)$.
A graph that is a copy of a subgraph of $G$ will also
be called a subgraph of $G$. If $H$ is a
subgraph of $G$ we write $H \subseteq G$.

A subgraph of $G$ having
the same order as $G$ is called a 
\textit{spanning subgraph}\index{spanning subgraph}
of $G$. If $U \subseteq V(G)$ then the subgraph of $G$
\textit{induced}\index{induced subgraph} by $U$ has
vertex set $U$ and 
edge set $\{uv | u \in U, v\in U, uv \in E(G)\}$.
We denote the subgraph of $G$ induced by $U \subseteq V(G)$
by $G\langle U \rangle$. If no confusion can
result we simply write $\langle U \rangle$. 
If $H \subseteq V(G)$ then
we write $G - H$ for the graph $G\langle V(G)\setminus H \rangle$,  
and if $H =\{w\}$ we write $G -w$ instead of
$G - \{w\}$. If $u \in V(G)$ and $v \in V(G)$,
but $uv \notin E(G)$, then $G+ uv$ is the
graph with vertex set $V(G)$ and edge
set $E(G) \cup \{uv\}$.  Similarly, if $uv \in E(G)$,
then $G - uv$ is the graph with
vertex set $V(G)$ and edge set $E(G) \setminus \{uv\}$.

A graph $G$ is \textit{complete}\index{complete graph}
if $uv \in E(G)$ for all distinct vertices $u$, $v$
in $V(G)$. The complete graph of order $n$ is
denoted by $K_n$.

The \textit{complement}\index{complement of a graph}
of a graph $G$, denoted by $\overline{G}$,
is the graph with vertex set $V(\overline{G}) = V(G)$
and $uv \in E(\overline{G})$ if and only if
$u \in V(G)$, $v \in V(G)$ and $uv \notin E(G)$.
For example, $|V(\overline{K_n})| =n$ and
$E(\overline{K_n}) = \emptyset$.

Other definitions will be given where they are needed.
For any concept not defined here
we use the definition given in 
Chartrand and Lesniak~\cite{cl}.

In Section~\ref{background} we mention some early 
papers on detours.  In Section~\ref{prop-det-seq}
we derive some properties of detour sequences, and
in Section~\ref{graph-type} we determine the detour
sequences of some graphs. We conclude this chapter
by deriving an upper bound for the 
n\textit{th} detour
chromatic number of a graph $G$ in terms of $\tau(G)$. 

\section{Background}\label{background}
The concept of 
the detour order of a vertex
appeared probably for the first time (with a different notation) 
in Problem~46 (formulated by L. Lov\'{a}sz) of
\cite{erka68} on page 366. Also, a longest path 
in a graph
was called a detour path by
Kapoor, Kronk and Lick~\cite{kapoor}, and the 
length of such a path
the detour number\index{detour number} of the graph.
They showed, amongst other results,
that if $G$ is a connected graph of order
$n$ then $\tau(G) \ge 1+\min\{n-1, 2\delta(G)\}$.
%They determined lower bounds for the detour
%number, amongst other results.
Banerjee and Harary~\cite{banerjee} determined
the minimum detour number among all
connected graphs of given order $n$ and size $m$.

In \cite{bomi87} Borowiecki and Mih\'{o}k
posed the problem of characterising
detour sequences, that is,
to find necessary and sufficient conditions for a 
given sequence of positive integers to be 
the detour sequence of some graph. They referred
to a detour sequence as a longest path 
degree sequence\index{longest path degree
sequence}.  The degree of a vertex 
may be regarded as the  number of paths
of order 2 starting at that vertex. Instead of number of paths,
we can consider the order of the 
longest path starting at the vertex.
Thus characterising longest path degree
sequences is analogous to the problem of characterising 
degree sequences\index{degree sequence of a graph} of graphs,
studied, for example, in \cite{erga60,ha62,ha55}.

The detour sequences of trees\index{detour sequence of trees}
have been characterised
by Lesniak~\cite{lesniak}, and Dobrynin and 
Me{l\kern-.15em'\kern-.09em}nikov~\cite{dome02} have
characterised the detour sequences of
several families of cubic graphs,\index{detour sequence of cubic graphs}
and have listed the detour sequences of
all connected cubic graphs with at most 
20 vertices.

The invariant $\tau(G)$ is hereditary, and it can be used
to define an interesting additive hereditary property 
\index{hereditary property of graphs}
of graphs (for details see \cite{bobr97}). 
In \cite{bufr02} the relationship between  $\tau(G)$ and
the detour chromatic number\index{detour chromatic number}
of $G$ is described. We discuss this relationship
in Section~\ref{det-chrom-num} of this chapter.
It is also shown in \cite{bufr02}
that the detour
chromatic number is related to the Path Partition
Conjecture. \index{Path partition conjecture}
This conjecture
states that, for any graph $G$,
if $\tau(G) = a+b$ then there exists a partition
$V_1$, $V_2$ of $V(G)$ such that
$\tau(G\langle V_1 \rangle) \le a$ and
$\tau(G\langle V_2 \rangle) \le b$.
This conjecture is discussed in \cite{fr02}
and \cite{dfb}.

\section{Properties of detour sequences}\label{prop-det-seq}
In this section we derive some
properties of detour sequences. 
We will investigate repetitions
in detour sequences, and whether or not all
detour orders between the least and greatest occur
in a detour sequence. We make these ideas precise
in the next two definitions.
\begin{definition}
Let $D$ be the detour sequence 
$d_1, d_2, d_3,\ldots,d_n$ of a graph $G$. 
A repetition\index{repetition in a detour sequence} in $D$ is a subsequence
of $D$ of the form 
$d_i, d_{i+1}, d_{i+2},\ldots,d_{i+m}$
such that
\begin{enumerate}
\item $m \geq 1$, $1 < i+m \leq n$ and
$d_i =d_{i+1} = d_{i+2}= \ldots =d_{i+m}$. 
\item If $1 \le j < i$ then $d_j < d_i$
and if $i+m < k \leq n$ then $d_k > d_i$.
\end{enumerate}
The length of the repetition\index{length of repetition}
is $m+1$.
\end{definition}
For example, the detour sequence of the graph
shown in Figure~\ref{ex-det-seq} has
two repetitions, namely
$5, 5$ of length $2$
and $6, 6, 6, 6$ of length $4$.
We denote a repetition of length $k$
which has the form
$n,n,n,\ldots,n,n$ ($n$ repeated $k$ times) 
by $(n)_k$. It will sometimes be
convenient to denote a single term
$n$ in a detour sequence also by $(n)_1$.
Thus, for example, 
the above detour sequence $5, 5, 6, 6, 6, 6$
can be written $(5)_2, (6)_4$.  Another simple
but useful example is the detour
sequence of a path of order $n$. This is \linebreak
$(a)_k,(a+1)_2,(a+2)_2,\ldots,(n-1)_2,(n)_2$
where $a =\lceil (n+1)/2 \rceil$ and $k=1$ for odd
$n$ and $k =2$ for even $n$.
\begin{definition}
Let $D:\ d_1, d_2, d_3,\ldots,d_{n-1},d_n$ be
the detour sequence of a graph $G$. 
If, for each positive integer $k$ with $d_1 \leq k \leq d_n$,
there exists a term $d_j$ in $D$ 
such that $d_j = k$, then
we say that the detour sequence of $G$ is \textit{full}.
\index{full detour sequence}
Equivalently, $D$ is full if and only if
$0 \leq d_{i} - d_{i-1} \leq 1$ for all $i$ 
such that $2 \leq i \leq n$. 
\end{definition}
For detour sequences which are not full,
the following terminology will be useful.
\begin{definition}
Let $D:\ d_1, d_2, d_3,\ldots,d_n$ be the 
detour sequence of a graph $G$.  Then
$\max_{2 \leq i \leq n}\{d_i -d_{i-1}\}$
is called the \textit{maximum gap}\index{gap of a detour sequence}
of $D$.
\end{definition}
The maximum gap of a detour sequence $D$ is denoted
by $\gap (D)$. 
Clearly a detour sequence is full if
and only if its maximum gap is equal to one or zero.
A detour sequence whose maximum gap is zero is
called a \textit{constant}\index{constant detour sequence}
detour sequence. Examples of graphs with constant
detour sequences are hamiltonian graphs. (A 
\textit{hamiltonian graph} is a graph which
has a spanning cycle.)
These graphs, and all previously studied connected graphs
$G$ with a constant detour sequence, satisfy
$\tau(G) = |V(G)|$.
In Chapter 2 we construct connected graphs $G$ 
with $\tau(G) < |V(G)|$ which have 
a constant detour sequence.

The corollary to Lemma~\ref{distinct}
gives an upper bound for the maximum gap
of the detour sequence of a connected graph.
\begin{lemma}\label{distinct}
Let $G$ be  connected graph with $v \in V(G)$. 
Then 
\begin{enumerate}
\item $\left\lceil \frac{\tau(G)+1}{2} \right\rceil 
\le \tau(v) \le \tau(G)$. 
\item There are at most 
$\left\lceil \frac{\tau(G)}{2} \right\rceil$ distinct terms
in the detour sequence of $G$. 
\end{enumerate}
\end{lemma}
\begin{proof}
Let $k =\tau(G)$. 
Let $L$ be a path of order $k$ in $G$.
Let $v \in V(G)$. If $v \in L$ then
clearly the detour order of $v$ is at least 
$\left\lceil (k+1)/2 \right\rceil$. If $v \notin L$ then,
since $G$ is connected, there is a path 
$P$ joining $v$ to some vertex $u \in L$
such that $V(P) \cap V(L) = \{u \}$.
Since the detour order of $u$ is at least
$\left\lceil (k+1)/2 \right\rceil$ it follows that the
same result holds for $v$.  Thus the detour
order of any vertex of $G$ lies between
$\left\lceil (k+1)/2 \right\rceil$ and $k$, which
proves part (1). Part (2) follows immediately,
since 
\begin{align}
k - \left\lceil\frac{k+1}{2} \right\rceil +1
= \left\lceil\frac{k}{2} \right\rceil.  \tag*{\qd}
\end{align}
\end{proof}
\begin{corollary}\label{bound-gap}
Let $G$ be a connected graph with
detour sequence $D$. Then
$\gap (D) \le \left\lfloor \frac{\tau(G)}{2} \right\rfloor$.
\end{corollary}
\begin{proof} 
By Lemma~\ref{distinct}, if $v \in V(G)$ then
\[
\left\lceil \frac{\tau(G)}{2} \right\rceil 
\le \tau(v) \le \tau(G). 
\]
Hence
\begin{align}
\gap (D) \le  \tau(G) 
- \left\lceil\frac{\tau(G)}{2} \right\rceil
= \left\lfloor\frac{\tau(G)}{2} \right\rfloor. \tag*{\qd}
\end{align}
\end{proof}
\begin{remark}\label{at-most-two}
If $G$ is a connected graph and
$\tau(G) = k$ is even 
there are at most two vertices
in $G$ with detour order $(k+2)/2$, and if $k$ is odd
there is at most one vertex in $G$ with detour order
$(k+1)/2$. Also the greatest term in 
any  detour sequence (with more than one term) 
belongs to a repetition of length at least two,
since a detour (of order greater than one)
has two endvertices.
\end{remark}

We can construct graphs which
realise the upper bound for the maximum gap 
given in Corollary~\ref{bound-gap}.
Also, for any non-negative integer $k$,
there exist graphs whose detour sequence
has maximum gap equal to $k$.
For example, let $G_{n,m}$ be a graph such that:
\begin{enumerate}
\item The vertex set $V(G_{n,m}) = A \cup B$, where
$|A| = n$, $|B| = m$ 
$(1 <  m \leq n)$ and $A \cap B = \{w\}$.
\item $G_{n,m}\langle A \rangle = K_n$
and $G_{n,m}\langle B \rangle = K_m$.
\item If $u \in A\setminus\{w\}$ and 
$v \in B\setminus \{w\}$ then $uv \notin E(G_{n,m})$.
\end{enumerate}
The graph $G_{n,m}$ is depicted in
Figure~\ref{pict}.
\begin{figure}[H]
\begin{center}
\psfrag{a}[c][c]{$K_n$}
\psfrag{b}[c][c]{$K_m$}
\psfrag{c}[c][c]{$w$}
\epsfig{file=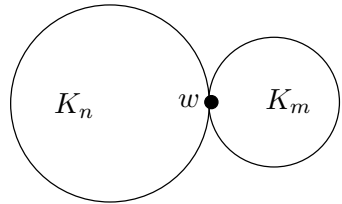}
\end{center}
\caption{\label{pict} The graph $G_{n,m}$}
\end{figure}
Clearly, if $v \in V(G_{n,m}) \setminus \{w\}$ we have
$\tau_{G_{n,m}}(v) = \tau(G_{n,m}) = n+m -1$ and
$\tau_{G_{n,m}}(w) = n$. Therefore, the
detour sequence $D_{n,m}$ 
of $G_{n,m}$ is $n,(n+m-1)_{(n+m-2)}$. 
Hence $\gap (D_{n,m}) = m-1$.
If $n = m$ then $\tau(G_{n,n}) =2n-1$
and we get
\[
\gap (D_{n,n}) = n-1 = 
\left\lfloor \frac {\tau(G_{n,n})}{2} \right \rfloor.
\]

We have the following lower bound for the 
length of a 
repetition\index{length of a repetition in a detour\\ sequence} 
in the detour sequence of a graph:
 
\begin{theorem}\label{rep-length}
If $G$ is connected and $|V(G)| = n > 1$ 
then the detour sequence of $G$
has a repetition of length at least 
$\left\lceil 2n/\tau(G) \right\rceil$.
\end{theorem}\pagebreak
%Uses: If $n > k$ objects are placed in $k$
% boxes, then some box has at least
% $\lceil n/k \rceil$ objects in it.
\begin{proof}
Let $\tau(G) = k$.
\begin{enumerate}
\item[(i)] If $k$ is even, by Lemma~\ref{distinct} there
are at most $k/2$ distinct elements in the
detour sequence of $G$.
%(Thinking of distinct elements as boxes
%and vertices as objects). So here we seem
%to be discussing the length of a repetition.
Hence some detour order is repeated at
least 
$\left\lceil \frac{n}{k/2} \right\rceil 
= \left\lceil \frac{2n}{k} \right\rceil$
times in the detour sequence of $G$.

\item[(ii)] If $k$ is odd, by Lemma~\ref{distinct}
there are at most
$\frac{k+1}{2}$ distinct elements in the detour sequence
of $G$. However, the detour order $\frac{k+1}{2}$
can occur at most once in the detour sequence of $G$.
If it does not occur, then we have $n$ vertices and
at most $\frac{k+1}{2} -1$ distinct elements in the detour
sequence. So in this case some detour order
must be repeated at least $\left\lceil \frac{2n}{k-1} \right\rceil$
times. Otherwise, by omitting the vertex with 
detour order $\frac{k+1}{2}$, we get $n-1$ vertices with
at most $\frac{k-1}{2}$ distinct detour orders. Hence some detour
order must be repeated at least 
$\left\lceil \frac{2(n-1)}{k-1}\right\rceil$
times. Since 
\[
\left\lceil \frac{2n}{k-1}\right\rceil \ge
\left\lceil \frac{2(n-1)}{k-1}\right\rceil \ge     
\left\lceil \frac{2n}{k}\right\rceil
\]
the result follows. \hspace*{\fill}\qd
\end{enumerate}
\end{proof}

%arbitarily large
%repetitions relative to order of graph - this will 
%also motivate constant detour graphs. Detour sequences
%of various classes of graphs.
%Unicyclic graphs, trees, complete multipartite graphs.
%Which sequences cannot be detour sequences? If they
%do not have a vertex repeated at least three times must be a path. 
%Gives a simple necessary condition. So for example
%3, 2, 2 is not a detour sequence.

The length of a repetition in $G$ is
obviously bounded above by $|V(G)|$.
Of course, 
all hamiltonian 
%%%%%%%%%%%%%%%%%%%%%%%graphsfor example, cycles and complete
graphs have repetitions of length $|V(G)|$.
In Chapter 2 we give non-trivial examples 
where repetitions of
length $|V(G)|$ occur. 

%Since obviously for any graph $G$ 
%we have $\tau(G) \leq |V(G)|$
It follows from Theorem~\ref{rep-length}
(and from Remark~\ref{at-most-two}) 
that the detour sequence of any
connected graph, except $K_1$, has a repetition of 
length at least two.
In Theorem~\ref{rep3} we  show that paths are 
the only connected graphs
whose detour sequences do not have 
a repetition with length greater
than two. First we prove a lemma.

\begin{lemma}\label{repetition}
Let $P: v_1,v_2,v_3,\ldots,v_n$, where $n \geq 3$,
be a path.
If $e \in E(\overline{P})$ then the detour
sequence of $P+e$ has a repetition of length
at least three.
\end{lemma}
\begin{proof}
Let $k = \left\lceil \frac{n}{2} \right \rceil$.
Then $v_k$ is a vertex with the least detour
order in $P$. In $P$, 
apart possibly from the
detour order of $v_k$, every detour order is repeated
exactly twice, and every detour order 
between $\tau(v_k)$ and $|V(P)|$ occurs 
in the detour sequence of $P$.
Since adding an edge $e$ does not result in a decreased
detour order for any vertex of $P$, it is sufficient
to show that some vertex of $P$ has a larger detour order
in $P+e$ than it has in $P$.  
Let $e = v_iv_j$ where $i+1 < j \leq n$
and $1 \leq i \leq n-2$. Let $G =P +e$.
We consider two cases, depending on whether or not
$v_k$ lies between $v_i$ and $v_j$. 
\begin{enumerate}
\item Assume without loss of generality
that  $k \leq i$. Then
$\tau_P(v_{i+1}) = i+1$. (Note that,
if $k = i$ and $n$ is even, then
$\tau_P(v_{i+1}) = \tau_P(v_i) = i+1$.)
In $G$ we have the
path $v_{i+1},v_{i+2},\dots,v_j,v_i,v_{i-1},\ldots,v_1$
of order $j > i+1$. 
Hence $\tau_G(v_{i+1}) \ge j > i+1 = \tau_P(v_{i+1})$.

%Similarly, if  $k \geq j$ we get 
%$\tau_G(v_{j-1})) > \tau_P(v_{j-1})$. 
%Then
%$\tau_P(v_{j-1}) = n-j+2$. In $G$ we have the path
%$v_{j-1}, v_{j-2},\ldots,v_i,v_j,v_{j+1},\ldots,v_n$
%of order $n-i+1$. 
%Hence $\tau_G(v_{j-1})) \ge n-i+1 > n-j+2 = \tau_P(v_{j-1})$.

\item  Suppose that $i < k < j$. 
Then 
$\tau_P(v_k) = n-k+1$. Assume without
loss of generality that $k-i \ge j-k$.
The path
$v_k, v_{k-1},\ldots,v_i,v_j,v_{j+1},\ldots,v_n$
in $G$ has order $(k-i) + (n-j) +2$, while
$\tau_P(v_k) = (j-k) + (n-j) +1$.
Hence
\begin{align}
\tau_G(v_k)\geq (k-i)+(n-j) +2 > (j-k)+(n-j)+1
=\tau_P(v_k). \tag*{\qd}
\end{align}
\begin{comment}
It is convenient
to consider the cases $n$ odd and $n$ even
separately:
\begin{enumerate} 
\item Suppose that $n$ is odd. Then 
$\tau_P(v_k) = k = n-k+1$.
Suppose that $k-i \ge j-k$. The path
$v_k, v_{k-1},\ldots,v_i,v_j,v_{j+1},\ldots,v_n$
in $G$ has order $(k-i) + (n-j) +2$, while
$\tau_P(v_k) = (j-k) + (n-j) +1$.
Hence
\[
\tau_G(v_k)\geq (k-i)+(n-j) +2 > (j-k)+(n-j)+1
=\tau_P(v_k).
\]
Suppose that $j-k > k-i$. The path
$v_k,v_{k+1},\ldots,v_j,v_i,v_{i-1},\ldots,v_1$
in $G$ has order $j-k+1+i$. Hence
\[
\tau_G(v_k) \ge j-k+1+i > k+1 > \tau_P(v_k).
\]
\item Suppose that $n$ is even. Then
$\tau_P(v_k) = \tau_P(v_{k+1})= k+1 =n-k+1$. 
Suppose that $k-i \ge j-k$. The path
$v_k,v_{k-1},\ldots,v_i,v_j,v_{j+1},\ldots,v_n$
in $G$ has order $(k-i)+(n-j)+2$.  Hence
\[
\tau_G(v_k) \geq (k-i)+(n-j)+2 > (j-k)+(n-j)+1
= \tau_P(v_k).
\]
Suppose that $k-i < j-k$.  The path
$v_k,v_{k+1},\ldots,v_j,v_i,v_{i-1},\ldots,v_1$
in $G$ has order $j-k+1+i$. Hence
\begin{align}
\tau_G(v_k) \ge j-k+1+i > k+1 =\tau_P(v_k). \tag*{\qd}
\end{align}
\end{comment}
\end{enumerate}
\end{proof}

\begin{remark}\label{rem-rep}
It follows from  Lemma~\ref{repetition}
that all graphs obtained from a path $P$ by successively adding
edges will have a repetition of length at least three.
This follows since the detour sequence
of a path is full, and adding an edge never decreases the detour
order of a vertex, and we start with the 
detour order of every vertex of $P$ 
(except possibly for the vertex with least
detour order)
being repeated twice.
\end{remark}  
\index{repetition length of a path detour\\ sequence}
\begin{theorem}\label{rep3}
If $G$ is a connected graph that is not a path, 
then the detour sequence
of $G$ has a repetition of length at least three.
\end{theorem}
\begin{proof}
If $\tau(G) < |V(G)|$ then 
$\left\lceil \dfrac{2|V(G)|}{\tau (G)} \right\rceil \geq 3$. Hence
by Theorem~\ref{rep-length}
it follows that $G$ has a repetition of length at least three.

Suppose that $\tau(G) = |V(G)|$, and let 
$H$ be a detour in $G$.
Since $G$ can be obtained from $H$ by
succesively adding edges to $H$, it
follows from Remark~\ref{rem-rep}
that $G$ has a
repetition of length at least three. \hspace*{\fill}\qd
\end{proof}

It follows from Theorem~\ref{rep3}
that paths are uniquely determined by
their detour sequences. It is easy to
give examples of non-isomorpic graphs,
other than paths, which have 
identical detour sequences.
The tree and unicyclic graph (defined
in Section~\ref{sec-tree}) shown
in Figure~\ref{same-det} each have
detour sequence $3,(4)_2, (5)_3$.
\begin{figure}[H]
\begin{center}
\psfrag{a}[c][c]{5}
\psfrag{b}[c][c]{4}
\psfrag{c}[c][c]{3}
\epsfig{file=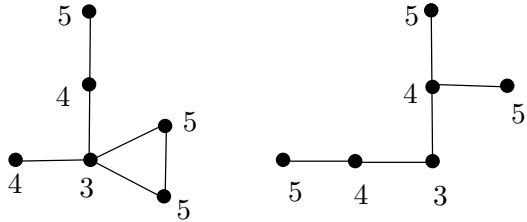}
\end{center}
\caption{\label{same-det} Graphs with identical detour sequences}
\end{figure}
The vertices in Figure~\ref{same-det}
are labelled with their detour orders.

Trees have the property that,
given any collection of
positive integers $k_1, k_2,\ldots,k_n$,
we can find a tree whose detour sequence
has repetitions of lengths 
$k_1, k_2,\ldots,k_n$.
This follows from
the chacterisation of their detour sequences
which we give in the next section.

\section{Detour sequences of types of graphs}\label{graph-type}
\label{sec-tree}
A \textit{tree}\index{tree}
is a connected graph with no cycles as 
subgraphs. We characterise the detour sequences
of trees\index{detour sequence of trees}
in Theorem~\ref{seq-tree}.
Our characterisation is the same as the characterisation
of eccentric sequences for trees 
obtained by Lesniak~\cite{lesniak}, but
our formulation and proof differs from that in \cite{lesniak}.
We use a proof based directly on detours.

\begin{theorem}\label{seq-tree}
A non-decreasing sequence is the detour sequence of a tree if and
only if it is $1$ or $2,2$ or has the form
\[
(a)_k,\  (a+1)_{k_1},\  (a+2)_{k_2}, \ldots, (m)_{k_l}
\]
where m is an integer with $m \geq 3$, 
$a= \left\lceil \dfrac{m+1}{2} \right\rceil$,
$l = m-a$, $k_j \ge 2$ for $1 \le j \le l$,
and $k=1$ if
$m$ is odd and $k=2$ if $m$ is even.
\end{theorem}
\begin{proof}
Let $T$ be a tree and suppose that
$\tau(T) = m$. If $m=1$ or $m =2$ we have detour
sequences $1$ and $2,2$ respectively. 
Suppose that $m \geq 3$. Let $L$ be a detour in $T$.
We can obtain $T$ by attaching appropiate trees to
the vertices of $L$, the attached trees being
pairwise disjoint.
%and having only a single vertex
%in common with $L$. 
Since $L$ is a detour, we can only attach trees to
the interior vertices of $L$. Also no vertex of 
the attached trees can have detour order greater
than $m$ or less than 
$\left\lceil \frac{m+1}{2} \right\rceil +1$, 
and if $v \in V(L)$ then 
$\tau_L(v) = \tau_T(v)$ since $L$ is
a detour in $T$. Therefore to get the detour
sequence of $T$, we can take the detour sequence
of $L$ and simply increase the length of some of the
repetitions of the detour sequence appropriately,
but leave the length of the repetition of the least
detour order in the detour sequence of $L$ unaltered.
This gives a detour sequence of the form stated
in the theorem.   

Now suppose we have a sequence of the form stated 
in the theorem. The trees $K_1$ and $K_2$
have detour sequence $1$ and $2,2$ respectively.
Suppose that $m \geq 3$. We construct a tree as follows:
Let $L:\ v_1,v_2,v_3,\ldots,v_{m}$ be a path of order
$m$. Attach $k_j-2$ disjoint paths of order $1$ to vertex
$v_{a+j-1}$ ($1 \leq j \leq l$) of $L$. The resulting graph 
is a tree with detour sequence given in the theorem. \hspace*{\fill}\qd 
\end{proof}

Next we investigate the detour sequences of simple,
connected graphs which possess exactly one cycle.
Such graphs are called 
\textit{unicyclic} graphs.\index{unicyclic graphs}
The structure of unicyclic graphs is described
in Lemma~\ref{unicyclic-Harary},
due to Anderson and Harary~\cite{harary}:
\begin{lemma}\label{unicyclic-Harary}
The following two statements are equivalent:
\begin{enumerate}
\item[(1)] $G$ is a unicyclic graph.
\item[(2)] $G$ consists of a cycle $C$ with disjoint trees
$T_i$ attached to some of the vertices $v_i \in V(C)$,
such that $|V(C) \cap V(T_i)| =1$.
\end{enumerate}
\end{lemma}
\begin{proof}
\begin{namelist}{XXXXXXXX}
\item[$(1) \implies (2)$:] Let $C$ be the cycle of $G$. 
Delete the  edges of $C$ and for each $v \in V(C)$ 
let $G_v$ be the component in which $v$ lies. Since 
$G$ has only one cycle $C$, and $V(G_v) \cap V(C) = \{v\}$,
it follows that $G_v$ is acyclic. Since $G_v$ is
connected $G_v$ is a tree.
\item[$(2) \implies (1)$:] Each of the attached trees is acyclic
so clearly $G$ has only the one cycle $C$. \hspace*{\fill}\qd
\end{namelist}
\end{proof}
Other statements which are equivalent to $G$ being
unicyclic are given in Anderson and Harary~\cite{harary}.

First we confine our attention to graphs $G$ which are obtained
by attaching disjoint paths to some (or all) of 
the vertices of a cycle, where 
these paths have only an endvertex in common with the cycle. 
We require that at least one vertex of $C$ has a path
attached, in order to exclude the trivial case of cycles.
We will call such graphs \textit{path-unicyclic} graphs.
\index{path-unicyclic graphs}
Let $C$ be the cycle of a path-unicyclic graph $G$.
Suppose that $w \in V(C)$ and there is a path $P$ attached
to $w$ of order $p$. Then if $w$ has detour order
$p+1$ we say that $w$ is an \textit{extreme} vertex 
\index{extreme vertex in a path-unicyclic\\ graph} of $G$.
In other words, $w$ is an extreme vertex of $G$
if no path in $G$ with endvertex $w$
is longer than the path $w,P$.

We will prove a number of lemmas,
the ultimate goal of which is to show that
the detour sequences of
unicyclic graphs are full.
\begin{lemma}\label{extreme1}
Let $w$ be an extreme vertex of a path-unicyclic graph
$G$ with cycle $C$. Then if $v \in V(C) \setminus \{w\}$
the detour order of $v$ is greater than the detour order
of $w$.
\end{lemma}
\begin{proof}
Let $P$ be the path of order $p$ attached to $w$. 
Then the path $w,P$ is a longest path
from $w$ which contains no
vertices of the cycle $C$, apart from $w$. Therefore
if $v \in V(C) \setminus \{w\}$ we can follow the cycle from
$v$ to $w$ and then use $P$ to get a path with order 
greater than the detour order of $w$. \hspace*{\fill}\qd
\end{proof}

\begin{lemma}\label{extreme2}
A path-unicyclic graph $G$ has at most one extreme
vertex.
\end{lemma}
\begin{proof}
Suppose that $v$ and $w$ are extreme vertices of $G$.
By Lemma~\ref{extreme1} the detour 
order of $v$ is greater than the detour
order of $w$ and vice versa, which 
is a contradiction. \hspace*{\fill}\qd
\end{proof}

\begin{lemma}\label{extreme3}
Let $G$ be a path-unicyclic graph and $v \in V(G)$
be a vertex with the least detour order.
If $G$ has an extreme vertex $w$ with a path $P$
attached to $w$, then $v \in V(P) \cup \{w\}$.
\end{lemma} 
\begin{proof}
First note that if $C$ is the cycle of $G$,
and if $u \in V(C) \setminus \{w\}$ with a path $Q$
attached to $u$, then every vertex of $Q$ has detour
order greater than the detour order of $u$. This is
because a longest path in $G$ 
from $u$ does not contain any
vertex of $Q$.  Hence the result follows from
Lemma~\ref{extreme1}. \hspace*{\fill}\qd
\end{proof} 

\begin{lemma}\label{not-extreme}
Let $G$ be a path-unicyclic graph with cycle $C$.
Suppose that $v$ is not an extreme vertex,
$v \in V(C)$, and $\tau_G(v) =d$.
If $v$ has an attached path $P$, with consecutive vertices 
$v_1, v_2, v_3, \ldots, v_p$, where $v_1$ is adjacent
to $v$, then the detour orders of $v_i$ are $d+i$,
$i=1,2,3, \ldots, p$.
\end{lemma}
\begin{proof}
Since $v$ is not an extreme vertex, a longest
path from $v$
does not contain any vertex of $P$. It also follows
that a longest path from $v_i$ does not contain the vertices
$v_{i+1}, v_{i+2}, \ldots, v_p$ of $P$. Hence the detour
order of $v_{i+1}$ exceeds that of $v_i$ by one, and
the result follows. \hspace*{\fill} \qd
\end{proof} 

\begin{lemma}\label{extreme-subpath}
Let $w$ be an extreme vertex of a path-unicyclic
graph $G$, and $P$ be the path attached to $w$.
Let $Q$ be a detour in $G$. Then $P$ is a 
sub-path of $Q$ such that 
one of the endvertices of $P$
is also an endvertex of $Q$.
\end{lemma}
\begin{proof}
Let $C$ be the cycle 
in $G$. There are only two possible forms for
$Q$:
\begin{enumerate}
\item $V(Q) = V(T) \cup V(C)$, where $T$ is a path
attached to a vertex $u \in V(C)$.
\item $V(Q) = V(L_1) \cup \{v_1,v_2,v_3,\dots,v_l\} \cup V(L_2)$,
where $L_1$ and $L_2$ are paths attached to vertices $x$
and $y$, respectively, of $C$ 
and $v_1, v_2, v_3, \ldots, v_l$ are
consecutive vertices of $C$ lying between $x$ and $y$.
\end{enumerate}
Suppose that $Q$ has the form (1) and $T \neq P$.
Since $w$ is an extreme vertex $|V(P)| + 1 \geq |V(C)|$.
Hence 
$|V(T)| + |V(P)| +2 > |V(T)| + |V(P)| +1 \geq |V(T)| + |V(C)|$.
Thus we get a longer path than $Q$ by starting with
$T$, following the cycle to $w \in V(C)$, and then using
the path $P$.  Hence $T = P$ and $P$ is a sub-path of $Q$. 

Now suppose that $Q$ has the form (2) and that neither
one of the paths $L_1$, $L_2$ are the path $P$. 
Suppose that
$w \in V(Q)$ and $w = v_i$, $1 \leq i \leq l$.
Since $w$ is an extreme vertex we have
\begin{align}
|V(P)| \geq l-i + 1 +|V(L_2)|. \tag{1} 
\end{align}
Since
$|V(P)| > |V(L_2)| > 0$ and $|V(P)| \geq |V(C)| > l-i +1$
equality cannot hold in (1) and
$|V(P)| > l-i + 1 +|V(L_2)|$.
Hence
\[
|V(Q)| = |V(L_1)| + (|V(L_2)| +1+(l-i)) +i < |V(L_1)|+i+1+|V(P)|.
\]
Therefore the path $L_1, x,v_1,v_2,\ldots,v_i,P$ is 
longer than $Q$, contradicting the fact that
$Q$ is a detour in $G$.
If $w \notin V(Q)$ then the path
$L_1,x,v_1,v_2,\ldots,v_l,\ldots,w,P$ is longer than
$Q$, since $|V(P)| > |V(L_2)| +l$, 
giving again a contradiction.
Hence one of the paths $L_1$ or $L_2$ is $P$, so $P$
is a sub-path of $Q$. \hspace*{\fill}\qd
\end{proof}

The next lemma enables us to determine the least
and greatest terms in the detour sequence of a
path-unicyclic graph.
\begin{lemma}\label{max-min}
Suppose that $G$ is a path-unicyclic graph with
cycle $C$. Let $V(C) = \{v_1,v_2,\dots,v_n\}$
and $p_i$ be the order of a path attached to $v_i$
(if no such path exists then we put $p_i = 0$).
Let $m$ and $M$ be the least and greatest
terms in the detour sequence of $G$. 
Then the following hold:
\begin{enumerate}
\item  $M = \max\{\tau(v_i) + p_i \}$,
where the maximum is taken  over all vertices $v_i \in V(C)$
which are not extreme vertices.

\item If $G$ has an extreme vertex then  
$m = \left\lceil \frac{M +1}{2} \right\rceil$.
If $G$ does not have an extreme vertex then
$m = \min \{\tau(v_i) \}$, where the minimum is taken over
all vertices $v_i \in V(C)$.

\item Suppose that  $w$ is an extreme vertex of $G$ with 
attachment path $P$ and that $m \le t \le M$.
Then there exists $v \in V(w,P)$ such that
$\tau_G(v) =t$.
 
\end{enumerate}
\end{lemma}
\begin{proof}
\begin{enumerate}
\item Let $v$ be an endvertex of a detour $Q$ in $G$.
Since by Lemma~\ref{extreme2}
there is at most one extreme vertex in $G$,
and clearly both endvertices of $Q$ cannot lie in a
single attached path of $C$, we may assume that $v$ is not an
extreme vertex and does not lie in the attached path of an extreme
vertex. %%% Changed next 
%%%%%%%%%%%%%%%%%%%%%%%%%%%%%%%%%%%%%%%%%%%%%%%
Then either $v$ lies in 
a path attached to a non-extreme vertex,
or $v = v_i \in V(C)$. % and $p_i =0$. 
In both cases  $\tau(v) \geq M$, 
otherwise, by Lemma~\ref{not-extreme},
we get a path longer than $Q$. %%%%%%%%%%%
%%%%%%%%%%%%%%%%%%%%%%%%%%%%%%%%%%%%%%%%%%%%%%%%%%%%%%%%%%%%%%%%%%
But, by definition of
$M$ and Lemma~\ref{not-extreme}, we have $\tau(v) \leq M$. 
Hence $\tau(v) = M$.
\item If $G$ does not have an extreme vertex, then, by
Lemma~\ref{not-extreme}, a vertex with 
minimum detour order lies on $C$.
Hence $m = \min \{\tau(v_i) \ |\  v_i \in V(C) \}$.
If $G$ has an extreme vertex $w$, then by Lemma~\ref{extreme-subpath}
a detour $Q$ has the path $P$ which is attached
to $w$ as a sub-path, and $P$ and
$Q$ share an endvertex. 
Let $u \in V(G)$ be such that $\tau_G(u) = m$.
By Lemma~\ref{extreme3}
$u \in V(P) \cup \{w \}$, so $u \in V(Q)$.
For all vertices $t \in V(P) \cup \{w\}$ we get
$\tau_Q(t) = \tau_G(t)$.
Hence $\tau_G(u) = \tau_Q(u)$ and
$\tau_Q(u) \le \tau_Q(r)$ for all $r \in V(P) \cup \{w\}$.
Since $w$ is an extreme vertex the path
$w,P$ is a longest path from $w$ in both $Q$ and $G$.
Therefore $\tau_Q(u)$ is the least detour order
in the detour sequence of path $Q$, and it
follows that
\begin{align}
m=\tau_G(u) =\tau_Q(u) = \left\lceil \frac{\tau(Q) +1}{2}\right\rceil
= \left\lceil\frac{M+1}{2}\right\rceil. \notag
\end{align}

\item Let $Q$ be a detour in $G$. It is shown above that
if $u \in V(G)$ with $\tau_G(u) =m$
then $u \in V(w,P)$. By Lemma~\ref{extreme-subpath},
if $v \neq w$ is an endvertex of $w, P$
then $\tau_G(v) = M$. Since  
$w, P$ is a longest path from $w$ in  $Q$
the result follows. \hspace*{\fill} \qd

\begin{comment}
Let $u \in V(G)$ have minimum detour
order $m$. By Lemma~\ref{extreme3}
$u \in V(P) \cup \{w \}$, so $u \in V(Q)$ and
therefore  $\tau_Q(u) = m$, and $\tau_Q(u) \le \tau_Q(q)$
for all $q \in V(Q)$. 
Therefore $m = \left\lceil \frac{M+1}{2} \right\rceil$. \hspace*{\fill}\qd
\end{comment}
\end{enumerate}
\end{proof}
\begin{comment}
\begin{lemma}
Let $G$ be a path-unicyclic graph, and $v$, $w$ be consecutive
vertices on the cycle $C$ of $G$ such that neither $v$ nor
$w$ has a path attached to it. 
Then $|\tau(v) - \tau(w)| \leq 1$.
\end{lemma}
\begin{proof}
If there exists
a longest path from $v$ with endvertex
$w$ then $\tau(w) = \tau(v)$.
For all the other cases $\tau(v) - \tau(w) = \pm 1$,
because $v$ and $w$ are adjacent and $v$ and $w$ do not
have an attachment path. \hspace*{\fill}\qd
\end{proof}
\end{comment}

We will also need the following lemma:

\begin{lemma}\label{full-gap}
Let $C$ be the cycle in a path-unicyclic graph $G$.
Let $v$ and $w$ be adjacent vertices on $C$ such that
$\tau(v) > \tau(w)+1$. Then,
for each integer $t$ such that $\tau(w) \leq t \leq \tau (v)$,
there is a vertex $u \in V(G)$ with $\tau (u) = t$.
\end{lemma}
\begin{proof}
Let $L$ be a longest path from $v$. Then $w \in V(L)$,
otherwise the path $w,L$ would have order %%%Changed here 
%%%%%%%%%%%%%%%%%%%%%%%%%%%%%%%%%%%%%%%%%%%%%%%%%%%%%%%%%%%%%%%%%%%%%%%%%
$\tau(v) +1$, contradicting our assumption
that $\tau(w) < \tau(v)$.
Also, since $\tau (w) < \tau (v) -1$, the vertex $w$ cannot be the 
first vertex  on $L$ after $v$. Hence $V(C) \subset V(L)$,
and $w$ is not an endvertex of $L$, otherwise $\tau (w) = \tau (v)$.
Therefore there is a path attached to $w$. If $w$ is an extreme
vertex the result follows by Lemma~\ref{max-min}, part (3).
%since 
%and Lemma~\ref{extreme3}. 
%%%%%%%%%%%
%% Explain more use of above lemmas
%%%%%%%%
%%%%%%%%%%%%%%%%%
Suppose that $w$
is not an extreme vertex and  let   
$P_w$ be the path attached to $w$.
Then $\tau (v) = |V(P_w)| + |V(C)|$.   
By Lemma~\ref{not-extreme} the vertices on $P_w$ have detour orders
$\tau (w) +k$, for $k=1,2,3,\ldots,p$ where 
$p = |V(P_w)|$. We have $\tau (w) <  \tau (v)$.
Thus to ensure that detour orders lying between 
$\tau (w)$ and $\tau (v)$ occur on $P_w$
we must find an integer $t$ with $1 \leq t \leq |V(P_w)|$ such
that $\tau (w) +t = |V(P_w)| + |V(C)|$. Hence
$t = |V(P_w)| + |V(C)| - \tau (w)$. Since
$\tau (w) \geq  |V(C)|$ and
$\tau (v) - \tau (w) > 1$
it follows that $1 < t \leq |V(P_w)|$ as required. \hspace*{\fill}\qd
\end{proof}

Now we can show that path-unicyclic
graphs have full detour sequences:
\begin{theorem}\label{path-uni}
The detour sequence of a path-unicyclic graph $G$ is full.
\end{theorem}
\begin{proof}
Suppose that $G$ contains an extreme vertex
$w$ with an attachment path $P$. Then it follows
from  part (3) of Lemma~\ref{max-min} that 
$V(P) \cup \{w\}$ contains vertices having all detour
orders between the least and greatest detour order
of $G$.

Now suppose that $G$ does not  contain
an extreme vertex. Let $m$ be the smallest number in the detour
sequence of $G$. By Lemma~\ref{not-extreme}
there exists $v \in V(C)$ with $\tau (v) =m$.
Give the cycle $C$ an orientation and 
let $v_1, v_2, v_3,\ldots,v_n$ (where $n = |V(C)|$) 
be consecutive vertices
on $C$, with $\tau (v_1) = m$.
We will show by induction that, for every 
$i=1,2,3,\ldots,n$, and for every
integer $t$ such that $m \leq t \leq \tau(v_i)$,
there exists $w \in V(G)$ with $\tau(w) = t$.
This is obviously true for $i=1$. Suppose that
it is true for $i =r-1$, for some $r$ with $2 \leq r \leq n$.
If $\tau (v_r) \leq \tau (v_{r-1}) + 1$, then
it follows from the induction hypotheses that the result
holds for $i =r$. If $\tau (v_r) \geq \tau (v_{r-1}) + 2$,
then by Lemma~\ref{full-gap} the result holds for $i = r$.
Hence, by induction, the result holds for every
$i=1,2,3,\ldots,n$.  By Lemma~\ref{max-min} 
it follows that, for some vertex 
$v_l \in V(C)$ with attachment path $P_l$, we
have $\tau (G) = \tau (v_l) + |V(P_l)|$.
We have shown above that every integer 
between $m$ and $\tau (v_l)$ is in
the detour sequence of $G$, and, by Lemma~\ref{not-extreme},
every integer between $\tau (v_l)$ and $\tau (G)$
is also in the detour sequence of $G$.
Hence the detour sequence of $G$ is full. \hspace*{\fill} \qd
\end{proof}
%% Perhaps have an earlier lemma: If $L$
%%% is a detour in a tree then then the detour sequence 
%%%% of the tree is obtained by repeating some
%%% of the detour orders of the path $L$. 
%%% Another Lemma If $L$ is a detour in a tree
%%% then least and greatest detour orders 
%%% occurs on $L$
%%% Or lemmas about how a tree is constructed
%%% from a longest path in the tree
%%% Or a lemma about the longest path
%%% from a vertex of a tree

We can  extend Theorem~\ref{path-uni}
to unicyclic graphs as follows:
\index{detour sequence of a unicyclic graph}
\begin{theorem}\label{unicyclic}
The detour sequence of a unicyclic graph $G$ is full.
\end{theorem}
\begin{proof}
By Lemma~\ref{unicyclic-Harary} the graph $G$ consists 
of a cycle $C$, with disjoint trees
$T_i$ attached to some of the vertices $v_i \in V(C)$,
such that $|V(C) \cap V(T_i)| =1$. The theorem is obviously
true if $G= C$, so we may assume that $G$ contains at least
one tree attached to $C$. Let $Q_i$ be a detour in $T_i$
with endvertex $v_i$. Let $G^{\prime}$ be the path-unicyclic
graph with cycle $C$ and paths $Q_i$ attached to $v_i \in V(C)$.
If $w \in V(C)$ then 
$\tau_G (w) = \tau_{G^\prime}(w)$,
since the detour orders of vertices on $C$ depends only
on $|V(C)|$ and the orders of  longest paths with an
endvertex $v_i$ in the attached trees $T_i$.
Also, if $u \in Q_i$, then
$\tau_G(u) = \tau_{G^{\prime}}(u)$,
and if $u_1 \in T_i$ and $u_2 \in T_i$ are adjacent
then $|\tau_G(u_1) - \tau_G(u_2)| \leq 1$ since
$u_1,u_2$ is the only path in $G$ which joins
$u_1$ and $u_2$.
Since the detour sequence of $G^{\prime}$ is full
it follows that the detour sequence of $G$ is full. \hspace*{\fill} \qd
\end{proof}

The detour sequence of any path-unicyclic graph is generated by the
following algorithm:\index{detour sequence algorithm}\\
Let $v_1,v_2,\ldots,v_n$ be consecutive
vertices on the cycle. Let $k_i$ denote boxes
(whose contents will be positive integers)
attached to $v_i$. We also denote the contents
of box $k_i$ by $k_i$, and the phrase
``replace $k_i$ with $x$"
means ``replace the contents of box
$k_i$ with $x$".

\begin{enumerate}
\item Initially $k_i = |V(C)| = n$ for $i=1,2,\ldots,n$.
\item Attach a path $P_i$ of order  $p_i$ 
to a vertex $v_i$ of the cycle. 
\item If $p_i \geq k_i$, replace $k_i$ with $p_i +1$.
\item If $i >1$ and  $k_{i-j} < |V(C)| + p_i - j+1$, 
$j=1,2, \dots, i-1$,
replace $k_{i-j}$ with $|V(C)| + p_i - j +1$.
\item If $i < n$ and $k_{i+j} < |V(C)| + p_i - j+1$, 
$j=1,2, \dots, n-i$,
replace $k_{i+j}$ with $|V(C)| + p_i - j +1$.
\item Iterate steps 2 to 5 for each
attached path.
\end{enumerate}

Then $k_i$ is the detour order of vertex $v_i$ on the cycle
when paths $P_i$ are attached  to vertices $v_i$. This is obvious
since at each step $k_l$ either remains the same if there
is not a longer path from $v_l$ using the 
newly attached  path $P_i$, or is
replaced by the order of a longest path from $v_l$ 
which uses the newly attached path $P_i$. The detour orders
of the vertices which do not belong to the cycle 
are obtained by comparing $k_i$ with $p_i$ as follows:
\begin{enumerate}
\item If $k_i > p_i +1$ then the  vertices of $P_i$ have detour
orders $k_i +1, k_i+2, \ldots, k_i + p_i$. 
This follows from Lemma~\ref{not-extreme}. 
\item If $k_i = p_i +1$ (the only other possibility by the algorithm),
then the detour orders on $P_i$ first
decrease to a minimum away from the cycle and then increase to a maximum.
We can use Lemma~\ref{max-min} to
obtain this minimum and maximum as follows: (In the following,
if there is no attached path to vertex $v_i$ of the cycle, take
$p_i =0$.) Let $M = \max \{k_i + p_i\}$, where the maximum is 
taken over all vertices $v_i$ of the cycle 
such that $p_i + 1 < k_i$. 
%If no such
%vertices exist, take the maximum over all vertices of
%the cycle. 
Let $m = \left\lceil \dfrac{M+1}{2} \right\rceil$.
If $M$ is odd then the detour orders of vertices on $v_i,P_i$ are 
$k_i,k_i -1, k_i - 2, \ldots,m-1, m, m+1, m+2, \ldots, 2m-1$.
If $M$ is even, the minimum $m$ occurs twice, and the
detour orders on $v_i,P_i$ are
$k_i, k_i-1, k_i-2, \ldots, m,m,m+1,m+2, \ldots, 2m-2$.
Note that it may be that $m = k_i$.
\end{enumerate}   
%We will call a detour sequence obtained by this algorithm a 
%\textit{path-unicyclic detour sequence}.
%\index{path-unicyclic detour sequence}

\noindent \textbf{Example.}\\
Suppose that the cycle has order 8, and we attach paths of order
4, 3, and 7 to vertices 
$v_3$, $v_6$ and $v_8$ respectively. After attaching the path
of order 4 to $v_3$ the algorithm gives:
\[
k_1 = 11, k_2 =12, k_3 = 8, k_4=12, k_5=11, k_6=10, k_7 =9, k_8 =10.
\]
Now attach the path of order 3 to vertex $v_6$. By the algorithm
we get:
\[
k_1 =11, k_2 =12, k_3 =9, k_4 =12, k_5 =11, k_6 =10, k_7 =11, k_8 =10
\]
Now attach the path of order 7 to $v_8$. 
The algorithm gives:
\[
k_1 =15, k_2 =14, k_3 =13, k_4 = 12, k_5 =13, k_6 =14, k_7 = 15, k_8 =10
\]
Since $k_3 > 5$, $k_6 > 4$ and $k_8 > 8$ the 
vertices of the path of order 4 
have detour orders 14, 15, 16, 17;  
the vertices of the path of order 3 
have detour orders 15,16,17; 
and the vertices of the path of
order 7 have detour orders 11,12,13,14,15,16,17.
The detour sequence is therefore
\[
10, 11, (12)_2, (13)_3, (14)_4, (15)_5, (16)_3, (17)_3.
\]
As expected, we see that every detour order 
between the least and the 
greatest occurs in this 
detour sequence, i.e. the detour sequence
is full. \hspace*{\fill}\qd

Another family of graphs with full detour sequences
are complete multipartite graphs, \index{complete multipartite graph} 
defined as follows:
\begin{definition}\label{multipartite}
A graph $G$ is a complete multipartite graph
if the vertex set $V(G)$ can be partitioned into $k \geq 2$
subsets $V_1,V_2,\ldots,V_k$ such that every element
of $E(G)$ joins a vertex of $V_i$ to a vertex of $V_j$,
$i \neq j$, and if $u \in V_i$ and $v \in V_j$
then $uv \in E(G)$.
\end{definition}
We denote a complete  multipartite graph
which has  $|V_i| = n_i$, 
$i=1,2,\ldots,k$ by $K(n_1,n_2,\ldots,n_k)$
or $K_{n_1,n_2,\ldots,n_k}$. The complete  multipartite graph
$K_{1,3}$ is  called the \textit{claw}\index{claw}.
The vertex
in $K_{1,3}$ which has degree three is called 
the \textit{centre}\index{centre of a claw}
of the claw.

We can determine the detour sequences of complete
multipartite graphs as follows:
\index{detour sequence of complete,\\ multipartite graphs} 
\begin{theorem}\label{multi-seq}
Let $G = K(n_1,n_2,n_3,\ldots,n_p)$, where $p \geq 2$
and $0 \leq n_1 \leq n_2 \leq n_3 \leq \ldots \leq n_p$.
Let $N = n_1+n_2+n_3+\ldots +n_p$.
Then the following holds:  
\begin{enumerate}
\item[(a)] If $2n_p \leq N$ then $G$ 
has detour sequence $(N)_N$.
\item[(b)] If $2n_p > N$ then the detour sequence of $G$
is $(2(N -n_p))_{N-n_p},\ (2(N-n_p)+1)_{n_p}$.
\end{enumerate}
\end{theorem}   
\begin{proof}
The theorem is obviously true for $N \leq 2$. We assume
that $N \geq 3$. 
\begin{enumerate}
\item[(a)] We have $(N - n_p) \geq  \dfrac{N}{2}$.
Since $\delta(G) = N - n_p$ we get $\delta(G) \geq \dfrac{N}{2}$.
Hence by the well known theorem of Dirac~\cite{dirac} it 
follows that $G$ is hamiltonian.  Note that $N \geq 3$.

\item[(b)] Let $V_p$ be a partite set with $|V_p| = n_p$.
A longest path from $v \in V_p$ uses vertices
from $V_p$ and $V(G) \setminus V_p$ alternately.
Since $|V(G) \setminus V_p| = N -n_p$, and such a path starts
and ends in $V_p$, its order is
$2(N-n_p)+1$.

Let  $w \in V(G) \setminus V_p$
be  a vertex 
in  a partite set $V_j$, where $j \ne p$.
Then a longest path from $w$ also use vertices
from $V(G) \setminus V_p$ and $V_p$ alternately,
but now such a path starts in $V_j$ and ends
in $V_p$, so its order is
$2(N-n_p)$.\  \hspace*{\fill}\qd
\end{enumerate}
\end{proof}
\begin{corollary}
A sequence is a detour sequence of a complete
multipartite graph if and only if it has one of the two
following forms:
\begin{enumerate}
\item $(N)_N$, where $N$ is a positive integer.
\item $(2n)_n, (2n+1)_m$, where $m$, $n$ are positive integers
and $m > n$.
\end{enumerate}
\end{corollary}
\begin{proof}
The ``only if" part of the proof follows directly from 
Theorem~\ref{multi-seq}. We prove the ``if" part:
\begin{enumerate}
\item Suppose we have a sequence of this form. Then the complete 
multipartite graph $K(1,1,1,\ldots,1)$ with $N$ singleton
partite sets has detour sequence $(N)_N$. 
\item The complete multipartite graph 
$K(1,1,1,\ldots,1,m)$ with $n$ singleton partite sets and one
partite set with $m$ vertices, $m >n$, has the required
form for its detour sequence. \hspace*{\fill}\qd
\end{enumerate}
\end{proof}

\section{Detour chromatic numbers}\label{det-chrom-num}
We conclude this chapter by establishing
an upper bound for the
n\emph{th detour-chromatic number} of a 
graph $G$ by using the greatest term in the 
detour sequence of $G$. We first introduce
the relevant terminology and theorems on 
which our result is based.

An $n$\emph{-detour colouring} \index{detour colouring of a graph}
of $G$ is a colouring of the vertices of $G$ such
that no path of order greater than $n$ 
is monocoloured. The $n$\emph{th detour-chromatic number} of $G$, 
\index{detour chromatic number of a graph}
denoted by $\chi_{n}(G)$, is the
minimum number of colours required for an $n$-detour 
colouring of $G$. These
chromatic numbers were introduced by Chartrand, 
Geller and Hedetniemi in 1968
(see \cite{cgh}).

The path of order $n$ is denoted by $P_{n}$.  We say that a 
set $W$ of
vertices in $G$ is $P_{n+1}$\emph{-free } 
\index{$P_{n+1}$\emph{-free} set of vertices}
if $G\langle W \rangle$ 
has detour order at most $n$. 
Thus an $n$-detour colouring
of $G$ corresponds to a partition of the vertex 
set of $G$ into $P_{n+1}$-free sets.

The first detour-chromatic number, 
$\chi_{1},$ is the ordinary chromatic
number $\chi$. The following bound 
for $\chi$ is well-known (see for example 
\cite{cl}, Corollary 8.8, on page 226).

\begin{theorem}
$\chi(G)\leq\tau(G)$ for every graph $G$.
\end{theorem}

The above result follows from the 
observation that $\tau(G-M)\leq\tau(G)-1$
for every maximal independent set $M$ of $G$.

The following bound for $\chi_{n}$ appears in \cite{cgh}.

\begin{theorem}
\label{thm:bound1}
If $G$ is any graph and $2\leq n\leq \tau (G)-1$, then 
\begin{equation*}
\chi _{n}(G)\leq 
\left\lfloor \frac{1}{2}(\tau (G)-n-1)\right\rfloor +2.
\end{equation*}
\end{theorem}

The proof of Theorem \ref{thm:bound1} relies on the 
observation that $\tau(G-M)\leq\tau(G)-2$ for every maximum 
$P_{n+1}$-free subset $M$ of $G$,
if $2 \leq n \leq \tau(G)$.  In 
\cite{dfb} the following stronger result
was proved:

\begin{theorem}
\label{thm:bound2}
Let $G$ be a graph and $n$ an integer such that $2\leq n\leq\tau(G)$. If $M$
is a maximal $P_{n+1}$-free subset of $V(G)$, then 
\begin{equation*}
\tau(G-M)\leq\tau(G)-\frac{2n+2}{3}\,.
\end{equation*}
\end{theorem}

This result enables us 
to prove the following theorem:
\begin{theorem}
If $G$ is any graph, then 
\begin{equation*}
\chi_{n}(G)\leq\left\{ 
\begin{array}{ll}
\left\lceil \frac{\tau(G)-n}{\lceil(2n+2)/3\rceil}\right\rceil +1 & 
\mbox{
if }2\leq n\leq\tau(G) \\ 
1 & \mbox{ if }n > \tau(G).
\end{array}
\right.
\end{equation*}
\end{theorem}
\begin{proof}
We use induction on $\tau(G)$. The result obviously holds for all graphs $K$
with $\tau(K)=2$. Suppose the result is true for all graphs $H$ 
with $\tau(H)<k$ for some $k>2$. Let $G$ be an arbitrary graph with 
$\tau(G)=k$. If $n\geq k$ the 
result holds for $G$, so we may suppose that $n<k$. Let $M$ be
a maximal $P_{n+1}\mathrm{-free}$ subset of $V(G)$. 
By Theorem~\ref{thm:bound2} 
\begin{equation*}
\tau(G-M)\leq k-\left\lceil \frac{2n+2}{3}\right\rceil <k
\end{equation*}
and therefore, by the induction assumption, 
\begin{equation*}
\chi_{n}(G-M)\leq\left\{ 
\begin{array}{cc}
\left\lceil \dfrac{\tau(G)-\left\lceil \frac{2n+2}{3}\right\rceil -n}
{\left\lceil (2n+2)/3\right\rceil }\right\rceil +1 
& \mathrm{if}\quad2\leq
n\leq\tau(G-M) \\ 
1 & \mathrm{if}\quad n>\tau(G-M).
\end{array}
\right.
\end{equation*}
By including the subset $M$ in any $P_{n+1}\mathrm{-free}$ 
partition of $G-M$
we get a $P_{n+1}\mathrm{-free}$ partition of $G$. Hence 
\begin{equation*}
\chi_{n}(G)\leq\chi_{n}(G-M)+1.
\end{equation*}
We now verify that the inequality for $\chi_{n}(G)$ holds for all the
possible values for $n$. First, if $2\leq n\leq\tau(G-M)$ then 
\begin{align*}
\chi_{n}(G) & \leq\left\lceil \dfrac{\tau(G)-\left\lceil \frac{2n+2}{3}
\right\rceil -n}{\left\lceil (2n+2)/3\right\rceil }\right\rceil +1+1 \\
& =\left\lceil \dfrac{\tau(G)-n}{\left\lceil (2n+2)/3\right\rceil }%
\right\rceil +1.
\end{align*}
Next, if $\tau(G-M)<n\leq\tau(G)-\lceil(2n+2)/3\rceil$ then 
\begin{align*}
\chi_{n}(G) & \leq1+1 \\
& \leq\left\lceil \dfrac{\tau(G)-n}{\left\lceil (2n+2)/3\right\rceil }
\right\rceil +1
\end{align*}
because in this case 
\begin{equation*}
\frac{\tau(G)-n}{\left\lceil (2n+2)/3\right\rceil }\geq1.
\end{equation*}
Finally, if $\tau(G)-\lceil(2n+2)/3\rceil < n<k$ then 
\begin{align*}
\chi_{n}(G) & \leq1+1 \\
& =\left\lceil \dfrac{\tau(G)-n}{\left\lceil (2n+2)/3\right\rceil }%
\right\rceil +1
\end{align*}
because in this case 
\begin{align}
0<\frac{\tau(G)-n}{\left\lceil (2n+2)/3\right\rceil } < 1. \tag*{\qd}
\end{align}
\end{proof}

\chapter{Connected, Nontraceable, Detour Graphs}

\section{Introduction}
\begin{comment}
Homogeneously traceable graphs are therefore detour graphs
with detour deficiency zero.
 These constructions are
used to give examples of nontraceable detour graphs
with chromatic number $k$, $k \geq 2$, and girths up to 7.
Moreover it is shown that, for all positive integers
$l \geq 1$ and $k \geq 3$, there are nontraceable 
detour graphs with chromatic number $k$ and 
detour deficiency $l$. 
\end{comment}
In this chapter we are concerned with connected, nontraceable, 
graphs which have constant detour sequences, that is,
they have the largest possible number of 
repetitions in their detour sequences. A graph is called
a \textit{detour graph}\index{detour graph}
if its detour sequence is constant. 
We begin by introducing more graph terminology.

A path in a graph $G$ is called a 
\textit{hamiltonian path} \index{hamiltonian path} if it 
contains all the vertices of $G$. A graph which
contains a hamiltonian path is said
to be \textit{traceable}.\index{traceable graph} 
Similarly, a cycle in $G$ is called 
a {\it hamiltonian cycle} \index{hamiltonian cycle}
if it contains 
all the vertices of $G$. 
In such a case $G$ is called a \textit{hamiltonian} graph. 
\index{hamiltonian graph}
A graph $G$ 
is called \textit{hamiltonian connected} 
\index{hamiltonian connected graph} if each 
pair of distinct vertices $u,v$ are endvertices of 
a hamiltonian path in $G$ 
(see \cite{sk84}). A graph $G$ is said to be 
\textit{homogeneously traceable} 
\index{homogeneously traceable graph} if every vertex is an initial 
vertex of a hamiltonian path in $G$. 
Thus a homogeneously traceable graph is a detour
graph with detour deficiency zero.

A \textit{component}\index{component of a graph}
of a graph $G$ is a maximal connected subgraph of $G$,
and a \textit{cut-vertex}\index{cut-vertex of a graph}
of $G$ is a vertex $v \in V(G)$ such that
$G-v$ has more components than $G$. 
A connected graph $G$
is \textit{2-connected}\index{2-connected graph}
if $G$ has no cut-vertices
and $|V(G)| \ge 3$.

Lastly, a graph operation that will be useful in our
constructions of detour graphs is the 
\textit{cartesian product} 
\index{cartesian product of graphs} of two graphs
$G_1$ and $G_2$, denoted by $G_1 \times G_2$.
The vertex set of $G_1 \times G_2$ is
$V(G_1) \times V(G_2)$, and two vertices
$(u_1,u_2)$ and $(v_1, v_2)$ are adjacent
if and only if either $u_1 = v_1$
and $u_2v_2 \in E(G_2)$, 
or $u_2 = v_2$ and $u_1v_1 \in E(G_1)$.

The study of nonhamiltonian, homogeneously traceable
graphs (abbreviated to NHHT graphs) \index{NHHT graphs}
was initiated by Skupi\'{e}n in 1975
(see \cite{sk84}, \cite{sk80}), and continued
by Chartrand, Gould and Kapoor~\cite{chago79}.
From an existence theorem proved in \cite{chago79} 
we know that there are no
NHHT graphs of order
$3,4,\dots,8$, but that NHHT graphs
exist for all orders greater than 8 (see also \cite{sk84}).
The class of connected, nontraceable, detour graphs 
(abbreviated to CND graphs)\index{CND graphs} 
shares many of the properties of NHHT 
graphs, and CND graphs can be seen as a
natural generalisation of NHHT graphs.

In Section~\ref{properties} we state some properties of
CND graphs which are shared by NHHT
graphs, and we use these
properties to show that the detour order
of a CND graph must be greater than eight.

In Section \ref{construct} 
we establish the existence of
CND graphs by giving constructions for
infinite families of CND graphs of all orders
greater than 17, and all detour deficiences greater
than zero. These constructions also give
examples of CND graphs with girths up to 7.

In
Section~\ref{bipartite} 
we give another construction
for CND graphs.
We use this  construction 
to show that for arbitrary positive 
integers $k \geq 2$ and $l \geq 1$ there
exist CND graphs with chromatic number $k$
and detour deficiency at least $l$.
The results in this section were obtained
in collaboration  with G. Semani{\v s}in
and R. Vla{\v c}uha.
In Section~\ref{open} we conclude this chapter by 
stating some open problems.

\section{Properties of CND graphs}
\label{properties}

The following simple, but useful,
lemma gives some properties of detours in 
nontraceable
graphs that will be required in the sequel.

\begin{lemma}
\label{obvious}
Let $P : v_1,v_2, v_3, \ldots, v_{p-1}, v_p$
be a detour in a 
nontraceable graph $G$.
Let $H$ be a component of $G - V(P)$.
Suppose that  $v_l \in V(P)$ is adjacent to a vertex
$w \in V(H)$,  and a vertex $v_k \in V(P)$
is adjacent to a vertex $u \in V(H)$
(w=u is allowed),
where $l > k$.
Then
\begin{enumerate}
\item $v_k$ and $v_l$ are not
consecutive vertices
of $P$ i.e. $l \geq k+2$.

\item $v_k \neq v_1$ and $v_l \neq v_p$.
\item  $v_1$
is not adjacent to $v_{l-1}$ or $v_{l+1}$
or $v_{k+1}$, and 
$v_p$ is not adjacent to $v_{k-1}$ 
or $v_{k+1}$ or $v_{l-1}$.

\item $v_1$ is not adjacent to $v_p$.

\item If $v_1$ is adjacent to $v_j$, then $v_p$ is
not adjacent to $v_{j-1}$.
\end{enumerate}
\end{lemma}
\begin{proof}
In this proof $L$ denotes any path in $H$ joining
the vertices $u$ and $w$.
\begin{enumerate}
\item Suppose that $v_k$ and $v_l$ 
are consecutive vertices on $P$.
Then
$v_1,v_2,\ldots,v_k,L,$\newline
$v_l,\ldots,v_p$ 
is a path longer than $P$, which 
contradicts the fact that $P$ is a detour.

\item If $v_k = v_1$ (i.e. if $k=1$) then $v_1$ has a
neighbour which is not a member of $V(P)$ and we get a path
longer than $P$. A similar argument holds
if $v_p = v_l$. 

\item If $v_1$ is adjacent to $v_{l-1}$ then the
path 
$v_{k-1},\ldots,v_1,v_{l-1}, v_{l-2},\ldots, v_k,L,v_l,\ldots,v_p$
is longer than $P$.
See  Figure~\ref{longer-1}. 
If $v_1$ is adjacent to $v_{l+1}$ then the path 
$v_{k+1},v_{k+2},\ldots,v_l,L,v_k,v_{k-1},
\ldots,v_1,v_{l+1},v_{l+2},\ldots,v_p$
is longer than $P$. If $v_1$ is adjacent to $v_{k+1}$ then the path
$v_p,\ldots,v_l,L,v_k,v_{k-1},
\ldots,v_1,v_{k+1},v_{k+2},\ldots,v_{l-1}$
is longer than $P$. Thus $v_1$ is not adjacent to $v_{l-1}$,
$v_{l+1}$, or $v_{k+1}$, and, by symmetry, $v_p$ is not adjacent to
$v_{k-1}$,$v_{k+1}$ or $v_{l-1}$.
\begin{figure}[H]
\begin{center}
\psfrag{a}[c][c]{$v_1$}
\psfrag{b}[c][c]{$v_k$}
\psfrag{c}[c][c]{$v$}
\psfrag{d}[c][c]{$v_{l-1}$}
\psfrag{e}[c][c]{$v_l$}
\psfrag{f}[c][c]{$v_p$}
\psfrag{g}[c][c]{$H$}
\psfrag{h}[c][c]{$u$}
\psfrag{i}[c][c]{$w$}
\psfrag{j}[c][c]{$v_{k-1}$}
\psfrag{k}[c][c]{$v_{k+1}$}
\psfrag{l}[c][c]{$v_{l+1}$}
\epsfig{file=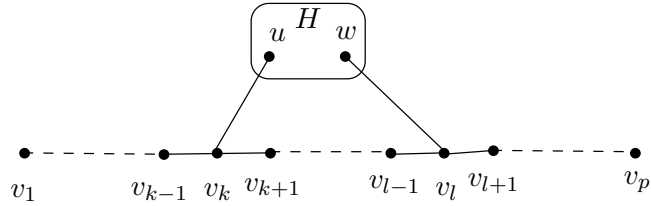}
\end{center}
\caption{\label{longer-1} Diagram for Lemma~\ref{obvious}, Part 3}
\end{figure}

\item If $v_1$ is adjacent to $v_p$
then the vertices of $P$ form a cycle
of order $p$. It follows that $\tau(u) > p$,
and we get a path longer than $P$.

\item If $v_p$ is adjacent to $v_{j-1}$ then $G$
contains the cycle 
$v_1,v_2,\ldots,v_{j-1},v_p, v_{p-1},\ldots,v_j,v_1$
of order $p$. See Figure~\ref{longer}. 
As in (4) above, this implies that
$P$ is not a detour in $G$.  \hspace*{\fill}\qd
\end{enumerate} \end{proof}
\begin{figure}[H]
\begin{center}
\psfrag{a}[c][c]{$v_1$}
\psfrag{b}[c][c]{$v_{j-1}$}
\psfrag{c}[c][c]{$v_j$}
\psfrag{d}[c][c]{$v_p$}
\epsfig{file=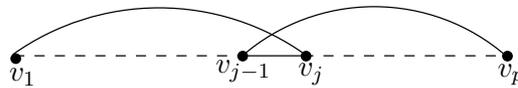}\\
\caption{\label{longer} Diagram for Lemma~\ref{obvious}, Part 5}
\end{center}
\end{figure}
As we have already stated, many of the properties 
of NHHT graphs are shared by
CND graphs. The proof of
the following theorem is very similar
to the proofs given in 
Skupie\'{n}~\cite{sk84}, Skupie\'{n}~\cite{sk81} and 
Chartrand, Gould and Kapoor~\cite{chago79} for NHHT graphs. 
One has
only to replace,
where necessary,
the order of the graph with the
detour order. But since these results are scattered amongst
the above references, we include the proofs to make the
thesis self-contained.
\begin{theorem}\label{prop-th}
Let $G$ be a CND 
graph of order $n$.
Then
\begin{enumerate}
\item $G$ is $2$-connected (hence $\delta (G) \ge 2$).
\item Every vertex of $G$ has at most one neighbour of degree 2.
\item $G$ contains a detour $P$ such that the endvertices
of $P$ each have degree at least three.
\item $\Delta (G) \le \tau(G)-4$.
\item If $T$ is the set of vertices of degree 2 in $G$, then
$|V(G) \setminus T | \ge |T|$.
\item $|E(G)| \ge \left\lceil \dfrac{5n}{4} \right\rceil$.
\end{enumerate}
\end{theorem}
\begin{proof}
\begin{enumerate}
\item Clearly $|V(G)| >3$.
Suppose that $v \in V(G)$ is a cut-vertex.
Let $G_1$ and $G_2$ be components of $G - v$.
Let $L$ be a longest path starting at $v$. Since $v$
is a cut-vertex we may suppose that $V(L) \subset V(G_1) \cup \{v\}$.
Suppose that $w \in V(G_2)$. Then $w \notin V(L)$, and since
$G_2 \cup \{v\}$ is connected, it follows that
$\tau (w) > \tau (v)$, which contradicts the fact that
$G$ is a detour graph.
\item Suppose that $v \in V(G)$ has two neighbours, 
$u$ and $w$, each  with degree two. Let $L$ be a detour from $v$.
Because $\tau(u) = \tau(v) =\tau(w)$ it follows that
$u \in V(L)$ and $w \in V(L)$. Since $\deg (u) = \deg (w) = 2$
either $u$ or $w$ is an endvertex of $L$.
Suppose, without loss of generality,
that $u$ is an endvertex of $L$.
Then the vertices $V(L) \cup \{u\}$ form
a cycle $C$ with $|V(C)| = \tau (G)$. Since $\tau (G) < |V(G)|$
there exists a vertex $v_1 \in V(G) \setminus V(C)$. Since
$G$ is connected $\tau (v_1) > \tau (v)$, which contradicts
the fact that $G$ is a detour graph.  
\item Let $L_1: v_1,v_2,v_3,\ldots,v_{n-1},v_k$ 
be a detour in $G$.
Since $L_1$ is a longest path and, by part (1),
$\deg (v_1) \geq 2$,
the vertex $v_1$ is adjacent to a vertex of $L_1$, say 
$v_m$ where $3 \le m < k$. Then $v_m$ is adjacent to
$v_1$ and $v_{m-1}$. By part (2) at least one
of $v_1$ and $v_{m-1}$ has degree greater than two.
If $\deg(v_1) \geq 3$ then $L_1$ is a detour which has
an endvertex with
degree at least $3$.  If
$\deg (v_{m-1}) \geq 3$ then the path
$L_2 : v_{m-1},v_{m-2},v_{m-3},\ldots,v_1,v_m,v_{m+1},v_{m+2},\ldots,v_k$
is a detour with an endvertex of degree at least $3$.
Hence we may assume that $G$ has a 
detour, say $L_3$, with an endvertex
$w$ of degree at least $3$. We can now apply 
the same argument used above to show that, if the other
endvertex of $L_3$ does not have degree at least 3,
then  we can
construct another detour $Q$, which also has $w$ as an endvertex,
and such that the other endvertex of $Q$ 
has degree at least $3$.

\item Let $v_1 \in V(G)$ with $\deg (v_1) = \Delta (G)$.
Using the same argument given in part (3) we can choose
a detour  $Q : v_1,v_2,v_3,\ldots,v_{n-1},v_k$ 
such that
$v_k$ has degree at least $3$. All the neighbours of
$v_k$ (and $v_1$) lie on $Q$. Thus, apart from $v_{k-1}$,
the vertex $v_k$ is adjacent to at least two vertices
of $Q$. By Lemma~\ref{obvious} parts (5) and (4) 
it follows that 
that $v_1$ is adjacent to at most $\tau (G) -4$
vertices of $Q$. Hence $\Delta (G) \leq \tau(G)-4$.

\item By part (2) of this theorem, no vertex  in $T$
is adjacent to two vertices of degree two. Hence each
vertex in $T$ is adjacent to at least one 
vertex of $V(G) \setminus T$. Since no vertex of
$V(G) \setminus T$ is adjacent to two vertices
of $T$, it follows that
$|V(G) \setminus T| \geq |T|$. We depict the
situation in Figure~\ref{deg-two}.
\begin{figure}[H]
\begin{center}
\psfrag{a}[c][c]{$T$}
\psfrag{b}[c][c]{$V(G)\setminus T$}
\epsfig{file=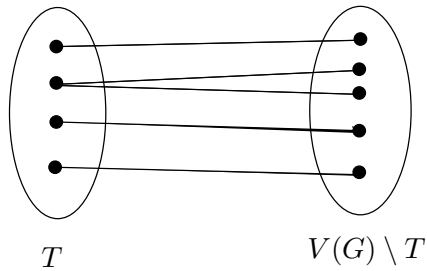}\\
\caption{\label{deg-two} Diagram for Theorem~\ref{prop-th}, Part 5}
\end{center}
\end{figure}
\item The graph $G$ will have the least size possible
when $|V(G) \setminus T| = |T|$, and $\deg (v) =3$
for every $v \in V(G) \setminus T$.
Then $n/2 = |V(G) \setminus T| = |T|$
and therefore
\begin{align}
|E(G)| \geq \frac{2(n/2)+3(n/2)}{2} = \frac{5n}{4}. \tag*{\qd}
\end{align}
\end{enumerate} \end{proof}
We will use Theorem~\ref{prop-th} to establish a lower bound
on $\tau(G)$,  where $G$ is a CND graph.
The following  lemma will be useful. 
\begin{lemma}
\label{deg2}
Let $G$ be a CND graph
with $\tau (G) = 8$. Let 
$P: v_1, v_2, \ldots,v_8$ be a detour
of $G$ whose endvertices each have degree
at least three. Let $H$ be a component of $G-V(P)$. Then
$H$ consists of a single vertex $v$ and $\deg (v) =2$.
\end{lemma}
\begin{proof}
Suppose that $|V(H)| > 1$.
Since $G$ is $2$-connected there exist 
distinct vertices $v \in V(H)$
and $w \in V(H)$ such that $v$ is adjacent to 
$v_l \in V(P)$ and $w$ is adjacent to
$v_k \in V(P)$, where $l > k$.
We may suppose that $k \geq 3$, 
$l \leq 6$ and $l \geq k+3 \geq 6$, otherwise
we get a  path longer than $P$. 
Hence $l = 6$ and $k = 3$.
By Lemma~\ref{obvious}, and since
$\deg (v_1) \geq 3$, it follows that $v_1$
is adjacent to at least two (excluding $v_2$)
vertices
of $P$, which, again by Lemma~\ref{obvious},
must be $v_3$ and $v_6$.
Similarly $v_8$ must also be adjacent to
$v_3$ and $v_6$. But this is not possible,
since by Theorem~\ref{prop-th} we get
$\Delta (G) \leq \tau(G) -4 = 4$. Hence $|V(H)| = 1$.

Next we show that $\deg (v) = 2$.
Since $G$ is
$2$-connected and $\tau(G) =8$ we get
$2 \leq \deg (v) \leq 4$. 
Suppose that $\deg(v) =4$.
Let $v_i$ be the neighbour of $v$ on $P$
which is closest to $v_1$, and $v_j$ be the neighbour 
closest to $v_8$. Then by Lemma~\ref{obvious} part
(1) and (2) we have $i > 1$ and $i+6 \leq j < 8$.
Thus we have the contradiction $1 < i < 2$.
If $\deg (v) = 3$ there are only two cases to
consider:
\begin{enumerate}
\item  $v$ is adjacent to $v_2$, $v_4$ and $v_6$:\\
By Lemma~\ref{obvious} part (3)
vertex $v_1$ must be adjacent to $v_4$
and $v_6$. But for the same reason $v_8$ must be
adjacent to either $v_4$ or $v_6$, which 
contradicts $\Delta (G) \leq 4$.

\item  $v$ is adjacent to $v_2$, $v_5$ and $v_7$:\\
Again, by Lemma~\ref{obvious}(3), 
vertex $v_1$ must be adjacent to $v_5$ and $v_7$,
and $v_8$ must also be adjacent to 
$v_5$,  contradicting  $\Delta (G) \leq 4$. \hspace*{\fill}\qd
\end{enumerate}
\end{proof}

We are now able to give a lower
bound for $\tau (G)$.

\begin{theorem}
\label{tau}
Let $G$ be a CND graph.
Then $\tau (G) \geq 9$.
\end{theorem}
\begin{proof}
By Theorem~\ref{prop-th} $\delta (G) \geq 2$.
Hence $\Delta (G) \geq 3$, otherwise $G$ is
a hamiltonian graph.
By Theorem~\ref{prop-th} $\tau(G) \geq \Delta (G) + 4 \geq 7$.

If $\tau(G) =7$, then by 
Theorem~\ref{prop-th} we get $\Delta (G) = 3$.
Let $Q:\ w_1,w_2,w_3,\ldots, w_7$ be a detour in $G$.
By Theorem~\ref{prop-th}(3) we may suppose that
$\deg (w_1) = 3$ and $\deg(w_7) = 3$.
Let $w_1$ be adjacent to $w_j$ and $w_k$
where $j < k$.
Then, since $\Delta (G) = 3$, and by 
Lemma~\ref{obvious}(4) and (5), $w_7$ is not adjacent to 
$w_1$, $w_{j-1}$, $w_j$ and $w_k$, leaving
only two possible neighbours for $w_7$ on $Q$. Thus
$w_7$ has a neighbour not on $Q$, which
contradicts the fact that $Q$ is a detour.
Hence $\tau (G) \neq 7$.

Now suppose that  $\tau (G) = 8$.
Then, by Theorem~\ref{prop-th}, $\Delta(G) \leq 4$.
Let $P : v_1, v_2, v_3, \ldots, v_8$ be a detour
of $G$ with $\deg (v_1) \geq 3$ and
$\deg (v_8) \geq 3$. Let $v \in V(G) \setminus V(P)$.
By Lemma~\ref{deg2} $v$ is adjacent to exactly
two vertices of $P$.
We show that the following cases, which
by Lemma~\ref{obvious} exhausts all 
possible ways for $v$ to be connected
to $P$,  cannot occur:
\begin{enumerate}
\item $v$ adjacent to $v_2$ and $v_4$:\\
By Lemma~\ref{obvious} vertex  $v_1$ cannot be 
adjacent to $v_3$, $v_5$ or $v_8$,  so 
we need only consider the following subcases:
\begin{enumerate}
\item $v_1$ adjacent to $v_4$ and $v_6$:\\
By Lemma~\ref{obvious},
and since $\Delta(G) \leq 4$, vertex  $v_8$ must 
be adjacent to $v_6$ and $v_2$.
Since $v_2$ is adjacent to $v$,
and $\deg (v) =2$, 
it follows from 
Theorem~\ref{prop-th}(2) that $\deg (v_3) \geq 3$.
But $v_3$ cannot
be adjacent to a vertex on $P$, other
than $v_2$ and $v_4$, without
getting a path longer than $P$, or
a vertex with degree bigger than 4.
Therefore $v_3$ is adjacent
to some vertex $w \in  V(G) \setminus V(P)$,
and it follows from Lemma~\ref{deg2} 
that $w$ is also adjacent to some vertex 
of $P$ other than $v_3$. But this
again implies that we get a path longer
than $P$, or a vertex with degree bigger than 4.

\item $v_1$ adjacent to $v_4$ and $v_7$:\\
Then by Lemma~\ref{obvious} vertex 
$v_8$ must be adjacent to $v_5$
and $v_2$. We then get the path
$v_6,v_5,v_8,v_7,v_1,v_4,v,v_2,v_3$
which is longer than $P$.
\item $v_1$ adjacent to $v_6$ and $v_7$:\\
By Lemma~\ref{obvious}  vertex $v_8$ must be adjacent to $v_2$
and $v_4$. Since $v_4$ is adjacent to $v$ and $v_5$,
and $\deg(v) =2$,  by
Theorem ~\ref{prop-th} we get $\deg (v_5) \geq 3$.
But if $v_5$ is adjacent to $v_7$ we get the path
$v, v_4, v_8, v_7, v_5, v_6, v_1, v_2, v_3$
which is longer than $P$. Since $\deg(v_2) = 4$ vertex $v_5$ is not
adjacent to $v_2$. If $v_5$ is adjacent to $v_3$ we get
the path
$v_1, v_2, v, v_4, v_3, v_5, v_6, v_7, v_8$,
which is  longer than $P$. 
By Lemma~\ref{obvious} vertex $v_5$ is not
adjacent to $v_8$. Therefore $v_5$ is adjacent to some vertex 
$u \in V(G) \setminus V(P)$. By Lemma~\ref{obvious} this
is not possible, because $v_1$ is adjacent to $v_6$.
\end{enumerate}

\item  $v$ adjacent to $v_2$ and $v_5$:\\
By Lemma~\ref{obvious} vertex  $v_1$ is adjacent to $v_5$ and $v_7$. 
Apart from $v_7$, the vertex $v_8$ can only
be adjacent to $v_2$, contradicting $\deg (v_8) \geq 3$.

\item $v$ adjacent to $v_2$ and $v_6$:\\
By Lemma~\ref{obvious}, vertex $v_1$ can only be 
adjacent to $v_4$
and $v_6$, and $v_8$ can only be adjacent to
$v_4$ and $v_2$ (since $\Delta (G) \leq 4$). 
Because $v_6$ is adjacent to $v$ and $v_5$,
and $\deg(v) = 2$, by Theorem~\ref{prop-th}(2) 
it follows that $\deg(v_5) \geq 3$. But $v_5$
cannot be adjacent to $v_7$, for then
we get the path
$v_1, v_4, v_3, v_2,v, v_6, v_5,v_7, v_8$ 
which is
longer than $P$.  Also $v_5$ is not adjacent
to $v_3$, because then we have the path
$v_8, v_7, v_6, v, v_2, v_1,v_4,v_5, v_3$ 
which is longer than $P$.
But if $u \in V(G) \setminus V(P)$ and $u$ is adjacent 
to $v_5$ we again get a path
$u, v_5, v_4, v_3, v_2, v, v_6,v_7, v_8$ 
which is longer than $P$.

\item $v$ adjacent to $v_2$ and $v_7$:\\
By Lemma~\ref{obvious}, vertex $v_1$ is not
adjacent to $v_3$, $v_6$ and $v_8$.
Hence we have three subcases:
\begin{enumerate}
\item $v_1$ adjacent to $v_4$ and $v_5$:\\
Then $v_8$ must be adjacent to $v_5$
and $v_2$, and in that case
$v_3,v_4,v_1,v_2,v,v_7,v_8,$ \newline
$v_5,v_6$
is a path longer than $P$.
\item $v_1$ adjacent to $v_4$ and $v_7$:
\begin{enumerate}
\item If $v_8$ is adjacent to $v_4$ and $v_5$,
or if $v_8$ is adjacent to $v_2$ and $v_5$,
we get a path
$v_3,v_4,v_1,v_2,v,v_7,v_6,v_5,v_8$
which is longer than $P$.
\item If $v_8$ is adjacent to $v_4$ and $v_2$:\\
By Theorem~\ref{prop-th} we get $\deg (v_6) \geq 3$
and $\deg (v_3) \geq 3$. The only possibility
is that $v_6$ is adjacent to $v_3$. But
then we get the path
$v_8,v_7,v,v_2,v_3,v_6,v_5,v_4,v_1$
which is longer than $P$.
\end{enumerate}
\item $v_1$ adjacent to $v_5$ and $v_7$:\\
By Lemma~\ref{obvious} vertex $v_8$
can only be adjacent to $v_2$ and $v_5$
(apart from $v_7$). By Theorem~\ref{prop-th}
we get $\deg(v_6) \ge 3$ and $\deg (v_3) \ge 3$.
The only possibility is that $v_6$ is adjacent to $v_3$.
But then we get the path
$v_6, v_3, v_4, v_5,v_1, v_2, v,$\newline
$v_7, v_8$ 
which is longer than $P$.
\end{enumerate}

\item $v$ adjacent to $v_3$ and $v_5$:\\
By Lemma~\ref{obvious} vertex $v_1$ is not adjacent
to $v_4$, $v_6$ or $v_8$. Thus we have 
three subcases to consider:
\begin{enumerate}
\item $v_1$ adjacent to $v_3$ and $v_5$:\\
It follows from Lemma~\ref{obvious} and $\Delta (G) \le 4$
that, apart from $v_7$, the vertex $v_8$ can only
be adjacent to $v_6$, which
contradicts $\deg (v_8) \geq 3$.
\item $v_1$ adjacent to $v_5$ and $v_7$:\\
It again follows from Lemma~\ref{obvious} and $\Delta(G) \le 4$
that, apart from $v_7$, the vertex $v_8$ can only
be adjacent to $v_3$, 
contradicting $\deg (v_8) \geq 3$.
\item $v_1$ adjacent to $v_3$ and $v_7$:\\
By Lemma~\ref{obvious} and $\Delta(G) \le 4$
vertex $v_8$ can only be adjacent to
$v_5$ and $v_7$, contradicting $\deg(v_8) \ge 3$.

\end{enumerate}

\item $v$ adjacent to $v_3$ and $v_6$:\\
By Lemma~\ref{obvious} the vertex $v_1$ must be adjacent
to $v_3$ and $v_6$. 
Then, apart from $v_7$, the vertex $v_8$
cannot be adjacent to any 
other vertices of $P$, again contradicting
$\deg (v_8) \geq 3$.
\end{enumerate}
By symmetry it follows that:\\
1.  The case $v$ adjacent to $v_3$ and $v_7$
is equivalent to the case $v$ adjacent to  $v_2$ and $v_6$.\\
2. The case $v$ adjacent to $v_4$ and $v_6$
is equivalent to the case $v$ adjacent to $v_3$ and $v_5$.\\
3.  The case $v$ adjacent to $v_4$ and $v_7$
is equivalent to the case $v$ adjacent to $v_2$ and $v_5$.\\
4.  The case $v$ adjacent to $v_5$ and $v_7$
is equivalent to the case $v$ adjacent to $v_2$ and $v_4$.

Thus all cases have been considered. \hspace*{\fill}\qd
\end{proof}

Theorem~\ref{tau} implies immediately that there are
no CND graphs with order less than 10,
and in the next section we show that CND
graphs of all orders greater than 17 exist. Whether
or not a CND graph of order $n$ exists,
where $10 \leq n \leq 17$, remains
an open problem.

\section{Constructions of CND graphs}
\label{construct}

In this section we describe four types of  constructions
for CND graphs.  The 
constructions in this section all follow the same 
pattern: We specify a certain type of multigraph $L$
(ie multiple edges and loops allowed in $L$),
inflate the vertices of $L$ with graphs
of a certain type, and also, in some cases, insert
graphs of another type into the edges of $L$.
Note that, although $L$ is allowed to be a multigraph,
our constructions are such that the CND graphs
constructed will all be simple. 
We will denote the length of a longest trail in 
a multigraph $L$ by $t(L)$,
and $E(L)$ denotes the set of edges of $L$.
We say a trail $L$
\textit{spans} $G$, or is a 
\textit{spanning trail} of $G$,
if every vertex of $G$ belongs to $L$.

\begin{definition}
An \textbf{admissible} multigraph $L$
\index{admissible multigraph}
is a multigraph with the following
properties:
\begin{enumerate}
\item[A-1] $t(L) < |E(L)|$.
\item[A-2] For each vertex $v \in V(L)$, and edge
$e \in E(L)$ which is incident with $v$,
there is a trail of length $t(L)$ 
beginning $v, e, \ldots$ which spans $L$.
\end{enumerate}
\end{definition}
\begin{lemma}
\label{trail}
Let $L$ be an admissible multigraph. 
Then every vertex of $L$
has odd degree greater than one.
\end{lemma}
\begin{proof}
First note that it is obvious that 
$\deg (v) > 1$ for all $v \in V(L)$.
Now suppose that a vertex $v$ of $L$ has even degree.
Let $T$ be a longest trail in $L$ starting at $v$.
Then every edge incident with $v$ 
must be an  edge of $T$ (otherwise we get a longer
trail), and hence the vertex $v$ must 
also be the last vertex of $T$.
Hence $T$ is a closed trail, and since at least one edge $e$
of $L$ does not belong to $T$, we get a trail
in $L$ of length $t(L) +1$. \hspace*{\fill}\qd
\end{proof}

It follows easily from Lemma~\ref{trail} that the 
two smallest, admissible, multigraphs 
are $K_4$ and the multigraph $C_2 \times K_2$
shown in Figure~\ref{admissable}.  
\begin{figure}[H]
\begin{center}
\epsfig{file=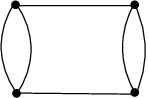}\\
\caption{\label{admissable} The graph $C_2 \times K_2$}
\end{center}
\end{figure}
In fact we have the following lemma:
\begin{lemma}\label{cubic-admissible}
The graphs  $C_n \times K_2$,
where $n \geq 3$, are all cubic, 
admissible graphs.
\end{lemma}
\begin{proof}
Let $C_1 : v_1,v_2,v_3,\ldots,v_n,v_1$ and
$C_2 : w_1,w_2,w_3,\ldots,w_n,w_1$ be two
cycles in $C_n \times K_2$, shown in 
Figure~\ref{m-admissable}.
\begin{figure}[H]
\begin{center}
\psfrag{a}[c][c]{$v_1$}
\psfrag{b}[c][c]{$v_2$}
\psfrag{c}[c][c]{$v_3$}
\psfrag{d}[c][c]{$v_4$}
\psfrag{e}[c][c]{$v_{n-2}$}
\psfrag{f}[c][c]{$v_{n-1}$}
\psfrag{g}[c][c]{$v_{n}$}
\psfrag{h}[c][c]{$w_1$}
\psfrag{i}[c][c]{$w_2$}
\psfrag{j}[c][c]{$w_3$}
\psfrag{k}[c][c]{$w_4$}
\psfrag{l}[c][c]{$w_{n-2}$}
\psfrag{m}[c][c]{$w_{n-1}$}
\psfrag{n}[c][c]{$w_n$}
\psfrag{o}[c][c]{$e_2$}
\psfrag{p}[c][c]{$e_1$}
\epsfig{file=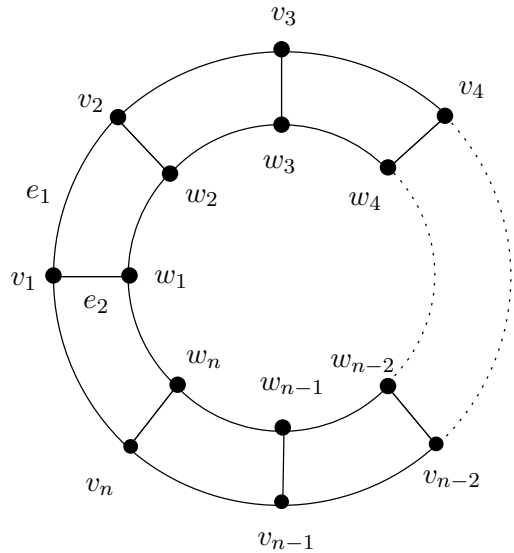}
\caption{\label{m-admissable} The graph $C_n \times K_2$}
\end{center}
\end{figure}
\begin{enumerate}
\item We have $|E(C_n \times K_2)| = 3n$
and $|V(C_n \times K_2)| = 2n$.
Since $C_n \times K_2$ is a cubic graph, we get
$t(L) \leq 2n +1$. Hence, for $n \geq 3$,
$t(L) < |E(C_n \times K_2)|$.
\item By symmetry it is sufficient to consider
the two trails beginning $v_1,e_1, \ldots$ and 
$v_1,e_2,\ldots$ (See Figure~\ref{m-admissable}). The trail
$L_1 : v_1, e_1,v_2,v_3,\ldots,v_{n},v_1, w_1, w_2,\ldots,w_n,w_1$
spans $C_n \times K_2$ and $|E(L)| = 2n+1 = t(L)$.

The trail 
$L_2 : v_1, e_2,w_1,w_2,w_3,\ldots,w_n,v_n,v_1,v_2,\ldots,v_{n-1},v_n$
also spans $C_n \times K_2$ and 
$|E(L_2)| = 2n -2 +3 = 2n+1$. \hspace*{\fill}\qd 
\end{enumerate}\end{proof}
\subsection{First construction}
We need the following type of graph:
\begin{definition}
\label{method1}
A simple graph $G$ is said to be
an \textit{I-type} graph if 
\index{I-type graph}
it has a distinguished set $D =\{a, b, c\}$
of three vertices such that
\begin{enumerate}
\item[I-1] $G$ has  no spanning path
with both endvertices in $D$.
\item[I-2] For any vertex $v \in G$ there is a path
$P$ with endvertex $v$ and an endvertex in $D$,
and the remaining pair of vertices 
in $D$ are joined by
a path $Q$, disjoint from $P$, such that 
$P$ and $Q$ together span $G$.
(The cases where $P = a$ or $P = b$
or $P = c$ are included.)
\end{enumerate}
\end{definition}

An obvious example of an $I$-type graph 
is the net (ie. $K_3$ with a pendant
leaf attached to each vertex, shown in Figure~\ref{net}),
\index{net}
where we take the distinguished vertices
to be the vertices of degree one.
\begin{figure}[H]
\begin{center}
\epsfig{file=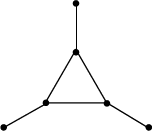}\\
\caption{\label{net} The net}
\end{center}
\end{figure}
We give another example in Lemma~\ref{pete}. 
\begin{figure}[H]
\begin{center}
\psfrag{x}[c][c]{a}
\psfrag{y}[c][c]{b}
\psfrag{z}[c][c]{c}
\psfrag{d}[c][c]{$v_1$}
\psfrag{e}[c][c]{$v_2$}
\psfrag{f}[c][c]{$v_3$}
\psfrag{g}[c][c]{$v_4$}
\psfrag{h}[c][c]{$v_5$}
\psfrag{i}[c][c]{$v_6$}
\psfrag{j}[c][c]{Petersen Graph}
\psfrag{k}[c][c]{$PG^\odot$}
\psfrag{l}[c][c]{$v$}
\epsfig{file=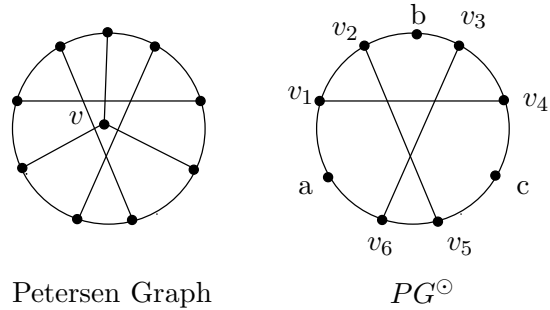}\\
\caption{\label{vd-petersen} Petersen and $PG^\odot$ Graphs}
\end{center}
\end{figure}
Let $PG^\odot$ denote the graph\index{vertex-deleted Petersen graph} 
obtained by deleting a vertex from
the Petersen graph \index{Petersen graph}
(see Figure~\ref{vd-petersen}).
Let $\{a, b, c\} \subset V(PG^\odot)$
be the vertices in $PG^\odot$ of degree
two which were formerly neighbours 
of the deleted vertex. 
\begin{lemma}
\label{pete}
The graph $PG^\odot$ is an $I$-type graph
with  distinguished 
vertex set $D = \{a, b, c\}$.
\end{lemma}
\begin{proof}
\begin{enumerate}
\item[I-1] Suppose that $a, S, b$ is a 
spanning path in $PG^\odot$, and the vertex
deleted from the Petersen graph to get $PG^\odot$
was $v$. Then $v,a,S,b$ is a hamiltonian cycle 
in the Petersen graph, which contradicts the
fact that the Petersen graph is not hamiltonian.  
\item[I-2] By symmetry we need consider only the
distinguished vertex $a$ and the neighbour 
$v_1$ of $a$. In the first case we let
$P = a$
and let  $Q$ be the path
$b$, $v_3$, $v_6$, $v_5$, $v_2$, $v_1$, $v_4$, $c$.
In the second case we let $P_1$ be
the path $v_1, a$ and we let $Q_1$ be the path
$b$, $v_2$, $v_5$, $v_6$, $v_3$, $v_4$, $c$. \hspace*{\fill}\qd
\end{enumerate}
\end{proof}
Further interesting
properties of $PG^\odot$ can 
be found in \cite{sk84}.  

Let $L$ be a cubic, admissible multigraph 
of order $k$.
Let $F$ be a graph obtained by 
inflating each vertex $v_i \in V(L)$
with an $I$-type graph $G_i$, $(i=1,2,\ldots, k)$.
That is, we delete 
the vertex $v_i$,
adding a copy of graph $G_i$ 
in its place, and joining
the three former neighbours of the 
deleted vertex to the three 
distinguished vertices of the copy of $G_i$
by using the three edges
which were incident with $v_i$. 
This process is carried out 
sequentially on all the vertices of $L$.
We will say that $F$ is
obtained by inflating $L$
with $I$-type graphs.
For example, the inflation of the \index{inflation of a graph}
admissible multigraph shown in Figure~\ref{admissable}
with I-type graphs $G_1$, $G_2$,
$G_3$ and $G_4$ is illustrated in
Figure~\ref{inflation}.
\begin{figure}[h]
\begin{center}
\psfrag{a}[c][c]{$a_1$}
\psfrag{b}[c][c]{$b_1$}
\psfrag{c}[c][c]{$c_1$}
\psfrag{d}[c][c]{$G_1$}
\psfrag{e}[c][c]{$a_2$}
\psfrag{f}[c][c]{$b_2$}
\psfrag{g}[c][c]{$G_2$}
\psfrag{h}[c][c]{$c_2$}
\psfrag{l}[c][c]{$G_3$}
\psfrag{i}[c][c]{$b_3$}
\psfrag{j}[c][c]{$a_3$}
\psfrag{k}[c][c]{$c_3$}
\psfrag{p}[c][c]{$G_4$}
\psfrag{m}[c][c]{$a_4$}
\psfrag{n}[c][c]{$c_4$}
\psfrag{o}[c][c]{$b_4$}
\epsfig{file=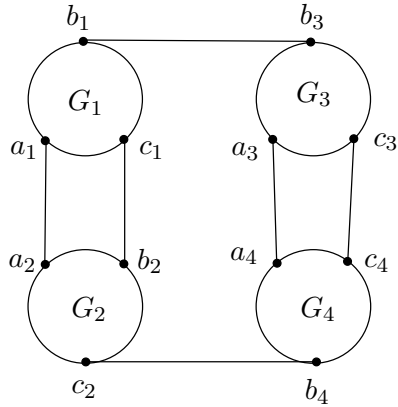}\\
\caption{\label{inflation} Inflation of a multigraph with 
$I$-type graphs.}
\end{center}
\end{figure}

The construction of CND graphs using 
$I$-type graphs is explained 
in  Theorem~\ref{method1-th}.
\begin{theorem}
\label{method1-th}
Let $L$ be a cubic, admissible multigraph
of order $k$. Let $F$ be obtained by 
inflating $L$ with
$I$-type graphs $G_i$, $i=1,2,3,\ldots,k$.
Then $F$ 
is a CND graph with detour order
\[
2-k+\sum_{i=1}^k|V(G_i)|.
\]
\end{theorem}
\begin{proof}
By Condition I-1, if a path $W$ in $F$ contains 
vertices from each copy
of $G_i$, then 
at least $k-2$ vertices of $F$ are not in $W$,
since at least one vertex must be omitted
from every copy of $G_i$ which 
does not contain an endvertex
of $W$. Hence
\[
\tau(F) \leq 2-k+\sum_{i=1}^k|V(G_i)|
\]
and $F$ is
nontraceable.  
Next we show that each vertex 
$v$ of $F$ is an endvertex
of a path of order 
\[
2-k+\sum_{i=1}^k|V(G_i)|.
\]
Let $v \in V(F)$, and let $G_0$ 
be the $I$-type subgraph
of $F$ containing $v$. 
Then there is a path $P$ from
$v$ to one of the three distinguished
vertices of $G_0$  (call it $a_0$) 
such that the two remaining distinguished
vertices are joined by a path $Q$,
disjoint from $P$, where $P$ and $Q$
together span $G_0$. Let $e_1$
be the edge incident with $a_0$ whose
other endvertex is not in $G_0$.
Let $v_0$ be the vertex in $L$ corresponding
to $G_0$. Then, by A-2,  $L$ has a longest trail
$T$ starting $v_0, e_1, \ldots,$
which spans $L$. Now each of
the three edges of $L$ incident with
$v_0$ lies in $T$, otherwise we can
add a missing edge to $T$ to obtain a trail
in $L$ that is longer than $T$.
Thus $v_0$ occurs exactly twice
in $T$. Let $v_l$ be the last
vertex of $T$. Then $v_l$ also occurs
exactly twice in $T$, while
every vertex in $V(L) \setminus \{v_0, v_l\}$
occurs exactly once in $T$. We can now
construct a path $P$ in $F$ starting
at $v$ that exits from $G_0$ at $a_0$
and then moves through the 
subgraphs $G_i$ in accordance with the
trail $T$. Such a path will visit
each of $G_0$ and $G_l$ 
(the $I$-type graph corresponding 
to $v_l$) twice and all the
other $G_i$ only once each.
Thus we can choose $P$ so that it
contains all the vertices of $G_0$
and all those of $G_l$,
as well as all but one vertex of
each of the remaining $G_i$.
Hence $|V(P)| = 2-k+\sum_{i=1}^k|V(G_i)|$.
Since $\tau(F) \leq  2-k+\sum_{i=1}^k|V(G_i)|$,
and $v \in F$ was
chosen arbitrarily, it follows that $F$ is a
CND graph. \hspace*{\fill}\qd
\end{proof}

\begin{remark}
\label{remark1}
By Lemma~\ref{cubic-admissible}
the cubic graphs $C_n \times K_2$, $n \geq 3$
are all admissible graphs. 
Hence we get an infinite family of cubic
CND graphs (all of girth 5) by letting 
\index{cubic CND graphs}
the admissible
graphs be $C_n \times K_2$, $n \geq 3$,
and taking all the $I$-type graphs to 
be copies of  $PG^\odot$. 
If instead we use $K_4$ 
or the multigraph $C_2 \times K_2$ 
shown in Figure~\ref{admissable} for 
the cubic admissible  multigraph,  
we get two of the four smallest presently
known cubic CND graphs, each of
order 36. In Figure~\ref{small-cubic}
we illustrate the case where we use $C_2 \times K_2$.
\begin{figure}[H]
\begin{center}
\epsfig{file=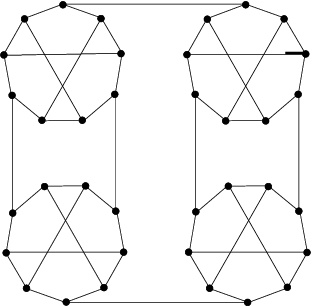}
\caption{\label{small-cubic} Cubic CND graph based on $C_2 \times K_2$}
\end{center}
\end{figure}
In Figure ~\ref{small-cubic2} we show the graph obtained
by using $K_4$. 
\begin{figure}[H]
\begin{center}
\epsfig{file=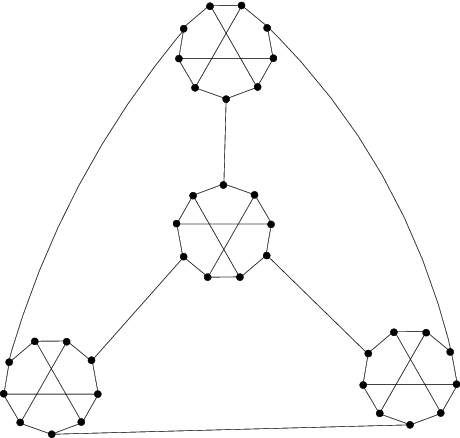}
\caption{\label{small-cubic2} Cubic CND graph based on $K_4$}
\end{center}
\end{figure}
We will
describe the other two in Subsection~\ref{sec-const}. 
None of the entries in 
Dobrynin and 
Me{l\kern-.15em'\kern-.09em}nikov's~\cite{dome02}
list of detour sequences of connected, cubic graphs
with at most 20 vertices are detour sequences 
of CND graphs. Hence a cubic CND graph of least order
has more than 20 vertices and at most 36 vertices. 
\end{remark}

\subsection{Second construction}
\label{sec-const}
Here we need a new type of graph,
which we call \textit{homogeneously
connected from two vertices}, 
\index{homogeneously connected from two\\ vertices}
or briefly
a HCTV graph. \index{HCTV graph}
The definition is:

\begin{definition}
\label{HCTV}
A simple graph $M$ is called
\textbf{homogeneously connected 
from two vertices} 
if it is $K_1$, or it 
has two vertices $x$, $y$ such that 
each vertex of $M$ is an initial 
vertex of a hamiltonian path 
with the other endvertex 
either $x$ or $y$. The 
vertices $x$, $y$ are 
called \textbf{anchors} of $M$.\index{anchors in a HCTV graph}
\end{definition}

One can easily see that if a graph $G$ is hamiltonian 
connected then it is a HCTV graph. 
On the other hand, if it 
is a HCTV graph, then it is homogeneously traceable. 
Moreover, from the definition 
it follows that the anchors of a HCTV
graph $G$ are 
endvertices of a hamiltonian path in $G$. 
Examples of 
HCTV graphs are the complete graphs 
$K_n$, $n \ge 1$, 
and the complete  balanced bipartite graphs 
$K_{n,n}$, $n\ge 1$.
\begin{figure}[H]
\begin{center}
\psfrag{a}[c][c]{$C_3 \times K_2$}
\psfrag{b}[c][c]{$F$}
\epsfig{file=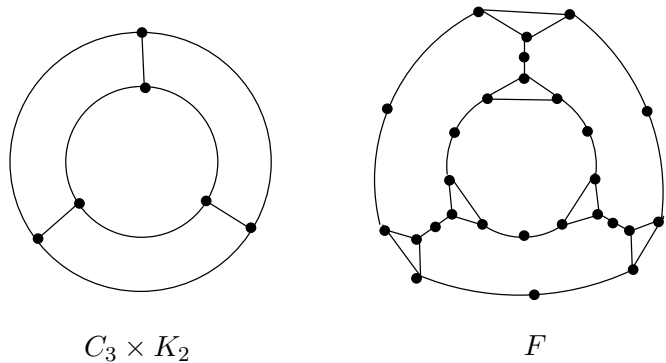}\\
\caption{\label{infl-insertion} An illustration of the 
inflation and insertion operations} 
\end{center}
\end{figure}
Our next construction for CND graphs is
described in the following theorem.
We will need a similar inflation
operation on graphs or multigraphs
to that described in the first construction.
Starting with any graph or multigraph
we can \textit{inflate}
a vertex $v$
with a complete graph $K_n$, where 
$n \geq \deg(v)$. In other words,
we delete the vertex $v$ 
and add a copy of $K_n$
in its place, joining
the former neighbours of the 
deleted vertex to distinct vertices
of $K_n$ (which we will
call the \textit{inflation}\index{inflation vertices of $K_n$} 
vertices of $K_n$) 
by using the edges
which were incident 
with $v$.  
We also need the operation of
\textit{inserting} a HCTV graph $M$
\index{inserting a HCTV graph}
in an edge $e$ of 
any graph or multigraph. This
is done by deleting $e$,
and then joining the
two vertices which were
incident with $e$ to the anchors
of $M$ by a matching if $M \neq K_1$,
or to $M$ if $M= K_1$.  
This is illustrated in Figure~\ref{infl-insertion}
for the case where we inflate $C_3 \times K_2$
with copies of $K_3$ and insert $K_1$ in
each edge of $C_3 \times K_2$ to
get the graph $F$.

\begin{theorem}
\label{method2}
Let $L$ be an admissible multigraph
and $k = |E(L)|$. Let $M_i$, $i=1,2,3,\ldots,k$
be HCTV graphs, all of the same order $m$.
Let $F$ be obtained by first
inflating each vertex $v$ of  $L$
with  a copy of $K_n$
for some $n \geq \deg(v)$, 
and then inserting 
a copy of a HCTV graph $M_i$ in 
the edge $e_i$ of $L$, 
$i = 1,2, \ldots, k$. 
Then $F$ is a CND graph.
\end{theorem}
\begin{proof}
Suppose that the longest trail in  $L$
has length $l$. Then no path
in $F$ can contain vertices from more than $l$ copies
of the HCTV graphs, since otherwise we can use the
path in an obvious way to construct a 
trail in $L$ with length
greater than $l$. 
Hence $\tau(F) \leq |V(F)| - (k-l)m$.
Next we show that each vertex of $F$ is an endvertex 
of a path of order $|V(F)| - (k-l)m$.
Let $v \in V(F)$. We have two cases
to consider:
\begin{enumerate}
\item Suppose that
$v$ belongs to one
of the complete subgraphs of $F$, say $K_v$,
obtained by inflating $L$.
Suppose that edge $e_1$ in $L$ was incident with
the vertex $v_1 \in V(L)$ which was inflated 
to give $K_v$. Let $w$ be the inflation vertex
of $K_v$ corresponding to $e_1$.
Then $L$ has a spanning trail $T$
of length $l$ starting
$v_1, e_1 \ldots$. We can now
construct a path $Q$ starting $v,\ldots,w,e_1\ldots$
that moves through the inflations
and the insertions in accordance
with the trail $T$.
The trail $T$ spans $L$, hence $Q$ contains
vertices from all the 
complete graphs used to inflate $L$.
It is easy to see that we can use the properties of the
complete graphs to ensure that $Q$ passes through all the
vertices of all the complete graphs used in
the inflation process, and the
properties of the HCTV graphs allow us  
to choose $Q$ so that
it contains all the vertices of the 
$l$ HCTV graphs lying
in the path. Hence $|V(Q)| = |V(F)| - (k-l)m$.

\item If $v \in V(M_i)$ for some $i$ we can 
similarly use a longest
trail in $L$  to construct a path in $F$
starting at $v$ and having
order $|V(F)| - (k-l)m$.
\end{enumerate}
Since $v \in V(F)$ was chosen arbitrarily, it follows that
$F$ is a CND graph. \hspace*{\fill}\qd
\end{proof}
\begin{figure}[H]
\begin{center}
\epsfig{file=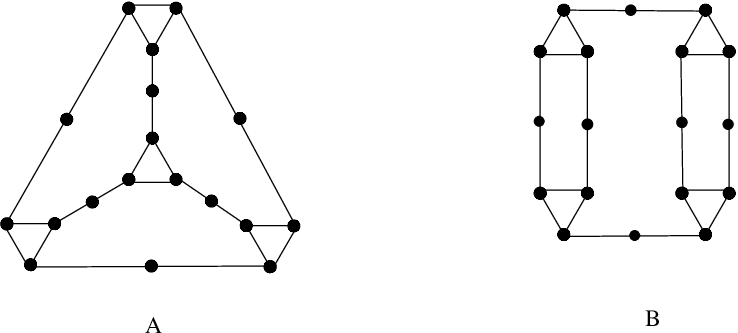}
\caption{\label{smallcnd} Smallest presently known CND graphs}
\end{center}
\end{figure}
\begin{remarks}
\label{remarks}
We get the following results from
this construction:
\begin{enumerate}
\item If we take 
the HCTV graphs to be $K_1$
and inflate each vertex of $L$
with a $K_3$, and choose $L$
to be either the graph in Figure~\ref{admissable} 
or $K_4$, then we get the two smallest 
presently known CND \index{smallest known CND graphs}
graphs, each of order 18 and size 24, 
with detour deficiency one. These two
graphs are shown in Figure~\ref{smallcnd} 
It is proved in
Chapter 3
%\cite{bufrsi} 
that these two
graphs are the smallest 
(with regard to both size and order)
$2$-connected, nontraceable,
claw-free graphs. (A graph is claw-free if
it does not contain $K_{1,3}$ as an 
induced subgraph.) %\index{claw-free graph}
Hence they are the smallest
claw-free, CND graphs.

\item Let $L$ be $C_n \times K_2$,
$n \geq 3$, the HCTV graphs
be $K_2$, and we inflate $L$
with $K_3$. The resulting CND graphs
have order $12n$ and size $15n$, $n \ge 3$.
Hence these CND graphs all realise the
lower bound on the size of CND graphs
given in Theorem~\ref{prop-th}. 
If instead of $K_2$ we use $K_1$ for
the HCTV graphs we get a family
of CND graphs having detour deficiency
$n-1$, $n \geq 3$, and, of course, girth 3.

\item Let $L$ be $K_4$ or the
admissible multigraph shown in Figure~\ref{admissable}, 
the HCTV graphs be $K_2$, and
we inflate $L$ with $K_3$.
The resulting two CND graphs
have order 24 and size 30,
so they also realise the lower
bound on the size of CND graphs. 
Thus we have examples of CND
graphs of order $12n$, for every $n \geq 2$, 
having the least possible number
of edges.

\item In Remark~\ref{remark1}
we gave two examples of cubic CND graphs
of order 36.
We get two more such graphs 
by choosing
$L$ to be either $K_4$ or
the multigraph shown in Figure~\ref{admissable},
the HCTV graphs to be copies of
$K_4 -e$ (ie the graph obtained by
deleting any edge from $K_4$),
and we inflate $L$ with copies
of $K_3$.

\item It is easy to get CND
graphs of all orders greater than 17.
For example, let $L$ be $K_4$, 
the HCTV graphs be $K_1$, and
we inflate $L$ with three
copies of $K_3$ and one
copy of $K_{n+3}$, $(n \geq 0)$.
Then we get CND graphs of
order $18+n$, for $n \geq 0$. 
\end{enumerate}
\end{remarks}
\addtocontents{toc}{\protect\hfill\protect\textit{Page}\vskip0pt} 
\subsection{Third construction}
We need the following definition of another 
type of graph:
\begin{definition}
\label{r-type}
A simple graph $G$ is said to be
an \textbf{R-type} graph if \index{R-type graph}
it has a distinguished set $D =\{a, b, c\}$
of three vertices such that
\begin{enumerate}
\item[R-1] For every pair $u$, $v$ of
vertices in $D$ there is a hamiltonian path
in $G$ with $u$ and $v$ as endvertices.   

\item[R-2]  For any vertex $v \in G$ there is a path
$P$ with endvertex $v$ and an endvertex in $D$,
and the remaining pair of vertices 
in $D$ are joined by
a path $Q$, disjoint from $P$, such that 
$P$ and $Q$ together span $G$.
(The cases where $P = a$ or $P = b$
or $P = c$ are included.)
\end{enumerate}
\end{definition}

Clearly, complete graphs $K_n$, $n \ge 3$, 
are $R$-type graphs, 
where the distinguished vertices can be any three
vertices of $K_n$, and in Figure~\ref{girth4}
we show an $R$-type graph with girth four.
\begin{figure}[H]
\begin{center}
\epsfig{file=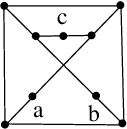}\\
\caption{\label{girth4} An $R$-type graph with girth four.}
\end{center}
\end{figure}

Using Definition~\ref{r-type} we get:
\begin{theorem}
Let $L$ be a cubic, admissible multigraph
of size $k$. Let $M_i$, $i=1,2,3,\ldots,k$
be HCTV graphs, all of the same order $m$.
Let $F$ be obtained by first
inflating all the vertices of $L$
with copies of $R$-type graphs,
and then inserting copies of the HCTV
graphs $M_i$ in each of the
former edges of $L$. 
Then $F$ is a CND graph.
\end{theorem}
\begin{proof}
We omit the details of the proof 
since they are very similar 
to the proof of Theorem~\ref{method2}.
We simply note that we can use
a longest trail in $L$ 
(let its length be $l$) 
to define a path $Q$, starting
at an arbitrary  vertex $v \in V(F)$, which
contains all the vertices of the $R$-type
graphs and omits all the vertices of $k - l$ 
of the HCTV graphs,
and no path in $F$ can contain 
more vertices than $Q$. \hspace*{\fill}\qd
\end{proof}

This construction gives examples of
CND graphs with girth 4, 
\index{CND graphs with girth 4}
if we use
the graph shown in Figure~\ref{girth4} as our R-type
graphs and $K_1$ for HCTV graphs. For
$L$ we can take any cubic, admissible 
multigraph, for example $K_4$.

\subsection{Fourth construction}
Here we will need the concept
of a maximal hypohamiltonian graph.
\index{maximal hypohamiltonian graph}
A graph $G$ is \textit{hypohamiltonian} 
\index{hypohamiltonian graph}
if $G$ is
not hamiltonian, but every vertex deleted subgraph
$G-v$ of $G$ is hamiltonian. A hypohamiltonan
graph $G$ is called \textit{maximal hypohamiltonian}
if $G + e$ is hamiltonian for each 
$e \in E(\overline{G})$.

We now describe a construction which allows
us to use certain maximal hypohamiltonian
graphs to construct CND graphs
with girth 7 and 6.

We begin by defining the types of graphs we
need for the construction.

\begin{definition}
A simple graph $G$ is said to be
a \textbf{$U$-type graph} 
\index{$U$-type graph}
if it contains
a  set $D = \{a, b, c\}$ of
three distinguished vertices such that
\begin{enumerate}
\item[U-1] For each pair of vertices
in $D$ there is a path in $G$ 
having those two vertices as endvertices 
and containing all vertices
of $G$ except the other vertex in $D$.
\item[U-2] There is no spanning path of $G$ with both
endvertices in $D$.
\item[U-3] $G$ is traceable from each vertex in $D$.
\item[U-4] If $v \in V(G)$,
and $v \notin \{a, b, c\}$,
then there is a path from $v$ to a vertex in $\{a, b, c\}$ which
spans $G$.
\end{enumerate}
\end{definition}

A maximal hypohamiltonian graph
which has at least one vertex of degree
three can be used to construct a  
$U$-type graph. This is proved in the next
theorem.

\begin{theorem}
Let $H$ be a maximal hypohamiltonian graph
containing a vertex $w$ of degree $3$. Let
the vertices $a$, $b$, $c$ of $H$ be adjacent to
$w$. Let $G = H -w$. Then $G$ is a
$U$-type graph with $D=\{a, b, c\}$.
\end{theorem}
\begin{proof}
\begin{enumerate}
\item $H-a$ is hamiltonian, and $b$, $c$
are the only vertices of $H-a$ adjacent to $w$.
Hence we have a hamiltonian cycle $C$ in 
$H-a$ containing $w$ and such that $b,w,c$ is
a path on $C$. Therefore the path
on $C$ from $b$ to $c$ which omits $w$ contains all
the vertices of $H-w$ except $a$. Similarly for
the other vertices in $D$.
Thus U-1 is satisfied.
\item If there was a spanning path in $H-w$
with endvertices in $D$ then $H$ would
be hamiltonian, hence U-2 is satisfied.

\item  U-3 follows since $H-w$ is hamiltonian.
\item Since $v$ is not adjacent to
$w$, and $H$ is maximal nonhamiltonian, it
follows that $H+vw$ is hamiltonian. A hamiltonian
cycle $C$ in $H+vw$ must contain the edge $vw$,
and one of $a$, $b$, $c$ must be adjacent to $w$
on $C$. Suppose that $a$ is adjacent to $w$ on $C$.
Then the path on $C$ from $v$ to $a$ which
omits $w$ contains all the vertices of $H-w$.
Thus U-4 is satisfied. \hspace*{\fill}\qd
\end{enumerate}
\end{proof}

We also need the following type of multigraph:
\begin{definition}
A \textbf{presentable} multigraph $S$ is
\index{presentable multigraph}
a  multigraph such that
\begin{enumerate}
\item[S-1] $S$ is cubic.
\item[S-2] There is a longest trail in $S$ beginning
$v,e,\ldots$ from each vertex $v$ and each edge
$e$ incident with $v$, and this longest trail spans $S$.
\item[S-3] There is a hamiltonian path
beginning $v,e,\ldots$ from each vertex
$v$ and edge $e$ incident with $v$.
\end{enumerate}
\end{definition}

Some examples of presentable multigraphs are
$K_4$, $C_n \times K_2$, the Petersen graph
and the multigraph shown in Figure~\ref{admissable}.

The construction of CND graphs 
using presentable multigraphs
and $U$-type graphs 
is described in  Theorem~\ref{p-th}.

\begin{theorem}
\label{p-th}
Let $S$ be a presentable
multigraph
of order $k$. Let $F$ be obtained by 
inflating $S$ with
$U$-type graphs $G_i$, $i=1,2,3,\ldots,k$.
Then $F$ 
is a CND graph.
\end{theorem}
\begin{proof}
If a path $P$ in $F$ contains vertices from each 
$G_i$ then at least $k-2$ vertices
of $F$ are not in $P$, since,
by condition U-2,
at least one vertex must be omitted from
every $G_i$ which does not contain an endvertex
of $P$. Hence 
\[
\tau(F) \leq \sum_{i=1}^k |V(G_i)|-k+2, 
\]
and therefore $F$ is nontraceable.  
Next we show that each vertex of $F$ is an endvertex
of a path of order 
$\sum_{i=1}^k |V(G_i)|-k+2$. Let $v \in V(F)$.
We have two cases to consider.
\begin{enumerate}
\item Suppose that $v$ is a distinguished vertex
of some $U$-type subgraph $G_0$ of $F$. 
Then, following the same procedure described
in the proof of Theorem~\ref{method1-th},
we can construct a path $Q$ in $F$ starting at $v$
with $|V(Q)| = \sum_{i=1}^k |V(G_i)|-k+2$.

\item Suppose that $v \in V(G_0)$,
for some $U$-type subgraph $G_0$ of $F$,
but that $v$ is not a distinguished vertex
of $G_0$. Then there is a path $P$ in $G_0$
from $v$ to some distinguished vertex 
of $G_0$, say $a_0$, which
spans $G_0$. Let $e_1$ be the edge
incident with $a_0$ whose other
endvertex is not in $G_0$.
Let $v_0$ be the vertex in $S$
corresponding to $G_0$.
Then $S$ has a hamiltonian path $Q$
starting $v_0,e_1\ldots$. We can now
construct a path $P$ in $F$
starting at $v$ that exits from $G_0$
at $a_0$ and then moves through the
$G_i$ in accordance with the path $Q$.
By properties U-3, U-4 we can choose
the path $P$ so that it contains all
the vertices of the first and
last $U$-type subgraphs in the path,
and (by U-1 and U-2) omits exactly one vertex from
each of the other $U$-type graphs.
Hence  $|V(Q)| = \sum_{i=1}^k |V(G_i)|-k+2$. 
\end{enumerate}
Since 
$\tau(F) \leq  \sum_{i=1}^k |V(G_i)|-k+2$ 
and $v \in V(F)$ was
chosen arbitrarily it follows that $F$ is a
CND graph. \hspace*{\fill}\qd \end{proof}

The Coxeter graph
(see Figure~\ref{cox})
has girth 7, is cubic, and is  maximal
hypohamiltonian, as shown by Stacho~\cite{stacho}.

Thus we get an infinite family of CND
graphs of girth 7 if we use $C_n \times K_2$, 
$n \geq 3$,
for presentable multigraphs 
and copies of a vertex deleted Coxeter graph 
for $U$-type graphs.
Clark and Entringer~\cite{isaacs}
showed that the 
Isaacs snarks $J_k$, $k \geq 7$, $k$ odd,\index{Isaacs snarks}
are maximal hypohamiltonian. Since these snarks,
and their vertex deleted
counterparts, all have girth 6 they can 
similarly be used to inflate
the vertices of $C_n \times K_2$ to give
infinite families of CND graphs
of girth 6. The bipartite CND graphs constructed in 
Section~\ref{bipartite} also have girth 6.
\begin{figure}[H]
\begin{center}
\epsfig{file=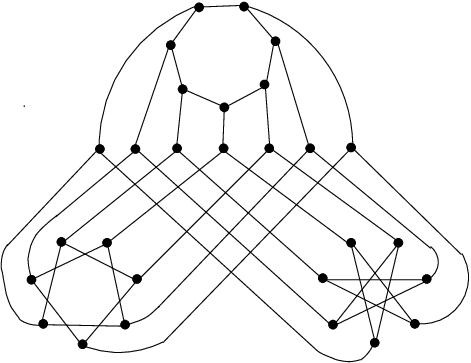}\\
\caption{\label{cox} Coxeter graph}
\index{Coxeter graph}
\end{center}
\end{figure}

\section{A construction for CND graphs
with prescribed chromatic number}
\label{bipartite}
The constructions in the previous section
did not enable us to find bipartite
CND graphs. In this section we describe 
a construction that
allows us to construct CND graphs
with prescribed chromatic numbers.
As before, our construction will
be based on inflations and insertions,
but here we will inflate the vertices
of a cycle $C_n$ with graphs
which satisfy the conditions described in the next 
definition.  

\begin{definition}
\label{dc}
Let $G$ be a simple graph of order $k \geq 3$
and $m \leq k-2$ be a positive integer.
We say that $G$ is an 
\textbf{inflator} graph 
\index{inflator graph}
with \textbf{drop}\index{drop of an inflator graph}
$m$ if it has
a set of two distinguished vertices
$\{a, b\}$ such that:
\begin{enumerate}
\item[D-1] $\tau_G(a,b) = k-m$.
\item[D-2] Each distinguished vertex is an
initial vertex of a path in $G$ containing
all the vertices of $G$ except the other
distinguished vertex.
\item[D-3] $G$ is traceable from each
distinguished vertex.
\item[D-4] Let $v \in V(G) \setminus \{a, b\}$.
Then there is a path $P$ from $v$ to
a distinguished vertex, and a path
$Q$ in $G$ with initial
vertex the remaining distinguished vertex,
such that $V(P) \cap V(Q) = \emptyset$
and $V(P) \cup V(Q) = V(G)$. The path
$Q$ may consist of a single distinguished
vertex.
\end{enumerate}
\end{definition}

Theorem~\ref{HTGraphs}
shows that inflator graphs
are closely related
to NHHT graphs.

\begin{theorem}
\label{HTGraphs}
Let $m$ and $k$ be positive integers
such that $m \leq k-2$.
Let $G$ be an inflator graph 
of order $k$ and drop $m$, 
with distinguished vertex set $\{a, b\}$. 
Let $H$ be the graph obtained from 
$G$ by adding a new
vertex $x$ and two new edges $ax$ and
$bx$. Then
$H$ is a NHHT graph.
\end{theorem}
\begin{proof}
Suppose that $H$ is hamiltonian. 
Then, since the vertex $x$ is of 
degree $2$, any hamiltonian cycle must 
contain the edges $ax$ and $bx$. But any path 
joining $a$ and $b$ in $G$ has order at most 
\[
\tau_G(a,b)=k-m 
< k=|V(G)|,
\]
and therefore $H$ cannot 
contain a hamiltonian cycle.

Now we prove that any vertex of the graph $G$ is an initial 
vertex of a hamiltonian path in $H$. Firstly, consider the 
vertex $a$. According to condition D-2 of Definition~\ref{dc} 
there exists a path $P$ in $G$ starting at
$b$ and containing 
all the vertices of $G$ except $a$.
Therefore the path $a,x,P$
is a hamiltonian path in 
$H$ starting at $a$. 
Similarly, $b$ is also the initial 
vertex of a hamiltonian path in $H$.

Consider now a vertex 
$v \in V(G) \setminus \{a, b\}$. 
By condition  
D-4 of Definition~\ref{dc}, 
we have a 
path $P$ from 
$v$ to a vertex in $\{a, b\}$, say $a$,
and a path $Q$, disjoint from $P$, 
starting at $b$ such that
$V(P) \cup V(Q) = V(G)$.
Then the path 
$P,x,Q$ is a hamiltonian path in 
$H$ starting at $v$.

Finally, consider the vertex $x$. 
According to the 
condition D-3 of Definition~\ref{dc} 
there exists a 
hamiltonian path in $G$ with 
initial vertex $a$. This path, together with 
the edge $xa$, forms a 
hamiltonian path in $H$ with 
initial vertex $x$. \hspace*{\fill} \qd
\end{proof}

One can easily verify that the graph 
$G^\odot$ shown in Figure~\ref{inflator}
is an inflator graph of
order $8$ and drop one.
\addtocontents{lof}{\protect\hfill\protect\textit{Page}\vskip0pt}
\begin{figure}[H]
\begin{center}
\epsfig{file=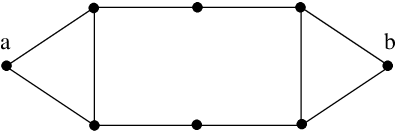}\\
\caption{\label{inflator} The smallest inflator graph.}
\index{smallest inflator graph}
\end{center}
\end{figure}
\begin{corollary}
\label{smallest}
The graph  $G^\odot$ is the smallest 
inflator graph (with respect to both 
order and size).
\end{corollary}
\begin{proof} 
Skupie\'n in \cite{sk84} proved that the  
smallest NHHT graph 
has order $9$ and size $12$. It can be obtained from 
$G^\odot$ by adding a new vertex $x$
and two new edges $ax$ and $bx$.  
Therefore, applying Theorem~\ref{HTGraphs}, we obtain 
the desired result. \hspace*{\fill}\qd
\end{proof}

In the next theorem we use the operation
of inflating each vertex of a cycle $C_n$
with a copy of an inflator graph. Explicitly,
we replace each vertex $v$ of $C_n$ with a
copy of an inflator graph $G$ by deleting $v$
and joining the two former neighbours of
$v$ in $C_n$ to the two distinguished vertices
of $G$ by a matching.  

We will also use HCTV graphs
(which we introduced in Definition~\ref{HCTV})
in the next theorem.  Also, if $M_i$,
$i=1,2,\ldots,n$ are HCTV graphs
and $G_i$, $i=1,2,\ldots,n$ are
inflator graphs, we will denote by
${\cal F}(G_1,G_2,\ldots,G_n,M_1,M_2,\ldots,M_n)$
the graph obtained by first
inflating each vertex $v_i$ of 
a cycle $C_n$, $n \geq 2$, with
$G_i$, and then inserting
$M_i$ in each former edge $e_i$ of $C_n$,
$i = 1,2,3,\ldots,n$.  
This construction is illustrated
in  Figure~\ref{inflation-insertion}.
If $G_1=G_2=\cdots=G_n =G$ and
$M_1=M_2= \cdots = M_n =M$, then
we simply write ${\cal F}(n,G,M)$.
\begin{figure}[H]
\begin{center}
\psfrag{a}[c][c]{$G_1$} \psfrag{b}[c][c]{$M_1$}
\psfrag{c}[c][c]{$G_2$} \psfrag{d}[c][c]{$M_2$}
\psfrag{e}[c][c]{$G_3$} \psfrag{f}[c][c]{$G_n$}
\psfrag{g}[c][c]{$M_n$}
\epsfig{file=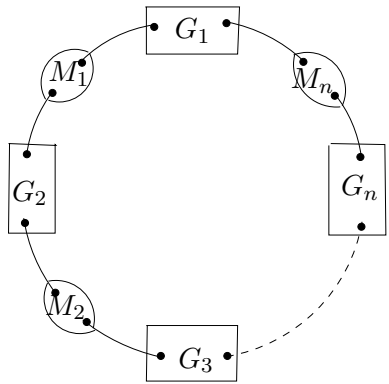}\\
\caption{\label{inflation-insertion} The construction of
${\cal F}(G_1,G_2,\dots,G_n,M_1,M_2,\dots,M_n)$.}
\end{center}
\end{figure}

\begin{theorem} 
\label{const_lema} 
Let $m$ and
$k_1,k_2,\dots,k_n$ be positive integers 
satisfying $m +2 \le\ min\{k_1,k_2,\dots,k_n\}$. Let
$G_i$, $i=1,2,\ldots, n$,  be inflator graphs
of order $k_i$, $i =1,2,\ldots n$ respectively.
Suppose that each $G_i$ has the same drop
$m$.
Let $M_i$, $i=1,2,\dots,n$ be HCTV graphs, each of 
order at least $m$.
Then the graph
\[
F = {\cal F}(G_1,G_2,\ldots,G_n,M_1,M_2,\ldots,M_n)
\]
is a CND graph.  
\end{theorem}
\begin{proof}
It is not difficult to see that if a path $P$ contains 
at least one vertex from each $M_i$, then it goes 
through at least $n-1$ graphs from $G_1,\dots, G_n$ and, 
according to the condition D-1
of Definition~\ref{dc}, 
at least $(n-1)m$ vertices of $G_1,\dots,G_n$
are not contained in $P$.
Hence in this case the order of $P$ is at most 
\[
\sum_{i=1}^n k_i+\sum_{i=1}^n |V(M_i)|-(n-1)m.
\]
Similarly, if a path $P$ contains no  vertices 
from exactly one HCTV graph, say $M_r$,
then $P$ goes through at least $n-2$ 
inflator graphs $G_i$. 
Therefore in this case its order
is at most 
\[
\sum_{i=1}^{n} k_i+\sum_{i=1}^{r-1} |V(M_i)|
+\sum_{i=r+1}^{n}|V(M_i)|-(n-2)m.
\]
Since $|V(M_r)| \ge m$ we get
\[
\sum_{i=1}^{n} k_i+\sum_{i=1}^{r-1} |V(M_i)|
+\sum_{i=r+1}^{n}|V(M_i)|-(n-2)m
\le
\sum_{i=1}^n k_i+\sum_{i=1}^n |V(M_i)|-(n-1)m.
\]
Therefore 
\[
\tau(F)
\le \sum_{i=1}^n k_i+\sum_{i=1}^n |V(M_i)|-(n-1)m.
\]
Next we show that each vertex of $F$
is an endvertex of a path of order
\[
\sum_{i=1}^n k_i+\sum_{i=1}^n |V(M_i)|-(n-1)m.
\]
We consider three cases.  
\begin{enumerate}
\item[(i)]
Firstly, let $v$ be a vertex
of one of the HCTV subgraphs of $F$,
say $M_i$. Let $P$ be a path starting
at $v$, then passing through all
the other vertices of $M_i$, and exiting
$M_i$ via an anchor vertex of $M_i$.
$P$ then follows  the underlying cycle $C_n$
of $F$, passing alternately through all the 
inflator subgraphs of $F$
and all the other HCTV subgraphs of $F$, 
ending in an
inflator subgraph, say $G_j$, where
$j=i$ or $j=(i+1)\mod n$.  
By property D-3 and D-1 of Definition~\ref{dc},
and since $P$ enters and leaves the 
other HCTV subgraphs via anchor vertices, 
it is easy
to see that we can choose $P$ to contain
all the vertices of the HCTV subgraphs
and all the vertices of $G_j$, while
omitting exactly $m$ vertices from each of the
other $n-1$ inflator subgraphs that $P$
passes through. Hence
\[
|V(P)| = 
\sum_{i=1}^n k_i+\sum_{i=1}^n |V(M_i)|-(n-1)m.
\]

\item[(ii)] Now suppose  that $v$ is a distinguished
vertex of some inflator subgraph, say $G_i$, of $F$.
Consider a path starting at $v$ which exits
$G_i$ from $v$ and then follows the underlying
cycle $C_n$ of $F$, passing through all the HCTV
and other inflator subgraphs, and finally entering
$G_i$ via the (so far) unused distinguished vertex of $G_i$.  
The properties of HCTV graphs and property D-2
of Definition~\ref{dc} imply that we can choose
$P$ such that 
\[
|V(P)| = 
\sum_{i=1}^n k_i+\sum_{i=1}^n |V(M_i)|-(n-1)m.
\]

\item[(iii)] Now suppose that $v$ is a vertex
of some inflator subgraph $G_i$ of $G$,
but is not a distinguished vertex of $G_i$.
By condition D-4 of 
Definition~\ref{dc} there is a path
$P_1$ from $v$ to a distinguished vertex
of $G_i$, say $b_i$, and a path $Q$,
disjoint from $P_1$, starting at the other distinguished
vertex of $G_i$, say $a_i$, such that
$V(P_1) \cup V(Q) = V(G_i)$.
We can then construct
a path $P$ starting with the subpath
$P_1$, then following the underlying
cycle through all the HCTV subgraphs 
and all the inflator subgraphs,
finally ending with the
subpath $Q$.
Therefore we can choose $P$ to contain
all the vertices of $G_i$, and all the
vertices of the HCTV subgraphs, while omitting exactly
$m$ vertices from each of the other
$n-1$ inflator graphs. Hence
\[
|V(P)| = 
\sum_{i=1}^n k_i+\sum_{i=1}^n |V(M_i)|-(n-1)m.
\]
\end{enumerate}
Since
\[
\tau(F) \leq 
\sum_{i=1}^n k_i+\sum_{i=1}^n |V(M_i)|-(n-1)m
\]
it follows that $F$ is a CND graph. \hspace*{\fill}\qd
\end{proof}

Note that ${\cal F}(2, G^\odot, K_1)$
gives an alternative construction 
for one of the two smallest, claw-free CND graphs
constructed in Subsection~\ref{sec-const}.
See Figure~\ref{smallcnd}. 

In order to construct CND graphs with prescribed
chromatic number we introduce the two graphs
$B^\odot$ and $C^\odot$ shown in Figure~\ref{inflator-drop2}.
It is not difficult to verify that
both these graphs are inflator graphs
with drop two,\index{inflator graphs with drop two}
where the 
distinguished vertices of $B^\odot$ 
are the vertices labelled $a$ and $b$
in Figure~\ref{inflator-drop2}, 
and  the distinguished
vertices of $C^\odot$ are the vertices 
labelled $c$ and $d$. Clearly, both
$B^\odot$ and $C^\odot$ are bipartite
graphs. 
\begin{figure}[H]
\begin{center}
\psfrag{e}[c][c]{$B^\odot$}
\psfrag{g}[c][c]{a}
\psfrag{h}[c][c]{b}
\psfrag{f}[c][c]{$C^\odot$}
\psfrag{i}[c][c]{c}
\psfrag{j}[c][c]{d}
\epsfig{file=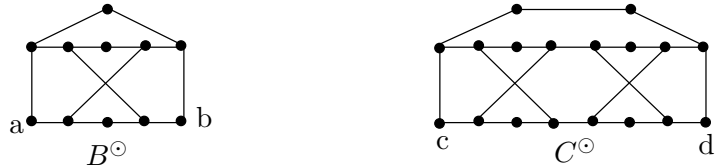}
\caption{\label{inflator-drop2} Two bipartite inflator graphs.}
\end{center}
\end{figure}

It follows 
from Theorem~\ref{const_lema}
that ${\cal F}(2, B^\odot, K_2)$ is a CND graph,
and, since $B^\odot$ is bipartite, it 
is easy to see that  
${\cal F}(2, B^\odot, K_2)$ is bipartite.
This is the smallest bipartite, CND graph that
we know of at present.\index{bipartite CND graph}

Using the graph $C^\odot$ we can prove
the following more general result.

\begin{theorem}
\label{det_bip_def}
For each positive integer $d\ge 1$ 
there exists a bipartite,
CND graph with detour deficiency at least $d$.
\end{theorem}
\begin{proof}
Let $n=\left\lceil \frac{d}{2}\right\rceil+1$. Then the graph
${\cal F}(n,C^\odot,K_2)$ is a 
bipartite, CND graph, and its 
detour deficiency is $2(n-1)\ge d$. \hspace*{\fill}\qd
\end{proof}

We remark that the graph $B^\odot$ can also be 
used for the construction of an infinite sequence of 
bipartite, CND graphs. However,  since 
the vertices $a$ and $b$ of $B^\odot$ have the 
same colour in any 2-colouring of $V(B^\odot)$, 
the graph ${\cal F}(n,B^\odot,K_2)$ 
is bipartite if and only if $n$ is even.

Using the graph  $G^\odot$, 
shown in Figure~\ref{inflator},
we also have

\begin{theorem} \label{main}
Let $c$ and $d$ be  positive integers, $c\ge 3$, 
$d\ge 1$. Then
there exists a CND graph with chromatic number $c$
and detour deficiency $d$.
\end{theorem}
\begin{proof}
Consider the graph 
${\cal F}(d+1,G^\odot,K_c)$. 
Since $G^\odot$ 
is $3$-colourable, the  graph 
${\cal F}(d+1, G^\odot,K_c)$ evidently 
has chromatic number $c$. 
By Theorem~\ref{const_lema}, 
the graph ${\cal F}(d+1, G^\odot,K_c)$ 
is a CND graph, and its 
detour deficiency is exactly $((d+1)-1)=d$. \hspace*{\fill}\qd
\end{proof}

If $G$ is a bipartite, CND graph we can 
improve the upper bound for $\Delta(G)$ given 
in Theorem~\ref{prop-th} as follows:
\begin{theorem}\label{deg-bipartite}
If $G$ is a bipartite CND graph then
\[
\Delta(G) \leq \left\lceil \frac{\tau(G) -2}{2} \right \rceil.
\]
\end{theorem}
\begin{proof}
Let $v_1 \in V(G)$ and let $P: v_1,v_2,\ldots,v_p$ be a detour in
$G$. Then all the neighbours of $v_1$ belong
to $P$ and $v_1v_p \notin E(G)$. A $2$-colouring
of $P$ gives
\[
\deg_G(v_1) \leq \left\lceil\frac{\tau(G) -2}{2}\right\rceil.
\]
Since $v_1$ was chosen arbitrarily we obtain
\begin{align}
\Delta(G) \leq 
\left\lceil\frac{\tau(G) -2}{2}\right\rceil. \tag*{\qd}
\end{align}
\end{proof}
\section{Open problems}
\label{open}

We conclude this chapter with a number 
of open problems. They are related
to the order of CND graphs. 
We will prove in Chapter 3 that the 
order of a claw-free,
CND graph is at least 18.
In Section~\ref{properties} we proved
that the order of a 
CND graph is at least 10.

\begin{problem}
What is the minimum order of a 
CND graph?
\end{problem}

An analogous problem can be formulated 
for bipartite graphs as
well, since the smallest bipartite,
CND graph provided by our
construction has order 26.

\begin{problem} What is the minimum order of a 
bipartite, CND graph?
\end{problem}

We gave four examples of cubic, CND graphs
of order 36, and these
are the smallest presently known cubic CND graphs.
In Remark~\ref{remark1} we mentioned that all 
cubic, CND graphs have order greater than 20.
  
\begin{problem}
What is the minimum order of a cubic,
CND graph?
\end{problem}

We have only given examples of graphs
of order $12k$, for $k \geq 2$, which realise
the lower bound on size given in Theorem~\ref{prop-th}.
\begin{problem}
For which other values of $n$ does a  CND graph
of order $n$ and 
size $\left\lceil \frac{5n}{4} \right\rceil$ exist? 
\end{problem}

Lastly, we only have examples of CND graphs
with girth up to $7$. So 
\begin{problem}
Do there exist  CND  graphs
of arbitrarily large girth?
\end{problem}

\chapter{Smallest Claw-free, CND Graphs}

\section{Introduction}
In Chapter 2 we stated that the two graphs
shown in Figure~\ref{smallcnd} were the two smallest,
claw-free, CND graphs.  We will prove that 
statement in this chapter. In fact, we will show
that the graphs in Figure~\ref{smallcnd}
are the smallest, claw-free,
$2$-connected, nontraceable graphs. 
Since in Theorem~\ref{prop-th} we showed that
all CND graphs are $2$-connected (and by
definition a CND graph is nontraceable)
it immediately follows that the two CND graphs
shown in Figure~\ref{smallcnd} are the two
smallest, claw-free, CND graphs.  
We begin by introducing some graph terminology.

If $H$, $G$ are graphs then we say that $G$ is 
$H$-free\index{H-free graph}
if $G$ has no induced subgraph isomorphic to $H$.
In particular, a graph is 
\textit{claw-free}\index{claw-free} if it has no induced subgraph
isomorphic to the claw $K_{1,3}$.
By a smallest graph in some
collection of graphs we mean a graph with the least 
order, and having the least size amongst all
graphs of that order in the collection.
  
The \emph{circumference} $c(G)$ is
the order of a longest cycle in $G$.
A circumference cycle $C$
is a longest cycle and we suppose that 
$C$ has an orientation.
Then $x^-$ denotes the predecessor
of $x \in V(C)$,
and  $x^+$ denotes the successor of $x$.
If $u \in V(C)$, $v \in V(C)$ we denote the path
on $C$ from $u$ to $v$ that accords with the
orientation of $C$ by $C[u, v]$ and the other
path on $C$ from $u$ to $v$ 
by $\overline{C}[u, v]$. The paths obtained
from $C[u, v]$ and $\overline{C}[u, v]$ by deleting
the endvertex $u$  are denoted by 
$C(u, v]$ and $\overline{C}(u, v]$ respectively,
and similarly for $C[u, v)$, $\overline{C}[u, v)$,
$C(u, v)$ and $\overline{C}(u, v)$.
Other terminology will be introduced
when required.

Many papers have been written on hamiltonian cycles
in $k$-connected, claw-free graphs for $k=2$ and $k=3$.  
Matthews and Sumner showed that a 3-connected, claw-free
graph with fewer than 20 vertices is hamiltonian, 
and conjectured that if $G$ is a 
4-connected, claw-free graph then $G$ is
hamiltonian.  (See \cite{ms} and \cite{m}.)
Here we will consider 
analogous questions for traceable graphs.

It is easy to see that the smallest
\index{smallest $k$-connected nontraceable\\ graph} 
$k$-connected, nontraceable graph 
is the complete bipartite graph
$K_{k+2, k}$.  
The existence of
$k$-connected, nontraceable graphs is more
interesting if we add the claw-free condition. 
Of course, the smallest 
connected, nontraceable, claw-free graph is
the net.
%a $K_3$ with a pendant leaf attached 
%to each vertex (called the net). 
Shepherd \cite{shp} 
and Duffus, Jacobson, and Gould \cite{djg}
have
shown that every $k$-connected, 
claw-free and net-free graph
is traceable for $k =1$, and Shepherd \cite{shp}
has shown that such graphs are hamiltonian connected
%(that is, every pair of distinct vertices
%are end-vertices 
%of a hamiltonian path) 
for $k =3$. 
Broersma, Kriesell and Ryj\'{a}\u{c}ek \cite{broer}
proved that, for $k \ge 2$, the following two 
statements are equivalent:
\begin{enumerate}
\item Every $k$-connected claw-free graph
is hamiltonian.
\item Every $k$-connected claw-free graph is traceable.
\end{enumerate}
Hence the conjecture of Matthews
and Sumner is equivalent to the conjecture
that every $4$-connected, claw-free graph is
traceable. The existence of a $3$-connected,
claw-free, nontraceable graph also follows,
since, for example, Matthews and Sumner
exhibit a $3$-connected, nonhamiltonian, claw-free
graph in \cite{ms}. 
In this 
chapter we consider the case
$k =2$ for claw-free graphs. 
%We show that all $2$-connected, claw-free graphs
%with less than 18 vertices are traceable,
%and obtain the two 
%smallest claw-free,
%$2$-connected, nontraceable graphs. 

A nontraceable graph $G$ is called 
maximal nontraceable (MNT) \index{maximal nontraceable graph}
if $G+e$ is traceable for any $e \in E(\overline{G})$.\index{MNT graphs} 
Zelinka \cite{zel}
studied these graphs, and gave constructions 
which define two classes of MNT graphs. 
He initially made the conjecture, which he later
retracted, that these two classes comprised
all MNT graphs. His constructions do not
produce $2$-connected, claw-free graphs.
In Section~\ref{last} we construct $2$-connected,
claw-free, MNT graphs of order $n$ for
each $n \geq 18$. Section~\ref{last} was written
in collaboration with J. Singleton.

\section{Smallest claw-free, 2-connected, nontraceable graphs}
In this section we show that there are two
smallest nontraceable, claw-free, 
2-connected graphs, each having
order 18  and size 24. We will use the following lemmas:
\begin{lemma}\label{bypass}
Let $C$ be a circumference cycle of a claw-free
graph $G$. If two distinct vertices $x$ and
$y$ on $C$
are adjacent, respectively, to vertices $u$  and $v$
($u = v$ allowed) in the same 
component $A$ of $G - V(C)$, then
each of the two paths with endvertices $x$
and $y$ on $C$ has at least
$\tau_A(u,v) + 2$ internal vertices.
\end{lemma}
\begin{proof}
Let $L$ be a path in $A$ which joins $u$
and $v$ and has order $\tau_A(u,v)$.
Since $C$ is a circumference
cycle, $ux^+ \notin E(G)$ and $ux^- \notin E(G)$.
Since $G$ is claw-free, this implies that
$x^-x^+ \in E(G)$. If $x^+ = y^-$, then the
cycle obtained from $C$ by replacing the path
on $C$ from $x^-$ to $y^+$ with the path 
$x^-,x^+,x,L,y,y^+$ has more vertices than $C$.
Hence $x^+ \neq y^-$ and, similarly,
$x^- \neq y^+$.
Let $P$ and $Q$ be the paths on $C$ from $x^+$ to
$y^-$ and from $y^+$ to $x^-$ respectively.
Since the cycle $P,y^+,y,L,x,x^-,x^+$ has order
$\tau_A(u,v) + |V(P)| + 4$ and
$|V(C)| = |V(P)| + |V(Q)| + 2$, it
follows that $|V(Q)| \geq \tau_A(u,v) + 2$. 
Similarly $|V(P)| \geq \tau_A(u,v) + 2$. \hspace*{\fill}\qd
\end{proof}

The next lemma describes the structure of
the components of $G- V(C)$ when $|V(G)| \leq 18$.

\begin{lemma}\label{chords}
Let $G$ be a nontraceable, 
$2$-connected, claw-free graph
such that $|V(G)| \leq 18$. Let $C$ be a 
circumference cycle of $G$. Then 
\begin{enumerate}
\item Each
component of $G - V(C)$ has 
a spanning path of
order at most three whose endvertices
are adjacent to distinct vertices of $C$.
\item  No vertex in $C$ is adjacent to vertices
in different components of $G-V(C)$.
\end{enumerate}
\end{lemma}
\begin{proof}
\begin{enumerate}
\item Let $A$ be a component of $G-V(C)$.
Since $G$ is $2$-connected there exist vertices
$u,v \in V(A)$  ($u=v$ is allowed)
adjacent to distinct vertices $x$ and $y$, respectively,
in $C$.
Let $\displaystyle s = \max \tau_A (u,v)$,
where the maximum is taken over vertices 
$u, v \in V(A)$ such that  $u, v$ are adjacent to
distinct vertices of $C$. 
That $s < 4$ follows from Lemma~\ref{bypass}, since $G$
is nontraceable and $|V(G)| \leq 18$.

Suppose that $s=3$. 
Let $L$ be a path of order $3$ in $A$
with interior vertex $w$ and 
endvertices $u$, $v$ adjacent to
vertices $x$, $y$ respectively on $C$.
Suppose that $V(A-V(L)) \ne \emptyset$.
Let $K$ be a component of $A-V(L)$ and
let $r \in K$ be adjacent to
a vertex of $L$.
\begin{enumerate}
\item[(a)] Suppose that $rw \in E(G)$.
Note that then $ru \notin E(G)$ and
$rv \notin E(G)$ since otherwise $s=3$ would
be contradicted. Since $G$ is claw-free 
we have $uv \in E(G)$. Since $G$ is $2$-connected,
some vertex of $K$ is adjacent to a vertex of $C$.
But then we get a contradiction with $s=3$.
\item[(b)] Suppose that $ru \in E(G)$.
Since $G$ is claw-free, either $rw \in E(G)$,
or $rx \in E(G)$, or $wx \in E(G)$. 
But $rw \notin E(G)$ and $rx \notin E(G)$ since
otherwise we have a contradiction with $s=3$.
So $wx \in E(G)$. Then no vertex of $K$ is adjacent
to $w$ or $v$, since this would contradict $s=3$.
Since $G$ is $2$-connected, some vertex $q \in V(K)$
is adjacent to a vertex of $C$, and this vertex 
must be $y$ or we get a contradiction with $s=3$.
But then there is a claw with centre y,
because $qv \notin E(G)$ 
($qv \in E(G)$ contradicts $s=3$),
and neither $q$ nor $v$ is adjacent to $y^{-}$
or $y^{+}$ since $C$ is a circumference cycle.
\end{enumerate}
Hence no vertex of $K$ is adjacent to $L$,
and it follows that $V(A-V(L)) = \emptyset$.
Similarly,  if $s= 2$ or $s=1$ we get $V(A - V(L)) = \emptyset$.

\item Let $u_1$, $u_2$ be
vertices in different components of $G-V(C)$,
and suppose that $x \in V(C)$ is adjacent to 
$u_1$ and $u_2$. Then $u_1x^+ \notin E(G)$ and
$u_2x^+ \notin E(G)$, otherwise we get a cycle
longer than $C$. Also, $u_1u_2 \notin E(G)$,
since $u_1$ and $u_2$ belong to different components
of $G-V(C)$. Hence $G\left\langle \{x, x^+, u_1,u_2\}\right \rangle$
is an induced claw, which contradicts the fact that $G$
is claw-free. \hspace*{\fill} \qd
\end{enumerate}
\end{proof}

Since our concern here is finding the
\textit{smallest} 2-connected, claw-free, nontraceable
graphs $G$ it follows from Lemma~\ref{chords}
that we may suppose, without loss of generality,
that the components of $G - V(C)$ are single vertices,
which we call \textit{bridge vertices}.\index{bridge vertices}
We may also suppose, without loss of 
generality, that each bridge vertex is adjacent
to exactly two vertices on $C$,
which are called \textit{attachment} vertices 
\index{attachment vertices}
of $C$. For example, the vertices
$a$, $b$, $c$ in Figure~\ref{subcase1} on
page~\pageref{subcase1} are bridge vertices,
and the vertices $w_1$, $v_3$, $w_3$, $w_2$, $v_2$
and $v_1$ are attachment vertices.
If $u$ and $v$ are 
consecutive attachment
vertices of $C$, then the path
$C[u, v]$, and their relatives 
described previously, are called 
\textit{attachment paths}. \index{attachment path}
If we want the
orientation of an attachment path to be determined
by the orientation of a path in which it
lies, and not by $C$, we simply denote
the attachment path by $[u, v]$, or $(u, v]$
and so on.
Attachment paths whose closures (formed by replacing round
brackets with square brackets)
have an attachment vertex \index{adjacent attachment paths}
in common are said to be adjacent, and
vertices of $C$ which are not attachment vertices
are called \textit{extra} vertices.\index{extra vertices}
 
Let $k = |V(G - V(C))|$. 
A \textit{skeleton path} $S$ of $G$ \index{skeleton path}
with respect to $C$ is 
a path formed from an alternating sequence of attachment paths
and bridge vertices,
in such a way that all the attachment vertices and
bridge vertices belong to $V(S)$,
as follows:
\[
(u_1, v_1],b_1,[u_2, v_2],b_2,[u_3, v_3],\cdots, b_k,[u_{k+1}, v_{k+1}). 
\]
Note that, if $k > 1$,  $u_1$ and $v_{k+1}$ are always
interior vertices of $S$, since all the attachment
vertices of $C$ are included in $S$ . 
Therefore, in the above description, $u_1$ and
$v_{k+1}$ will each have two labels.
Since there are $k$ bridge vertices it follows
that  $k-1$ attachment
paths will be omitted from any skeleton path $S$.  If
$G$ is nontraceable, then at least one of these omitted
attachment paths must contain extra vertices of $C$,
otherwise $S$
will be a hamiltonian path in $G$. The next
lemma gives the least number of extra vertices
that must be contained in the omitted attachment paths.
Note that, because of the way in which
$S$ is constructed, omitted attachment paths of $C$ cannot
be adjacent, since $S$ begins and ends with attachment
paths.

\begin{lemma}\label{skeleton}
Let $G$ be a $2$-connected, 
claw-free, nontraceable graph
and $C$ be a circumference cycle of $G$.
Suppose that all the components
of $G - V(C)$ are trivial, and each
component is adjacent to exactly
two vertices of $C$. Let $P$ be a
skeleton path of $G$ w.r.t.\ $C$.
Suppose that $P$ omits an attachment
path $C(u, v)$ with at most two extra vertices. 
Then $G$ has a path with vertex
set $V(P) \cup V(C(u, v))$.
\end{lemma}

\begin{proof}
If $u^-u \in E(P)$, let $P^*$ be the path obtained
from $P$ by replacing $u^-,u$ with $u^-,u^+,u$.
If $u^-u \notin E(P)$, then $u$ is an endvertex
of $P$. In this case, let $P^*$ be the path 
$P,u^+$.
If $V(C(u, v)) = \{ u^+ \}$ we are done. If not,
the desired path can be obtained from $P^*$ by
replacing $v,v^+$ with $v,v^-,v^+$ if
$vv^+ \in E(P^*)$, or adding $v,v^-$ if
$vv^+ \notin E(P^*)$. \hspace*{\fill}\qd
\end{proof}

\begin{corollary}\label{skel}
Suppose that $G$, $C$ and $P$ are as in
Lemma~\ref{skeleton}. Then
\begin{enumerate}
\item If no two attachment 
paths of $C$ which are omitted
by $P$ are adjacent to the same 
attachment path of $C$, then  at 
least one of the 
attachment paths omitted by $P$
contains at least three extra vertices.

\item  If two attachment 
paths of $C$ which are omitted
by $P$ are adjacent to the same 
attachment path of $C$, then either 
at least one of the 
attachment paths omitted by $P$
contains at least three extra vertices,
or at least two attachment paths omitted by $P$
each contains at least two extra vertices.
\end{enumerate}
\end{corollary}
\begin{proof}
\begin{enumerate}
\item Suppose that no two attachment paths
omitted by $S$ are adjacent to the same
attachment path. Then at least one of the 
omitted attachment paths
contains at least three extra vertices, 
since otherwise the
construction in Lemma~\ref{skeleton} can be 
repeated for each
omitted attachment path to give a
hamiltonian path. 
    
\item Suppose that no omitted attachment
path contains more than two extra vertices
and no pair of omitted attachment paths contain
in total more than three extra vertices.
Then we can again repeat the construction
in Lemma~\ref{skeleton} for each
omitted attachment path to get a hamiltonian
path in $G$. \hspace*{\fill}\qd
\end{enumerate} \end{proof}
\begin{figure}[H]
\begin{center}
\epsfig{file=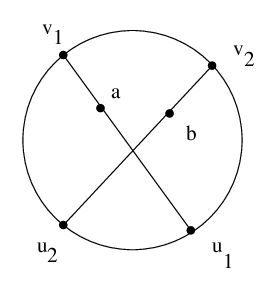}\\
\caption{\label{cross-attach} A pair of crossed attachment vertices.}
\end{center}
\end{figure}
Let $\{u_1, v_1 \}$, $\{u_2, v_2 \}$
be pairs of attachment vertices of $C$
adjacent to 
%components $A$, $B$ of $G-V(C)$ respectively.
bridge vertices $a$, $b$ respectively.
Then we say the pairs $\{u_1, v_1 \}$, $\{u_2, v_2 \}$
are pairs of \textit{crossed} attachment vertices of $C$
\index{crossed attachment vertices}
if (with some orientation of $C$) 
$u_2 \in V(C(u_1, v_1))$ 
and $v_2 \in V(\overline{C}(u_1, v_1))$.

The existence of a crossed pair of attachment vertices
on $C$ gives a lower bound on $c(G)$, as shown
in the next lemma:

\begin{lemma}\label{crossed}
Let $G$ be a claw-free, $2$-connected, nontraceable
graph. Suppose that $G$ has a circumference
cycle $C$ containing a pair of crossed
attachment vertices.
Then $c(G) \geq 14$.
\end{lemma}
\begin{proof}
Suppose that $C$ has a clockwise
orientation and we have a pair of 
crossed attachment vertices,
as shown in Figure~\ref{cross-attach}.
The pairs of attachment paths
$C[v_1, v_2]$, $C[u_1, u_2]$ and
$C[u_2, v_1]$, $C[v_2, u_1]$ must
each contain in total at least two internal
vertices, otherwise either
$a,C[v_1, v_2],b,\overline{C}[u_2, u_1],a$ or
$b,C[u_2, v_1],a,\overline{C}[u_1, v_2],b$ gives a
longer cycle than $C$. 
Since $G$ is claw-free, the vertices
$v_1^-$ and $v_1^+$ are adjacent,
and similarly for the other attachment
vertices. It is then easy to use these
extra edges to appropriately extend either
the cycle $a,C[v_1, v_2],b,\overline{C}[u_2, u_1],a$ or
the cycle $b,C[u_2, v_1],a,\overline{C}[u_1, v_2],b$ 
to show that, in order to
avoid a cycle longer than $C$, we must have
either:
\begin{enumerate}
\item At least three internal vertices in each 
of the four attachment paths $C[v_1, v_2]$,
$C[v_2, u_1]$, $C[u_1, u_2]$,
$C[u_2, v_1]$, which gives
$c(G) \ge 16$, or
\item At least four internal vertices 
in each attachment path of at least one  pair
of adjacent paths. Then, by
Lemma~\ref{bypass}, the remaining two attachment paths
contains in total at least another
two internal vertices, giving $c(G) \ge 14$. \hspace*{\fill}\qd
\end{enumerate} \end{proof}

We can now prove our main result
of this section:
\begin{theorem}\label{greater-18}
Let $G$ be a $2$-connected, claw-free,
nontraceable graph.  Then $|V(G)| \geq 18$.
If $|V(G)| =18$, and $|E(G)|$ has
the least value for all such graphs $G$,
then $G$ is isomorphic to either graph $A$ or $B$
in Figure~\ref{clawfree1}.
\begin{figure}[H]
\begin{center}
\epsfig{file=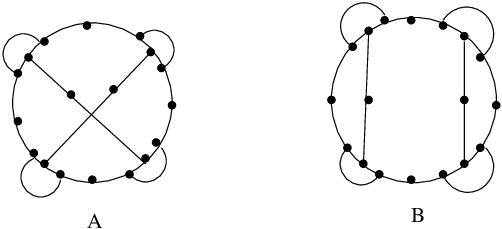}
\end{center}
\caption{\label{clawfree1} Graphs for Theorem~\ref{greater-18}}
\end{figure}
\end{theorem}
\begin{proof}
Suppose that $|V(G)| \leq 18$.  Let $C$ be a 
circumference cycle in 
$G$.
By Lemma~\ref{chords} we may suppose
that the components of $G - V(C)$ 
are all bridge vertices of $C$, and their attachment
vertices on $C$ are all distinct. We may also suppose
that each bridge vertex is adjacent to
exactly two vertices on $C$. Since $G$ is
nontraceable, $C$ has more than one bridge
vertex. We consider four cases, depending on
the number of bridge vertices.

\noindent \textbf{Case 1.} $C$ has exactly two bridge vertices.\\
There are two arrangements of 
attachment vertices, crossed or not
crossed, shown in Figure~\ref{case1}.
\begin{figure}[H]
\begin{center}
\psfrag{a}[c][c]{$A$}
\psfrag{b}[c][c]{$B$}
\epsfig{file=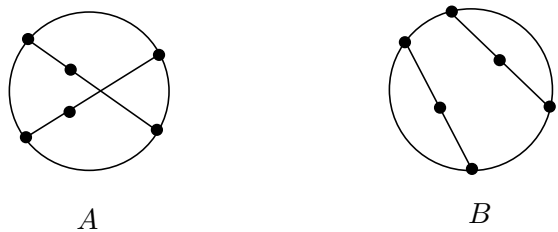}
\end{center}
\caption{\label{case1} Diagrams for Case 1}
\end{figure}
In this case we can always find a skeleton path satisfying
part (1) of Corollary~\ref{skel} which
omits any one of the four attachment paths 
and contains the remaining
three attachment paths of $C$. It follows from 
Corollary~\ref{skel} that each attachment path
of $C$ must contain at least three extra vertices, 
and therefore
$|V(G)| \geq 18$.

\noindent \textbf{Case 2.} $C$ has exactly three bridge vertices.\\
We have five sub-cases to consider, depending
on the arrangement of
attachment vertices on $C$.

\noindent \textbf{Sub-case (i)}\\
See Figure~\ref{subcase1}. 
By Lemma~\ref{bypass} each of the attachment paths
$[v_1, w_1]$, $[v_2, w_2]$ and $[v_3, w_3]$
contains at least
three extra vertices.\\ 
\hfill\begin{minipage}[b]{7.5cm}
\begin{figure}[H]
\begin{center}
\psfrag{a}[c][c]{$w_1$}
\psfrag{b}[c][c]{$v_3$}
\psfrag{c}[c][c]{$b$}
\psfrag{d}[c][c]{$w_3$}
\psfrag{e}[c][c]{$w_2$}
\psfrag{f}[c][c]{$c$}
\psfrag{g}[c][c]{$v_2$}
\psfrag{h}[c][c]{$v_1$}
\psfrag{i}[c][c]{$a$}
\epsfig{file=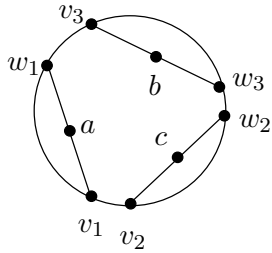}
\end{center}
\caption{\label{subcase1} Diagram for Case 2, Sub-case (i)}%Diagram for Case 2, Sub-case (i)}
\end{figure}
\end{minipage}
\hfill
\begin{minipage}[b]{6cm}
\begin{figure}[H]
\begin{center}
\epsfig{file=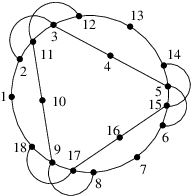}
\end{center}
\caption{\label{trace}Hamiltonian cycle} %Hamiltonian cycle for Case 2, Sub-case (i)}
\end{figure}
\end{minipage}\hfill\\
If we put three vertices
in each of these attachment paths, and add the necessary
edges to make the graph claw-free, we obtain
the graph shown in Figure~\ref{trace}, which has
a hamiltonian cycle $1,2,\ldots,18, 1$ indicated
in the figure.
Hence we must have at least two more vertices
in $G$, so $|V(G)| \geq 20$.\\

\noindent \textbf{Sub-case (ii)}\\
See Figure~\ref{subcase2}. By Lemma~\ref{bypass} the attachment paths
$[w_1, v_1]$ and $[v_3, w_3]$ must each
contain at least three extra vertices.
Consider the following two skeleton paths, both of
which satisfy part (1) of Corollary~\ref{skel}:\\
%\begin{align*}
\hspace*{\fill} $T_{21}:\ (v_1, w_1], a, [v_1, v_2], 
b, [w_2, w_3], c, [v_3, w_3)$ \hspace*{\fill}\\
\hspace*{\fill} $T_{22}:\ (w_1, v_1], a, 
[w_1, w_2], b, [v_2, v_3], c, [w_3, v_3)$ \hspace*{\fill}\\
%\end{align*}

Then $T_{21}$ omits attachment paths $(w_2, w_1)$ and $(v_2, v_3)$; 
$T_{22}$ omits attachment paths $(v_1, v_2)$ and $(w_3, w_2)$.
By Corollary~\ref{skel} at least one of the 
attachment paths in each of
these disjoint pairs of omitted attachment paths contains
at least three extra vertices. Hence $|V(G)| \geq 21$.
\begin{figure}[H]
\begin{center}
\psfrag{a}[c][c]{$w_1$}
\psfrag{i}[c][c]{$v_3$}
\psfrag{e}[c][c]{$b$}
\psfrag{g}[c][c]{$w_3$}
\psfrag{f}[c][c]{$w_2$}
\psfrag{h}[c][c]{$c$}
\psfrag{d}[c][c]{$v_2$}
\psfrag{c}[c][c]{$v_1$}
\psfrag{b}[c][c]{$a$}
\epsfig{file=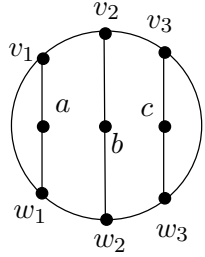}\\
\end{center}
\caption{\label{subcase2} Diagram for Case 2, Sub-case (ii)}
\end{figure}

\noindent \textbf{Sub-case (iii)}:
\begin{figure}[H]
\begin{center}
\psfrag{a}[c][c]{$w_1$}
\psfrag{i}[c][c]{$v_3$}
\psfrag{e}[c][c]{$b$}
\psfrag{g}[c][c]{$w_3$}
\psfrag{f}[c][c]{$w_2$}
\psfrag{h}[c][c]{$c$}
\psfrag{d}[c][c]{$v_2$}
\psfrag{c}[c][c]{$v_1$}
\psfrag{b}[c][c]{$a$}
\epsfig{file=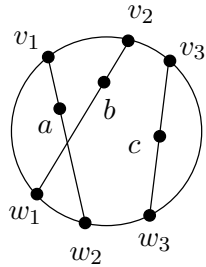}\\
\end{center}
\caption{\label{subcase3} Diagram for Case 2, Sub-case (iii)}
\end{figure}
Consider the following skeleton paths
which all satisfy part (1) of Corollary~\ref{skel}:\\
%\begin{align*}
\hspace*{\fill} $T_{31}:\  (w_1, v_1], a, 
[w_2, w_1], b, [v_2, v_3], c, [w_3, v_3)$ \hspace*{\fill}\\
\hspace*{\fill} $T_{32}:\ (v_2, v_1], a, 
[w_2, w_1], b, [v_2, v_3], c, [w_3, w_2)$ \hspace*{\fill}\\
\hspace*{\fill} $T_{33}:\  (v_1, v_2], b, 
[w_1, v_1], a, [w_2, w_3], c, [v_3, w_3)$ \hspace*{\fill}\\
%\end{align*}
Then $T_{31}$ omits attachment paths $(v_1, v_2)$ and $(w_3, w_2)$; 
$T_{32}$ omits attachment paths $(w_1, v_1)$ and $(v_3, w_3)$;
$T_{33}$ omits attachment paths $(v_2, v_3)$ and $(w_2, w_1)$.  
At least one of each of these pairs of attachment paths
omitted by $T_{31}$, $T_{32}$ and $T_{33}$  contains
at least three extra vertices. Whichever three of the
above attachment paths we choose, we can find a skeleton path
containing those three attachment paths and omitting
another pair, which must contain in total
at least three extra vertices.
Hence $|V(G)| \geq 21$.\\

\noindent \textbf{Sub-case (iv)}\\
Consider the following skeleton paths (see Figure~\ref{subcase4})
which all satisfy part (1) of Corollary~\ref{skel}:\\
%\begin{align*}
\hspace*{\fill} $T_{41}:\  (v_2, v_1], b, [w_3, w_2], 
c, [v_3, v_2], a, [w_1, w_2)$ \hspace*{\fill}\\
\hspace*{\fill} $T_{42}:\ (w_1, v_1], b, [w_3, v_3], 
c, [w_2, w_1], a, [v_2, v_3)$ \hspace*{\fill}\\
\hspace*{\fill} $T_{43}:\ (v_1, v_2], a, [w_1, v_1], 
b, [w_3, w_2], c, [v_3, w_3)$ \hspace*{\fill}\\
%\end{align*}
Then $T_{41}$ omits attachment paths $(w_1, v_1)$ and $(v_3, w_3)$; 
$T_{42}$ omits attachment paths $(v_1, v_2)$ and $(w_3, w_2)$;
$T_{43}$ omits attachment paths $(v_2, v_3)$ and $(w_2, w_1)$.  
\begin{figure}[H]
\begin{center}
\psfrag{a}[c][c]{$w_1$}
\psfrag{i}[c][c]{$v_3$}
\psfrag{e}[c][c]{$b$}
\psfrag{g}[c][c]{$w_3$}
\psfrag{f}[c][c]{$w_2$}
\psfrag{h}[c][c]{$c$}
\psfrag{d}[c][c]{$v_2$}
\psfrag{c}[c][c]{$v_1$}
\psfrag{b}[c][c]{$a$}
\epsfig{file=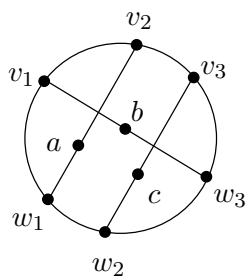}\\
\end{center}
\caption{\label{subcase4} Diagram for  Case 2, Sub-case (iv)}
\end{figure}
At least one of each of these pairs of attachment paths
omitted by $T_{41}$, $T_{42}$ and $T_{43}$ contains
at least three extra vertices. Applying the same argument 
as in sub-case (iii) 
to the above three skeleton paths and their omitted
attachment paths shows that $|V(G)| \geq 21$. 
\pagebreak

\noindent \textbf{Sub-case (v)}\\
\begin{figure}[H]
\begin{center}
\psfrag{a}[c][c]{$w_1$}
\psfrag{i}[c][c]{$v_3$}
\psfrag{e}[c][c]{$b$}
\psfrag{g}[c][c]{$w_3$}
\psfrag{f}[c][c]{$w_2$}
\psfrag{h}[c][c]{$c$}
\psfrag{d}[c][c]{$v_2$}
\psfrag{c}[c][c]{$v_1$}
\psfrag{b}[c][c]{$a$}
\epsfig{file=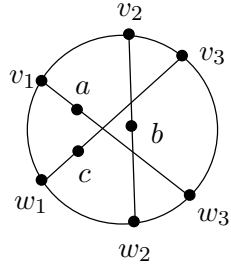}\\
\end{center}
\caption{\label{subcase5} Diagram for Case 2, Sub-case (v)}
\end{figure}
Consider the following skeleton paths
which all satisfy part (1) of 
Corollary~\ref{skel}:\\
%\begin{align*}
\hspace*{\fill} $T_{51}:\ (v_2, v_3], c, [w_1, w_2], 
b, [v_2, v_1], a, [w_3, w_2)$ \hspace*{\fill}\\
\hspace*{\fill} $T_{52}:\ (w_1,  v_1], a, [w_3, v_3], 
c, [w_1, w_2], b, [v_2, v_3)$ \hspace*{\fill}\\
\hspace*{\fill} $T_{53}:\ (v_3, v_2], b, [w_2, w_1], 
c, [v_3, w_3], a, [v_1, w_1)$ \hspace*{\fill}\\
%\end{align*}
Then $T_{51}$ omits attachment paths $(w_1, v_1)$ and $(v_3, w_3]$; 
$T_{52}$ omits attachment paths $(v_2, v_3)$ and $(w_2, w_1)$;
$T_{53}$ omits attachment paths $(v_1, v_2)$ and $(w_3, w_2)$.  
Again, applying the same argument as in sub-case (iii) to these
three skeleton paths and their omitted
attachment paths shows that
$|V(G)| \geq 21$.

\noindent \textbf{Case 3.} $C$ has exactly four bridge vertices.\\
There are 18 possible arrangements of
attachment vertices on $C$. (See Remark~\ref{chord}.)
It is
convenient to split these arrangements into 
three groups, depending on how effective
Lemma~\ref{bypass} is in forcing extra vertices
on $C$:

\noindent \textbf{Group 1.}\\
See Figure~\ref{group1}
If $u$, $v$ are attachment vertices on $C$ which are 
adjacent to the same bridge vertex, then it follows from
Lemma~\ref{bypass} that the attachment paths
$C[u,v]$ and $\overline{C}[u,v]$
each contains at least three internal vertices.
Hence, by inspecting each of the five arrangements 
in Figure~\ref{group1}, it is easy to
see that Lemma~\ref{bypass} forces at
least seven extra vertices on $C$,
giving $|V(G)| \geq 19$.
\begin{figure}[H]
\begin{center}
\epsfig{file=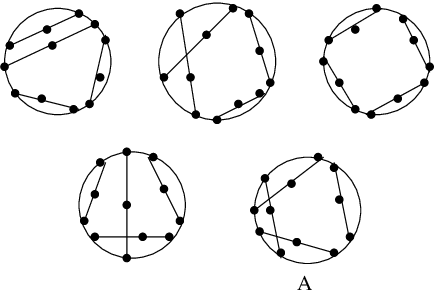}\\
\end{center}
\caption{\label{group1} Diagrams for Case 3, Group 1.}
\end{figure}
For example, consider the arrangement labelled
$A$ in Figure~\ref{group1}, where the vertices are 
labelled as shown in Figure~\ref{labelling1}.
\begin{figure}[H]
\begin{center}
\psfrag{A}[c][c]{$v_1$}
\psfrag{B}[c][c]{$v_2$}
\psfrag{C}[c][c]{$w_1$}
\psfrag{D}[c][c]{$v_3$}
\psfrag{E}[c][c]{$w_3$}
\psfrag{F}[c][c]{$v_4$}
\psfrag{G}[c][c]{$w_2$}
\psfrag{H}[c][c]{$w_4$}
\psfrag{I}[c][c]{$a$}
\psfrag{J}[c][c]{$b$}
\psfrag{K}[c][c]{$c$}
\psfrag{L}[c][c]{$d$}
\epsfig{file=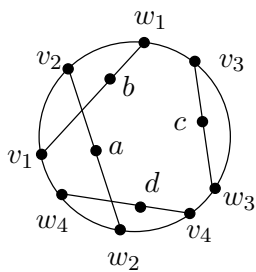}\\
\end{center}
\caption{\label{labelling1} Vertex labels for arrangement $A$ in 
Case 3, Group 1.}
\end{figure}
By Lemma~\ref{bypass}, the $v_1w_1$ path on $C$
which contains the vertex $v_2$ has at least 2
extra vertices, and the $v_4w_4$ path on $C$
containing vertex $w_2$ also has at least 2
extra vertices. The path $v_3w_3$ on $C$ containing
no attachment vertices has, by Lemma~\ref{bypass},
at least three extra vertices. Hence we get, in total,
at least 7 extra vertices for the arrangement $A$.

\newpage

\noindent \textbf{Group 2.}\\
\begin{figure}[H]
\begin{center}
\epsfig{file=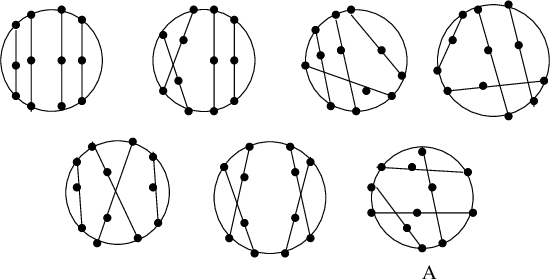}\\
\end{center}
\caption{\label{group2} Diagrams for  Case 3, Group 2.}
\end{figure}

In each of these seven arrangements we can
use a combination of Lemma~\ref{bypass} 
and a skeleton path argument to show that
$|V(G)| \geq 19$. For example, consider the 
arrangement labelled $A$ in Figure~\ref{group2},
and label its vertices as shown in 
Figure~\ref{labelling2}.
\begin{figure}[H]
\begin{center}
\psfrag{A}[c][c]{$w_4$}
\psfrag{B}[c][c]{$v_1$}
\psfrag{C}[c][c]{$v_2$}
\psfrag{D}[c][c]{$w_1$}
\psfrag{E}[c][c]{$v_3$}
\psfrag{F}[c][c]{$w_2$}
\psfrag{G}[c][c]{$v_4$}
\psfrag{H}[c][c]{$w_3$}
\psfrag{I}[c][c]{$c$}
\psfrag{J}[c][c]{$b$}
\psfrag{K}[c][c]{$a$}
\psfrag{L}[c][c]{$d$}
\epsfig{file=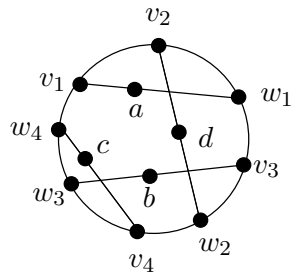}\\
\end{center}
\caption{\label{labelling2} Vertex labels for 
arrangement $A$ in Case 3, Group 2}
\end{figure}
By Lemma~\ref{bypass}, the $v_4w_4$ path on $C$
which contains the vertex $w_3$ has at least 2
extra vertices, and the $v_1w_1$ path on $C$
containing vertex $v_2$ also has at least 2
extra vertices. The skeleton path\\
\hspace*{\fill}
$(w_3,v_4],c,[w_4,w_3],b,[v_3,w_2],d,[v_2,w_1],a,[v_1,v_2)$
\hspace*{\fill}\\
omits attachment paths $(w_4,v_1)$, $(w_1,v_3)$
and $(w_2, v_4)$. By Corollary~\ref{skel}
at least one of these omitted attachment paths
contains at least three extra vertices. Thus
we get, in total, at 
least $7$ extra vertices, and, in this case,
$|V(G)| \ge 19$. Similarly for the other
arrangements in Figure~\ref{group2}.

\noindent \textbf{Group 3.}\\
\begin{figure}[H]
\begin{center}
\epsfig{file=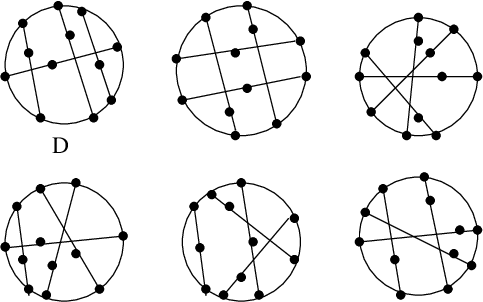}\\
\end{center}
\caption{\label{group3} Diagrams for Case 3, Group 3.}
\end{figure}
\begin{figure}[H]
\begin{center}
\psfrag{A}[c][c]{$v_1$}
\psfrag{B}[c][c]{$v_2$}
\psfrag{C}[c][c]{$v_3$}
\psfrag{D}[c][c]{$v_4$}
\psfrag{E}[c][c]{$w_1$}
\psfrag{F}[c][c]{$w_4$}
\psfrag{G}[c][c]{$w_3$}
\psfrag{H}[c][c]{$w_2$}
\psfrag{I}[c][c]{$d$}
\psfrag{J}[c][c]{$a$}
\psfrag{K}[c][c]{$b$}
\psfrag{L}[c][c]{$c$}
\epsfig{file=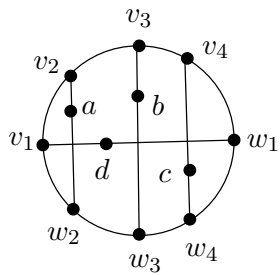}\\
\end{center}
\caption{\label{labelling3} Vertex labels for 
arrangement $D$ in Case 3, Group 3}
\end{figure}
It is easy
to verify that the
six arrangements shown in
Figure~\ref{group3} have the following 
property: Each one has
two disjoint sets $L$ and $M$,
each consisting of three attachment paths,
such that
\begin{enumerate}
\item There is a skeleton path
omitting $L$ and another skeleton path
omitting $M$; 
\item If $A \in L$ and $B \in M$ 
then there exists
a skeleton path containing both $A$ and $B$.
\end{enumerate}

We verify that the arrangement labelled $D$
in Figure~\ref{group3} has the above property. 
Label the vertices of $D$ as shown in Figure~\ref{labelling3}.
We first verify  part 1.\\
Choose $L$ and $M$ to be the sets of attachment
paths $\{(v_1,v_2), (v_3,v_4), (w_1,w_4)\}$
and $\{w_2,v_1), (v_2,v_3), (w_1,v_4)\}$ respectively.
Then the skeleton path\\
\hspace*{\fill}
$(w_2,v_1],d,[w_1,v_4],c,[w_4,w_3],b,[v_3,v_2],a,[w_2,w_3)$
\hspace*{\fill}\\
omits the attachment paths in $L$,
and the skeleton path\\
\hspace*{\fill}
$(w_4,w_1],d,[v_1,v_2],a,[w_2,w_3],b,[v_3,v_4],c,[w_4,w_3)$
\hspace*{\fill}\\
omits the attachment paths in $M$.\\
Now we verify part 2.  Let $P_1$,
$P_2$ and $P_3$ be the following skeleton paths:\\
\hspace*{\fill}
$P_1 : (w_1,w_4],c,[v_4,v_3],b,[w_3,w_2],a,[v_2,v_1],d,[w_1,v_4)$,
\hspace*{\fill}\\
\hspace*{\fill}
$P_2 : (v_1,w_2],a,[v_2,v_3],b,[w_3,w_4],c,[v_4,w_1],d,[v_1,v_2)$,
\hspace*{\fill}\\
\hspace*{\fill}
$P_3 : (v_3,v_4],c,[w_4,w_1],d,[v_1,w_2],a,[v_2,v_3],b,[w_3,w_4)$.
\hspace*{\fill}\\ 
Then the following five pairs of attachment paths from $L$
and $M$ are contained in skeleton path $P_3$:\\
\hspace*{\fill}
$(v_3,v_4) \in L$ and $(w_2,v_1) \in M$,
\hspace*{\fill}\\
\hspace*{\fill}
$(v_3,v_4) \in L$ and $(v_2,v_3) \in M$,
\hspace*{\fill}\\
\hspace*{\fill}
$(v_3,v_4) \in L$ and $(w_1,w_4) \in M$,
\hspace*{\fill}\\
\hspace*{\fill}
$(w_1,w_4) \in L$ and $(w_2,v_1) \in M$,
\hspace*{\fill}\\
\hspace*{\fill}
$(w_1,w_4) \in L$ and $(v_2,v_3) \in M$.
\hspace*{\fill}\\
The following two pairs of attachment 
paths are contained in skeleton path
$P_1$ : $(v_1,v_2) \in L$ and $(w_1,v_4) \in M$;
$(w_1,w_4) \in L$ and $(w_1,v_4) \in M$.
Finally, the following two pairs of
attachment paths are contained in skeleton
path $P_2$ ; $(v_1,v_2) \in L$ and $(w_2,v_1) \in M$;
$(v_1,v_2) \in L$ and $(v_2, v_3) \in M$.\\
%Similarly for the other arrangements in Figure~\ref{group3}.
The above property can be used to show that 
$|V(G)| \ge 19$ as follows:\\
By Corollary~\ref{skel},  $L$ has either an 
attachment path with at
least three extra vertices, or the attachment paths 
in $L$ contain
in total at least four extra vertices,
and the same holds for $M$. If each of $L$ and $M$ 
contains an attachment path with at least three extra
vertices, we can find another skeleton path
containing these two attachment paths, 
the omitted attachment paths of
which must contain in total 
at least three extra vertices.
Hence $|V(G)| \geq 21$. In the other cases,
$L$ and $M$  together
contain at least seven extra vertices, 
giving $|V(G)| \geq 19$.    

\noindent \textbf{Case 4.} $C$ has five or more  
bridge vertices.\\
If $C$ has no crossed pairs of attachment
vertices, then two of the bridge vertices
determine two different attachment paths,
each containing no other attachment vertices.
By Lemma~\ref{bypass} these two attachment paths
each contains at least three extra vertices.
Hence $|V(G)| \geq 15+6 = 21$.
If $C$ has a crossed pair of attachment vertices, 
then by 
Lemma~\ref{crossed} we get $c(G) \geq 14$, and hence
$|V(G)| \geq 14+5 = 19$. 

Thus in all cases $|V(G)| \geq 18$.
In cases (2), (3) and (4)
we have $|V(G)| \ge 19$.
Hence if $|V(G)| =18$, then case (1) applies.
Then each of the four attachment paths has
3 extra vertices, giving 16 edges. 
There are 4 more edges incident with the bridge
vertices. Since $G$ is claw-free,
the pairs of vertices on $C$,
which are adjacent to the attachment vertices,
are adjacent, giving another 4 edges. Thus $|E(G)| \ge 24$.
If $|E(G)| =24$, then $G$ is isomorphic to 
graph $A$ in Figure~\ref{clawfree1} if the attachment
vertices are crossed, and to graph $B$ otherwise. \hspace*{\fill}\qd
\end{proof}
\begin{remark}\label{chord}
All the diagrams used in the proof 
of Theorem~\ref{greater-18} can be
regarded as chord diagrams\index{chord diagrams}.
A chord diagram is a set of $2n$ points on an
oriented circle joined pairwise by $n$ chords,
and they have applications in the theory of knots.
Two chord diagrams are ismorphic if one
can be obtained by some rotation of the
other. Sawada~\cite{sawada} has developed
(and implemented) algorithms for generating
nonisomorphic chord diagrams. In \cite{sawada}
Sawada lists the $5$ nonisomorphic chord 
diagrams having $3$
chords, and the $18$ nonisomorphic chord diagrams having
$4$ chords. These chord diagrams agree with
the diagrams we have in Cases 2 and
3  in Theorem~\ref{greater-18}. 
(There are $105$ chord diagrams with $5$
chords. Fortunately we did not have to
consider those diagrams!)
\end{remark}

Graphs $A$ and $B$ in Figure~\ref{clawfree1} 
are clearly isomorphic, respectively,  to the CND graphs 
labelled $A$ and $B$
in Figure~\ref{smallcnd} on page~\pageref{smallcnd}.
Therefore they are $2$-connected, claw-free,
nontraceable graphs, and it follows from
Theorem~\ref{greater-18} that 
they are the smallest such graphs.

\section{Maximal nontraceable  graphs}\label{last}
Since each of the graphs $A$ and $B$ in Figure~\ref{clawfree1}
is nontraceable,  each
is a spanning subgraph of some MNT graph. 
In order to produce MNT graphs \index{MNT graphs}
from $A$ and $B$ we use the following
lemma (which follows from Corollary 7 in \cite{bdm},
but here we give a direct proof):
\begin{lemma}
\label{mnt}
The neighbours of a vertex 
of degree two in a MNT graph are adjacent.
\end{lemma}
\begin{proof}
Let $G$ be a MNT graph, and $v \in V(G)$ have degree two.
Let  $x$, $y$ be the neighbours of $v$.
If $xy \notin E(G)$ then, since $G$ is MNT, 
the graph $G + xy$ has
a hamiltonian path $P$. Since $xy \in E(P)$,
and $\deg(v) =2$,  the vertex $v$ is an endvertex
of the path $P$. We can suppose that the path $P$ 
starts with the subpath $v,x,y$ (or $v,y,x$).
Therefore, if we replace the subpath
$v,x,y$ in $P$ with $x,v,y$ (or $v,y,x$ in $P$ with
$y,v,x$) we get a hamiltonian path in $G$, 
contradicting the fact that $G$ is
nontraceable. \hspace*{\fill} \qd\\
\end{proof}

In both graphs $A$ and $B$ in Figure~\ref{clawfree1} 
we add edges 
between the neighbours of each vertex of
degree two. Let the resulting graphs
be $A^*$ and $B^*$, shown in Figure~\ref{span}.
It can be checked that both $A^*$ and $B^*$ are MNT graphs.
It can also be checked that if $v$ is a vertex
of degree two in graph $A$, and $w$ is any other vertex 
in $A$, then $A+vw$ is traceable. 
Hence, by
Lemma~\ref{mnt}, every MNT graph containing $A$
as a spanning subgraph
will also contain $A^*$ as a spanning subgraph.
In other words, $A^*$ is the only MNT graph containing
$A$ as a spanning subgraph.
The graph $A^*$ is also claw-free.
Similarly,  the only MNT graph
containing $B$ as a spanning subgraph
is the graph $B^*$ in Figure~\ref{span},
but $B^*$ is not claw-free.
\begin{figure}[H]
\begin{center}
\epsfig{file=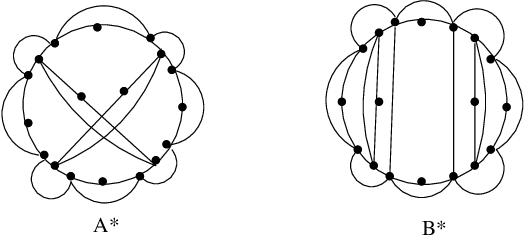}\\
\end{center}
\caption{\label{span} Graphs $A^*$ and $B^*$}
\end{figure}
Thus $A^*$ is the smallest 2-connected, 
claw-free, MNT graph. It can be drawn as
shown in Figure~\ref{clawfreemnt}.
\begin{figure}[H]
\begin{center}
\epsfig{file=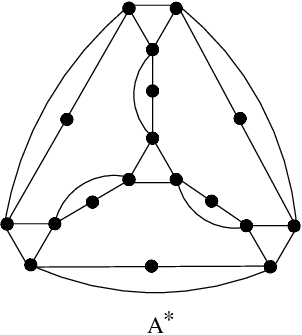}\\
\end{center}
\caption{\label{clawfreemnt} Graph $A^*$}
\end{figure}
It is shown in \cite{ms} that all
$2$-connected claw-free graphs
are $1$-tough. \index{$1$-tough graph}
(A graph $G$ is $t$-tough if
$|S| \geq t k(G-S)$ for every
cutset $S \subset V(G)$,
where $k(G-S)$ denotes the number of
components in $G-S$.)

The MNT graphs produced by Zelinka \cite{zel} are 
not 1-tough.  Hence, according to \cite{ms} 
these graphs are
not both 2-connected and claw-free.

We now give a construction of an infinite 
family of 2-connected, claw-free,
MNT graphs of order greater than 18: 
We construct a graph of order $18+m$
for every $m\geq 1$ by joining the 
vertices of a new $K_{m}$ to every vertex
of one of the $K_{3}$'s in $A^*$ which does not have a 
vertex of degree two. (We can also
produce such graphs by 
joining appropriate $K_{k}$'s to 
some or all four of the $K_{3}$'s
that do not contain vertices of degree two.)

At the Fourth Cracow Conference on 
Graph Theory (Czorsztyn 2002) Dudek \cite{du}
presented another construction 
which produced an infinite family of MNT
graphs which cannot be produced by 
using Zelinka's constructions. Also,
MNT graphs which are 2-tough 
are constructed in \cite{bbv}.

\backmatter

\newpage
\addcontentsline{toc}{chapter}{\protect\numberline{\ }{Index}}
\printindex
\end{document}